\documentclass{amsart}
\usepackage{amsmath}
\usepackage{amssymb}
\usepackage{amsthm}
\usepackage{epsfig}
\usepackage{amsfonts}
\usepackage{amscd}
\usepackage{mathrsfs}
\usepackage{enumerate}
\usepackage{latexsym}
\usepackage{url}
\usepackage{chngcntr}
\counterwithin{section}{part}

\newtheorem{theorem}{Theorem}[section]
\newtheorem{lemma}[theorem]{Lemma}
\newtheorem{proposition}[theorem]{Proposition}
\newtheorem{corollary}[theorem]{Corollary}

\theoremstyle{definition}
\newtheorem{definition}[theorem]{Definition}
\newtheorem{example}[theorem]{Example}

\newtheorem{remark}[theorem]{Remark}
\newtheorem{notation}[theorem]{Notation}

\theoremstyle{question}

\numberwithin{equation}{section}

\newtheorem*{xclaim}{Claim}

\begin{document}

\title[Ideals in $C_B(X)$ arising from ideals in $X$]{Ideals in $C_B(X)$ arising from ideals in $X$}

\author{M. R. Koushesh}

\address{Department of Mathematical Sciences, Isfahan University of Technology, Isfahan 84156--83111, Iran.}

\address{School of Mathematics, Institute for Research in Fundamental Sciences (IPM), P.O. Box: 19395--5746, Tehran, Iran.}

\email{koushesh@cc.iut.ac.ir}



\thanks{This research was in part supported by a grant from IPM}

\begin{abstract}
Let $X$ be a completely regular topological space. We assign to each (set theoretic) ideal of $X$ an (algebraic) ideal of $C_B(X)$, the normed algebra of continuous bounded complex valued mappings on $X$ equipped with the supremum norm. We then prove several representation theorems for the assigned ideals of $C_B(X)$. This is done by associating a certain subspace of the Stone--\v{C}ech compactification $\beta X$ of $X$ to each ideal of $X$. This subspace of $\beta X$ has a simple representation, and in the case when the assigned ideal of $C_B(X)$ is closed, coincides with its spectrum as a $C^*$-subalgebra of $C_B(X)$. This in particular provides information about the spectrum of those closed ideals of $C_B(X)$ which have such representations. This includes the non-vanishing closed ideals of $C_B(X)$ whose spectrums are studied in great detail. Our representation theorems help to understand the structure of certain ideals of $C_B(X)$. This has been illustrated by means of various examples. Our approach throughout will be quite topological and makes use of the theory of the Stone--\v{C}ech compactification.
\end{abstract}

\maketitle

\tableofcontents

\part{Introduction}\label{JHGG}

\bigskip

Throughout this article by a \textit{space} we mean a topological space. Completely regular spaces as well as compact spaces (and therefore locally compact spaces) are assumed to be Hausdorff. Also, the field of scalars is assumed to be $\mathbb{C}$. (Note, however, that our proofs are all valid in the real setting, though, some phrases are defined only in the complex setting.)

Let $X$ be a space. We denote by $C(X)$ the set of all continuous mappings $f:X\rightarrow\mathbb{C}$ and we denote by $C_B(X)$ the set of all bounded elements of $C(X)$. Let $f\in C(X)$. The \textit{zero-set} of $f$ is $f^{-1}(0)$ and is denoted by $\mathrm{z}(f)$, the \textit{cozero-set} of $f$ is $X\setminus f^{-1}(0)$ and is denoted by $\mathrm{coz}(f)$, and the \textit{support} of $f$ is $\mathrm{cl}_X\mathrm{coz}(f)$ and is denoted by $\mathrm{supp}(f)$. The set of all cozero-sets of $X$ is denoted by $\mathrm{Coz}(X)$. We denote by $C_{00}(X)$ the set of all $f\in C_B(X)$ with a compact support and we denote by $C_0(X)$ the set of all $f\in C_B(X)$ which vanish at infinity (that is, $|f|^{-1}([1/n,\infty))$ is compact for each positive integer $n$).

Let $X$ be a set. An \textit{ideal} ${\mathfrak I}$ in $X$ is a non-empty collection of subsets of $X$ such that
\begin{itemize}
  \item Any subset of an element of ${\mathfrak I}$ is in ${\mathfrak I}$.
  \item Any finite union of elements of ${\mathfrak I}$ is in ${\mathfrak I}$.
\end{itemize}
Let ${\mathfrak I}$ be an ideal in $X$. Then ${\mathfrak I}$ is a \textit{$\sigma$-ideal} in $X$ if it has any countable union of its elements. The elements of ${\mathfrak I}$ are called \textit{null} (or \textit{negligible}) sets. The ideal ${\mathfrak I}$ is \textit{proper} if $X$ is not null. Suppose that $X$ is a space. Then $X$ is \textit{locally null} (with respect to ${\mathfrak I}$) if each point of $X$ has a null neighborhood in $X$. One may view an ideal as a collection of sets whose elements are considered to be somehow ``small" or ``negligible"; every subset of an element of an ideal must also be in the ideal; this explains what is meant by ``smallness". An ideal ${\mathfrak I}$ in a set $X$ is called a \textit{bornology} if its elements cover the whole $X$, that is, $X=\bigcup{\mathfrak I}$. A bornology on a set is generally viewed as the ``minimum structure" required to address questions of boundedness.

In \cite{Ko6} (also \cite {Ko11} and \cite{Ko10}) for a locally separable metrizable space $X$ we have introduced and studied the closed ideal $C_S(X)$ of $C_B(X)$ consisting of all $f\in C_B(X)$ with a separable support. This has been done through the critical introduction of a certain subspace $\lambda X$ of the Stone--\v{C}ech compactification $\beta X$ of $X$ and by relating the algebraic structure of $C_S(X)$ to the topological properties of $\lambda X$, using techniques we had already developed in \cite{Ko3} (and \cite{Ko12}). This motivated our present study of (algebraic) ideals of $C_B(X)$ by means of (set theoretic) ideals of $X$, where $X$ is a space which is not required mostly to be beyond completely regular. We first show how ideals in $X$ may be assigned to ideals in $C_B(X)$. We then prove several representation theorems for the assigned ideals of $C_B(X)$. This is done by associating a certain subspace of $\beta X$ to each ideal of $X$. This subspace of $\beta X$ has a simple representation; indeed, in certain cases it coincides with familiar subspaces of $\beta X$, and in the case when the assigned ideal of $C_B(X)$ is closed, equals to its spectrum. This in particular provides information about the spectrum of those closed ideals of $C_B(X)$ which have such representations. The class of such ideals of $C_B(X)$ is reasonably large, and includes all non-vanishing closed ideals of $C_B(X)$, whose spectrums are studied in great detail. Our representation theorems help to understand the structure of certain ideals of $C_B(X)$. This has been illustrated by means of various examples whose study will comprise a large portion of this article. Our approach here will be quite topological and will make use of the theory of the Stone--\v{C}ech compactification.

This article is divided into two parts.

In the first part (Part \ref{HFPG}) for a space $X$ and an ideal $\mathfrak{I}$ of $X$ we consider the ideals $C^{\mathfrak I}_{00}(X)$ and $C^{\mathfrak I}_0(X)$ of $C_B(X)$ defined by
\[C^{\mathfrak I}_{00}(X)=\big\{f\in C_B(X):\mathrm{supp}(f)\mbox{ has a null neighborhood in }X\big\}\]
and
\[C^{\mathfrak I}_0(X)=\big\{f\in C_B(X):|f|^{-1}\big([1/n,\infty)\big)\mbox{ is null for each }n\big\}.\]
The above expressions of $C^{\mathfrak I}_{00}(X)$ and $C^{\mathfrak I}_0(X)$ are motivated by the corresponding definitions of $C_{00}(X)$ and $C_0(X)$, respectively. Indeed, $C^{\mathfrak I}_{00}(X)$ and $C^{\mathfrak I}_0(X)$ coincide with $C_{00}(X)$ and $C_0(X)$, respectively, if $X$ is a locally compact space and $\mathfrak{I}$ is the ideal of $X$ consisting of all subspaces with a compact closure in $X$. The main result of this part states that for a normal locally null space $X$ the ideals $C^{\mathfrak I}_{00}(X)$ and $C^{\mathfrak I}_0(X)$ of $C_B(X)$ are respectively isometrically isomorphic to $C_{00}(Y)$ and $C_0(Y)$ for a unique (up to homeomorphism) locally compact space $Y$, namely, for the subspace
\[Y=\bigcup\big\{\mathrm{int}_{\beta X}\mathrm{cl}_{\beta X}C:C\in\mathrm{Coz}(X)\mbox{ and }\mathrm{cl}_XC\mbox{ has a null neighborhood in $X$}\big\}\]
of $\beta X$. Furthermore, $C^{\mathfrak I}_0(X)$ is a closed ideal of $C_B(X)$ and $Y$ is the spectrum of $C^{\mathfrak I}_0(X)$. In addition
\begin{itemize}
  \item $X$ is dense in $Y$,
  \item $C^{\mathfrak I}_{00}(X)$ is dense in $C^{\mathfrak I}_0(X)$,
  \item $Y$ is compact if and only if $C^{\mathfrak I}_{00}(X)$ is unital if and only if $C^{\mathfrak I}_0(X)$ is unital if and only if ${\mathfrak I}$ is non-proper.
\end{itemize}

In the second part (Part \ref{KHGJ}) we consider specific examples. This specification, either of the space $X$ or the ideal ${\mathfrak I}$, enables us to study certain ideals of $C_B(X)$, and their spectrums, whenever relevant. This part is divided into three sections.

In the first section (Section \ref{HFLH}) we consider certain well known ideals of $\mathbb{N}$. These include the \textit{summable} ideal ${\mathfrak S}$ and the \textit{density} ideal ${\mathfrak D}$ of $\mathbb{N}$ defined by
\[{\mathfrak S}=\bigg\{A\subseteq\mathbb{N}:\sum_{n\in A}\frac{1}{n}\mbox{ converges}\bigg\}\]
and
\[{\mathfrak D}=\bigg\{A\subseteq\mathbb{N}:\limsup_{n\rightarrow\infty}\frac{|A\cap\{1,\ldots,n\}|}{n}=0\bigg\}.\]
For simplicity of the notation let us denote
\[\mathfrak{s}_{00}=C^{\mathfrak S}_{00}(\mathbb{N}),\quad\mathfrak{d}_{00}=C^{\mathfrak D}_{00}(\mathbb{N}),\quad\mathfrak{s}_0=C^{\mathfrak S}_0(\mathbb{N})\quad\mbox{and}\quad\mathfrak{d}_0=C^{\mathfrak D}_0(\mathbb{N}).\]
Then, it follows from our general results in Part \ref{HFPG} that
\[\mathfrak{s}_{00}=\bigg\{\mathbf{x}\in\ell_\infty:\sum_{\mathbf{x}(n)\neq0}\frac{1}{n}\mbox{ converges}\bigg\},\]
and
\[\mathfrak{d}_{00}=\bigg\{\mathbf{x}\in\ell_\infty:\limsup_{n\rightarrow\infty}\frac{|\{k\leq n:\mathbf{x}(k)\neq0\}|}{n}=0\bigg\},\]
are ideals of $\ell_\infty$ and
\[\mathfrak{s}_0=\bigg\{\mathbf{x}\in\ell_\infty:\sum_{|\mathbf{x}(n)|\geq\epsilon}\frac{1}{n}\mbox{ converges for each }\epsilon>0\bigg\},\]
and
\[\mathfrak{d}_0=\bigg\{\mathbf{x}\in\ell_\infty:\limsup_{n\rightarrow\infty}\frac{|\{k\leq n:|\mathbf{x}(k)|\geq\epsilon\}|}{n}=0\mbox{ for each }\epsilon>0\bigg\}\]
are closed ideals $\ell_\infty$. Also, $\mathfrak{s}_{00}$ and $\mathfrak{s}_0$ are respectively isometrically isomorphic to $C_{00}(S)$ and $C_0(S)$ with
\[S=\bigcup\bigg\{\mathrm{cl}_{\beta\mathbb{N}}A:A\subseteq\mathbb{N}\mbox{ and }\sum_{n\in A}\frac{1}{n}\mbox{ converges}\bigg\},\]
and $\mathfrak{d}_{00}$ and $\mathfrak{d}_0$ are respectively isometrically isomorphic to $C_{00}(D)$ and $C_0(D)$ with
\[D=\bigcup\bigg\{\mathrm{cl}_{\beta\mathbb{N}}A:A\subseteq\mathbb{N}\mbox{ and }\limsup_{n\rightarrow\infty}\frac{|A\cap\{1,\ldots,n\}|}{n}=0\bigg\}.\]
In particular, $\mathfrak{s}_{00}$ and $\mathfrak{d}_{00}$ are dense in $\mathfrak{s}_0$ and $\mathfrak{d}_0$, respectively. Furthermore, $S$ and $D$ are the spectrums of $\mathfrak{s}_0$ and $\mathfrak{d}_0$, respectively. As corollaries of the above representations it follows that the pair $\mathfrak{s}_{00}$ and $\mathfrak{d}_{00}$, the pair $\mathfrak{s}_{00}/c_{00}$ and $\mathfrak{d}_{00}/c_{00}$, and the pair $\mathfrak{s}_0/c_0$ and $\mathfrak{d}_0/c_0$ each contains an isometric copy of
\[\bigoplus_{n=1}^\infty\ell_\infty,\quad\bigoplus_{i<2^\omega}\frac{\ell_\infty}{c_{00}}\quad\mbox{and}\quad\bigoplus_{i<2^\omega}\frac{\ell_\infty}{c_0},\]
respectively.

In the second section (Section \ref{KHDS}) we consider non-vanishing closed ideals of $C_B(X)$ for a completely regular space $X$. (An ideal $H$ of $C_B(X)$ is called \textit{non-vanishing} if for each $x\in H$ there is some $h\in H$ such that $h(x)\neq 0$.) We first show how a non-vanishing closed ideal $H$ of $C_B(X)$ can be thought of as being $C^{\mathfrak I}_0(X)$ for an appropriate choice of an ideal ${\mathfrak I}$ of $X$. Our representation theorems will then enable us to determine the spectrum $\mathfrak{sp}(H)$ of $H$ as the subspace
\[\mathfrak{sp}(H)=\bigcup_{h\in H}\mathrm{cl}_{\beta X}|h|^{-1}\big((1,\infty)\big)\]
of $\beta X$. As corollaries we prove, among others, that the spectrum of a non-vanishing closed ideal $H$ of $C_B(X)$ is $\sigma$-compact if and only if $H$ is $\sigma$-generated and is connected if and only if $H$ is indecomposable (that is, $H$ is not the direct sum of any two non-zero ideals of $C_B(X)$). Also, components of the  spectrum of $H$ are all open in it (in particular, the spectrum of $H$ is locally connected) if and only if
\[H=\overline{\bigoplus_{i\in I}H_i},\]
where $H_i$ is an indecomposable closed ideal in $C_B(X)$ for any $i\in I$. (Here the bar denotes the closure in $C_B(X)$.) Further, for a collection $\{H_i:i\in I\}$ of non-vanishing ideals in $C_B(X)$ we study the relation between the spectrum of the ideals generated by $H_i$'s and the individual spectrums $\mathfrak{sp}(\overline{H_i})$'s. Specifically, we show that
\[\mathfrak{sp}\bigg(\overline{\bigvee_{i\in I}H_i}\bigg)=\bigcup_{i\in I}\mathfrak{sp}\big(\overline{H_i}\big)\]
provided that $H_i$ is non-vanishing for each $i\in I$,
\[\mathfrak{sp}\bigg(\overline{\bigoplus_{i\in I}H_i}\bigg)=\bigoplus_{i\in I}\mathfrak{sp}\big(\overline{H_i}\big)\]
provided that ${\bigoplus_{i\in I}H_i}$ is non-vanishing, and
\[\mathfrak{sp}\bigg(\overline{\bigcap_{i\in I}H_i}\bigg)=\mathrm{int}_{\,\mathfrak{sp}(C_B(X))}\bigg(\bigcap_{i\in I}\mathfrak{sp}\big(\overline{H_i}\big)\bigg)\]
provided that $\bigcap_{i\in I}H_i$ is non-vanishing. Here the bar denotes the closure in $C_B(X)$, $\bigvee_{i\in I}H_i$ denotes the ideal of $C_B(X)$ generated by $\bigcup_{i\in I}H_i$, the first $\bigoplus$ denotes the algebraic direct sum, and the second $\bigoplus$ denotes the topological direct sum. We also show how the collection of all non-vanishing closed ideals of $C_B(X)$ can be made into one-to-one correspondence with the collection of all closed open bornologies of $X$ (that is, bornologies whose each element is contained in a closed element as well as in an open element).

Finally, in the third section (Section \ref{KIJGD}) we consider ideals of $C_B(X)$ arising from ideals of $X$ related to its topology. More precisely, for a space $X$ and a topological property $\mathfrak{P}$ we consider the collection
\[{\mathfrak I}_\mathfrak{P}=\{A\subseteq X:\mathrm{cl}_XA\mbox{ has }\mathfrak{P}\},\]
which is an ideal in $X$ provided that $\mathfrak{P}$ is closed hereditary (that is, any closed subspace of a space with $\mathfrak{P}$ has $\mathfrak{P}$) and preserved under finite closed sums (that is, any space which is a finite union of its closed subspaces each having $\mathfrak{P}$ has $\mathfrak{P}$). We consider $C^{\mathfrak I}_{00}(X)$ and $C^{\mathfrak I}_0(X)$ where ${\mathfrak I}={\mathfrak I}_\mathfrak{P}$. For simplicity of the notation denote
\[C^\mathfrak{P}_{00}(X)=C^{{\mathfrak I}_\mathfrak{P}}_{00}(X)\quad\mbox{and}\quad C^\mathfrak{P}_0(X)=C^{{\mathfrak I}_\mathfrak{P}}_0(X)\]
and
\[\lambda_\mathfrak{P}X=\lambda_{{\mathfrak I}_\mathfrak{P}}X\]
whenever $X$ is completely regular. In this context we have
\[C^\mathfrak{P}_{00}(X)=\big\{f\in C_B(X):\mathrm{supp}(f)\mbox{ has a closed neighborhood in $X$ with }\mathfrak{P}\big\}\]
and
\[C^\mathfrak{P}_0(X)=\big\{f\in C_B(X):|f|^{-1}\big([1/n,\infty)\big)\mbox{ has }\mathfrak{P}\mbox{ for each }n\big\}.\]
Then $C^\mathfrak{P}_{00}(X)$ is an ideal of $C_B(X)$ and $C^\mathfrak{P}_0(X)$ is a closed ideal of $C_B(X)$ which contains $C^\mathfrak{P}_{00}(X)$.
The above representations of $C^\mathfrak{P}_{00}(X)$ and $C^\mathfrak{P}_0(X)$ further simplify under certain conditions on $X$ or $\mathfrak{P}$. In the case when $X$ is normal and locally-$\mathfrak{P}$ (that is, each point of $X$ has a neighborhood in $X$ with $\mathfrak{P}$) then $C^\mathfrak{P}_{00}(X)$ and $C^\mathfrak{P}_0(X)$ are respectively isometrically isomorphic to $C_{00}(Y)$ and $C_0(Y)$ for the (unique) locally compact space $Y=\lambda_\mathfrak{P}X$. Furthermore, $Y$ is the spectrum of $C^\mathfrak{P}_0(X)$. In addition
\begin{itemize}
  \item $X$ is dense in $Y$,
  \item $C^\mathfrak{P}_{00}(X)$ is dense in $C^\mathfrak{P}_0(X)$,
  \item $Y$ is compact if and only if $C^\mathfrak{P}_{00}(X)$ is unital if and only if $C^\mathfrak{P}_0(X)$ is unital if and only if $X$ has $\mathfrak{P}$.
\end{itemize}
Under certain conditions on $X$ and $\mathfrak{P}$, we further have
\begin{itemize}
  \item $C^\mathfrak{P}_{00}(X)=\{f\in C_B(X):\mathrm{supp}(f)\mbox{ has }\mathfrak{P}\}=C^\mathfrak{P}_0(X)$,
  \item $C_{00}(Y)=C_0(Y)$; in particular, $Y$ is countably compact.
\end{itemize}
As a particular example, for a locally separable metrizable space $X$ it follows that
\[C_S(X)=\big\{f\in C_B(X):\mathrm{supp}(f)\mbox{ is separable}\big\},\]
is a closed ideal of $C_B(X)$ which is isometrically isomorphic to $C_0(Y)$ for the locally compact space
\[Y=\bigcup\{\mathrm{cl}_{\beta X}S:S\subseteq X\mbox{ is separable}\}.\]
In particular, $Y$ is the spectrum of $C_S(X)$. Furthermore,
\begin{itemize}
\item $\dim C_S(X)=\mathrm{d}(X)^{\aleph_0}$ if $X$ is non-separable,
\item $C_0(Y)=C_{00}(Y)$; in particular, $Y$ is countably compact,
\item $Y$ is normal if and only if $Y$ is compact if and only if $C_S(X)$ is unital if and only if $X$ is separable.
\end{itemize}
Here $\dim$ denotes the vector space dimension and $\mathrm{d}(X)$ (called the \textit{density} of $X$) is defined by
\[\mathrm{d}(X)=\min\big\{|D|:D\mbox{ is dense in }X\big\}+\aleph_0.\]
Moreover, if $X$ is also non-separable, then there is a chain of length $\lambda$, where $\aleph_\lambda=\mathrm{d}(X)$, consisting of closed ideals $H_\mu$'s of $C_B(X)$ such that
\[C_0(X)\subsetneqq H_0\subsetneqq H_1\subsetneqq\cdots\subsetneqq H_\lambda=C_B(X),\]
with $H_\mu$, for each $\mu<\lambda$, being isometrically isomorphic to
\[C_0(Y_\mu)=C_{00}(Y_\mu),\]
where $Y_\mu$ is a locally compact countably compact space which contains $X$ densely.

The concluding results in this section deal with specific topological properties $\mathfrak{P}$, such as realcompactness and pseudocompactness. (Recall that a completely regular space $X$ is \textit{realcompact} if it is homeomorphic to a closed subspaces of some product $\mathbb{R}^\mathfrak{m}$ and is \textit{pseudocompact} if there is no continuous unbounded mapping $f:X\rightarrow\mathbb{R}$.) We show that if $\mathfrak{P}$ is realcompactness and $X$ is normal, then
\[\lambda_\mathfrak{P}X=\beta X\setminus\mathrm{cl}_{\beta X}(\upsilon X\setminus X)\]
where $\upsilon X$ is the Hewitt realcompactification of $X$. (Recall that the \textit{Hewitt realcompactification} of a completely regular space $X$ is a realcompact space $\upsilon X$ which contains $X$ as a dense subspace and is such that every continuous mapping $f:X\rightarrow\mathbb{R}$ is continuously extendible over $\upsilon X$; one may assume that $\upsilon X\subseteq\beta X$.) Also, if $X$ is completely regular, then
\[\lambda_{\mathfrak U}X=\mathrm{int}_{\beta X}\upsilon X\]
for the ideal
\[{\mathfrak U}=\langle U:\mbox{$U$ is an open subspace of $X$ with a pseudocompact closure}\rangle\]
of $X$.

Throughout this article the Stone--\v{C}ech compactification will be the main tool of study. We review certain properties of the Stone--\v{C}ech compactification in the following and refer the reader to the texts \cite{E}, \cite{GJ}, \cite{PW} and \cite{We} for further reading.

\subsection*{The Stone--\v{C}ech compactification} Let $X$ be a completely regular space. A \textit{compactification} $\gamma X$ of $X$ is a compact space $\gamma X$ which contains $X$ as a dense subspace. The \textit{Stone--\v{C}ech compactification} of $X$, denoted by $\beta X$, is the compactification of $X$ which is characterized by the property that every continuous mapping $f:X\rightarrow K$, where $K$ is a compact space, is continuously extendable over $\beta X$; this continuous extension of $f$ is necessarily unique and is denoted by $f_\beta:\beta X\rightarrow K$ (or in occasions by $f^\beta:\beta X\rightarrow K$ to simplify the notation). For a completely regular space the Stone--\v{C}ech compactification always exists. In what follows we will use the following properties of $\beta X$.
\begin{itemize}
  \item The space $X$ is locally compact if and only if $X$ is open in $\beta X$.
  \item Let $T$ be a closed and open subspace of $X$. Then $\mathrm{cl}_{\beta X}T$ is closed and open in $\beta X$.
  \item Let $X\subseteq T\subseteq\beta X$. Then $\beta T=\beta X$.
  \item Let $T$ be a $C^*$-embedded subspace of $X$, that is, $T$ is a subspace of $X$ such that every continuous bounded mapping $f:T\rightarrow\mathbb{R}$ can be extended continuously to a mapping $F:X\rightarrow\mathbb{R}$. (In particular, if $X$ is normal and $T$ is a closed subspace of $X$.) Then $\beta T=\mathrm{cl}_{\beta X}T$.
  \item Let $S$ and $T$ be zero-sets of $X$. (In particular, if $S$ and $T$ are disjoint closed and open subspaces of $X$.) Then
  \[\mathrm{cl}_{\beta X}(S\cap T)=\mathrm{cl}_{\beta X}S\cap\mathrm{cl}_{\beta X}T.\]
  In particular, disjoint zero-sets of $X$ have disjoint closures in $\beta X$.
\end{itemize}

\part{General theory}\label{HFPG}

\bigskip

In this part for a space $X$ we assign to each (set theoretic) ideal of $X$ certain (algebraic) ideals of $C_B(X)$. We then prove representation theorems for the assigned ideals of $C_B(X)$ in the case when $X$ is either completely regular or normal. This is done by means of associating a certain subspace of $\beta X$ to each ideal of $X$. This subspace of $\beta X$ plays a crucial role in our future study in Part \ref{KHGJ}.

This part is divided into two sections, in either of which for a space $X$ we introduce and study an ideal of $C_B(X)$ associated to each ideal of $X$. Results of this part are stated and proved in the most general context; Part \ref{KHGJ} will be devoted subsequently to the consideration of specific examples.

\section{The normed ideal $C^{\mathfrak I}_{00}(X)$ of $C_B(X)$}\label{GPG}

In this section for a space $X$ and an ideal $\mathfrak{I}$ of $X$ we introduce the ideal $C^{\mathfrak I}_{00}(X)$ of $C_B(X)$. The definition of  $C^{\mathfrak I}_{00}(X)$ is modeled on (and generalizes) the definition of the ideal $C_{00}(X)$ of $C_B(X)$ which consists of all elements in $C_B(X)$ whose support is compact.

Results of this section will generalize those we have already obtained in \cite{Ko6}, \cite{Ko10} and \cite{Ko11}.

Recall that for a space $X$ and a subset $A$ of $X$ a \textit{neighborhood} of $A$ in $X$ is a subset $U$ of $X$ such that $A\subseteq\mathrm{int}_XU$.

\begin{definition}\label{HG}
Let $X$ be a space and let $\mathfrak{I}$ be an ideal in $X$. Define
\[C^{\mathfrak I}_{00}(X)=\big\{f\in C_B(X):\mathrm{supp}(f)\mbox{ has a null neighborhood in }X\big\}.\]
\end{definition}

The following example justifies our use of the notation $C^{\mathfrak I}_{00}(X)$.

\begin{example}\label{HUJ}
Let $X$ be a locally compact space and let
\[\mathfrak{I}=\{A\subseteq X:\mathrm{cl}_X A \mbox{ is compact}\}.\]
Trivially, $\mathfrak{I}$ is an ideal in $X$. As we see now, in this case
\[C^{\mathfrak I}_{00}(X)=C_{00}(X).\]
Let $f\in C_B(X)$. It is obvious that if $\mathrm{supp}(f)$ has a null neighborhood $U$ in $X$, then $\mathrm{supp}(f)$ is compact, as it is closed in $\mathrm{cl}_X U$ and the latter is so. Now, suppose that $\mathrm{supp}(f)$ is compact. Since $X$ is locally compact, for each $x\in X$ there is an open neighborhood $V_x$ of $x$ in $X$ such that $\mathrm{cl}_X V_x$ is compact. The set $\{V_x:x\in X\}$ forms an open cover for $\mathrm{supp}(f)$. Therefore
\[\mathrm{supp}(f)\subseteq V_{x_1}\cup\cdots\cup V_{x_n}=V,\]
where $x_i\in X$ for each $i=1,\ldots,n$. Clearly, $V$ is a neighborhood of $\mathrm{supp}(f)$ in $X$, and it is null, as
\[\mathrm{cl}_XV=\mathrm{cl}_XV_{x_1}\cup\cdots\cup\mathrm{cl}_X V_{x_n},\]
(being the union of a finite number of compact subspaces) is compact.
\end{example}

In order to prove our representation theorem in this section we need to prove a number of lemmas. This we will do next.

\begin{lemma}\label{TTG}
Let $X$ be a space and let $\mathfrak{I}$ be an ideal in $X$. Then $C^{\mathfrak I}_{00}(X)$ is an ideal in $C_B(X)$.
\end{lemma}

\begin{proof}
Note that $C^{\mathfrak I}_{00}(X)$ is non-empty, as it contains $\mathbf{0}$. (Observe that there always exists a null subset of $X$; this constitutes a neighborhood for $\emptyset=\mathrm{supp}(\mathbf{0})$ in $X$.) To show that $C^{\mathfrak I}_{00}(X)$ is closed under addition, let $f,g\in C^{\mathfrak I}_{00}(X)$. Then, there exist null neighborhoods $U$ and $V$ of $\mathrm{supp}(f)$ and $\mathrm{supp}(g)$ in $X$, respectively. Note that $\mathrm{coz}(f+g)\subseteq\mathrm{coz}(f)\cup\mathrm{coz}(g)$. Thus
\[\mathrm{supp}(f+g)\subseteq\mathrm{supp}(f)\cup\mathrm{supp}(g)\subseteq\mathrm{int}_X U\cup\mathrm{int}_X V\subseteq\mathrm{int}_X (U\cup V).\]
Therefore $\mathrm{supp}(f+g)$ has a null neighborhood in $X$, namely $U\cup V$. Then $f+g\in C^{\mathfrak I}_{00}(X)$. Next, let $f\in C^{\mathfrak I}_{00}(X)$ and $g\in C_B(X)$. Note that $\mathrm{coz}(fg)\subseteq\mathrm{coz}(f)$. Thus $\mathrm{supp}(fg)\subseteq\mathrm{supp}(f)$. In particular, $\mathrm{supp}(fg)$ has a null neighborhoods in $X$, as $\mathrm{supp}(f)$ does. Therefore $fg\in C^{\mathfrak I}_{00}(X)$. That $C^{\mathfrak I}_{00}(X)$ is closed under scalar multiplication follows trivially.
\end{proof}

The subspace $\lambda_{\mathfrak I} X$ of $\beta X$ introduced below plays a crucial role in our study. The space $\lambda_{\mathfrak I} X$ has been first considered in \cite{Ko3} (and later in \cite{Ko4}, \cite{Ko7}, \cite{Ko13} and \cite{Ko8}) to study certain classes of topological extensions.

\begin{definition}\label{HWA}
Let $X$ be a completely regular space and let $\mathfrak{I}$ be an ideal in $X$. Define
\[\lambda_{\mathfrak I} X=\bigcup\big\{\mathrm{int}_{\beta X}\mathrm{cl}_{\beta X}C:C\in\mathrm{Coz}(X)\mbox{ and }\mathrm{cl}_XC\mbox{ has a null neighborhood in $X$}\big\}.\]
\end{definition}

The space $\lambda_{\mathfrak I} X$ just defined may have a better expression if one requires $X$ to be normal. This is the subject matter of the next proposition. (See also Lemma \ref{KPJG}.) We need the following simple lemma.

Observe that for a space $X$ and a dense subspace $D$ of $X$ we have
\[\mathrm{cl}_XU=\mathrm{cl}_X(U\cap D)\]
for every open subspace $U$ of $X$.

\begin{lemma}\label{LKG}
Let $X$ be a completely regular space. Let $f:X\rightarrow[0,1]$ be a continuous mapping and $0<r<1$. Then
\[f_\beta^{-1}\big([0,r)\big)\subseteq\mathrm{int}_{\beta X}\mathrm{cl}_{\beta X}f^{-1}\big([0,r)\big).\]
\end{lemma}

\begin{proof}
Note that
\[f_\beta^{-1}\big([0,r)\big)\subseteq\mathrm{int}_{\beta X}\mathrm{cl}_{\beta X}f_\beta^{-1}\big([0,r)\big)\]
and
\[\mathrm{cl}_{\beta X}f_\beta^{-1}\big([0,r)\big)=\mathrm{cl}_{\beta X}\big(X\cap f_\beta^{-1}\big([0,r)\big)\big)=\mathrm{cl}_{\beta X}f^{-1}\big([0,r)\big).\]
\end{proof}

\begin{proposition}\label{KJHF}
Let $X$ be a normal space and let $\mathfrak{I}$ be an ideal in $X$. Then
\[\lambda_{\mathfrak I} X=\bigcup\{\mathrm{int}_{\beta X}\mathrm{cl}_{\beta X}U:\mathrm{cl}_XU\mbox{ is null}\,\}.\]
\end{proposition}

\begin{proof}
Denote
\[T=\bigcup\{\mathrm{int}_{\beta X}\mathrm{cl}_{\beta X}U:\mathrm{cl}_XU\mbox{ is null}\}.\]
To show that $\lambda_{\mathfrak I} X\subseteq T$, let $C\in\mathrm{Coz}(X)$ such that $\mathrm{cl}_XC$ has a null neighborhood $V$ in $X$. Since $X$ is normal, there is an open neighborhood $U$ of $\mathrm{cl}_XC$ whose closure $\mathrm{cl}_XU$ is contained in $V$. Thus $\mathrm{cl}_XU$ is null and therefore $\mathrm{int}_{\beta X}\mathrm{cl}_{\beta X}U\subseteq T$ by the definition of $T$. But then $\mathrm{int}_{\beta X}\mathrm{cl}_{\beta X}C\subseteq T$. This show that $\lambda_{\mathfrak I} X\subseteq T$. To show the reverse inclusion, let $t\in T$. Note that $T$ is an open subspace of $\beta X$. Let $f:\beta X\rightarrow[0,1]$ be a continuous mapping such that
\[f(t)=0\quad\mbox{and}\quad f|_{\beta X\setminus T}=\mathbf{1}.\]
Observe that
$f^{-1}([0,1/2])$ is contained in $T$, and it is compact, as it is closed in $\beta X$. Therefore
\begin{equation}\label{FGFF}
f^{-1}\big([0,1/2]\big)\subseteq\mathrm{int}_{\beta X}\mathrm{cl}_{\beta X}U_1\cup\cdots\cup\mathrm{int}_{\beta X}\mathrm{cl}_{\beta X}U_n,
\end{equation}
where $\mathrm{cl}_XU_i$ is null for each $i=1,\ldots,n$. Thus, if we intersect the two sides of (\ref{FGFF}) with $X$ we have
\[X\cap f^{-1}\big([0,1/2]\big)\subseteq\mathrm{cl}_XU_1\cup\cdots\cup\mathrm{cl}_XU_n=\mathrm{cl}_XU,\]
where $U=U_1\cup\cdots\cup U_n$. Let $C=X\cap f^{-1}([0,1/3))$. Then $C\in\mathrm{Coz}(X)$ as $f^{-1}([0,1/3))\in\mathrm{Coz}(\beta X)$. (To see the latter, let
\[g=\max\Big\{\mathbf{0},\mathbf{\frac{1}{3}}-f\Big\}\]
and observe that $\mathrm{coz}(g)=f^{-1}([0,1/3))$.) Now $\mathrm{cl}_XC$ has a null neighborhood in $X$, namely $\mathrm{cl}_XU$. Therefore $\mathrm{int}_{\beta X}\mathrm{cl}_{\beta X}C\subseteq\lambda_{\mathfrak I} X$. But then $t\in\lambda_{\mathfrak I} X$, as $t\in f^{-1}([0,1/3))$ and  $f^{-1}([0,1/3))\subseteq\mathrm{int}_{\beta X}\mathrm{cl}_{\beta X}C$ by Lemma \ref{LKG}. This shows that $T\subseteq\lambda_{\mathfrak I} X$, and concludes the proof.
\end{proof}

\begin{definition}
Let $X$ be a space and let $\mathfrak{I}$ be an ideal in $X$. Then $X$ is called \textit{locally null} (with respect to $\mathfrak{I}$) if every point of $X$ has a null neighborhood in $X$.
\end{definition}

Observe that for a completely regular space $X$ and a continuous mapping $f:\beta X\rightarrow[0,1]$ we have $(f|_X)_\beta=f$, as both mappings are continuous and coincide on the dense subspace $X$ of $\beta X$.

Let $X$ be a space. An ideal $H$ of $C_B(X)$ is said to be \textit{non-vanishing} (or \textit{free} or \textit{of empty hull}) if for every $x\in X$ there is some $h\in H$ such that $h(x)\neq 0$.

\begin{lemma}\label{BBV}
Let $X$ be a completely regular space and let $\mathfrak{I}$ be an ideal in $X$. The following are equivalent:
\begin{itemize}
\item[\rm(1)] $X\subseteq\lambda_{\mathfrak I} X$.
\item[\rm(2)] $X$ is locally null.
\item[\rm(3)] $C^{\mathfrak I}_{00}(X)$ is non-vanishing.
\end{itemize}
\end{lemma}

\begin{proof}
(1) \emph{implies} (2). Let $x\in X$. Then $x\in\lambda_{\mathfrak I} X$ and therefore $x\in\mathrm{int}_{\beta X}\mathrm{cl}_{\beta X}D$ for some $D\in\mathrm{Coz}(X)$ such that $\mathrm{cl}_XD$ has a null neighborhood $V$ in $X$. But $V$ is then a neighborhood of $x$ in $X$ as well, as $x\in\mathrm{cl}_{\beta X}D\cap X=\mathrm{cl}_XD$.

(2) \emph{implies} (1). Let $x\in X$ and let $U$ be a null neighborhood of $x$ in $X$. Let $f:X\rightarrow[0,1]$ be a continuous mapping with $f(x)=0$ and $f|_{X\setminus \mathrm{int}_XU}=\mathbf{1}$. Let $C=f^{-1}([0,1/2))$. Then $C\in\mathrm{Coz}(X)$. Now $\mathrm{cl}_XC\subseteq f^{-1}([0,1/2])$ and $f^{-1}([0,1/2])\subseteq\mathrm{int}_XU$. Thus $U$ is a null neighborhood of $\mathrm{cl}_XC$ in $X$. Therefore $\mathrm{int}_{\beta X}\mathrm{cl}_{\beta X}C\subseteq\lambda_{\mathfrak I} X$. But then $x\in\lambda_{\mathfrak I} X$, as $x\in f_\beta^{-1}([0,1/2))$ and  $f_\beta^{-1}([0,1/2))\subseteq\mathrm{int}_{\beta X}\mathrm{cl}_{\beta X}C$ by Lemma \ref{LKG}.

(2) \emph{implies} (3). Note that $C^{\mathfrak I}_{00}(X)$ is an ideal in $C_B(X)$ by Lemma \ref{TTG}. Let $x\in X$ and let $V$ be a null neighborhood of $x$ in $X$. Let $W$ be an open neighborhood of $x$ in $X$ with $\mathrm{cl}_XW\subseteq\mathrm{int}_XV$. Let $g:X\rightarrow[0,1]$ be a continuous mapping with $g(x)=1$ and $g|_{X\setminus W}=\mathbf{0}$. Then $\mathrm{supp}(g)\subseteq\mathrm{cl}_XW$, as $\mathrm{coz}(g)\subseteq W$. Thus $V$ is a null neighborhood of $\mathrm{supp}(g)$ in $X$. Therefore $g\in C^{\mathfrak I}_{00}(X)$.

(3) \emph{implies} (2). Let $x\in X$. Then $x\in\mathrm{coz}(h)$ for some $h\in C^{\mathfrak I}_{00}(X)$. Since $\mathrm{supp}(h)$ has a null neighborhood in $X$ and $x\in\mathrm{supp}(h)$, it then follows that $x$ has a null neighborhood in $X$.
\end{proof}

\begin{lemma}\label{HDHD}
Let $X$ be a completely regular space and let $\mathfrak{I}$ be an ideal in $X$. For any subspace $A$ of $X$, if $\mathrm{cl}_{\beta X}A\subseteq\lambda_{\mathfrak I} X$ then $\mathrm{cl}_XA$ has a null neighborhood in $X$.
\end{lemma}

\begin{proof}
By compactness, we have
\begin{equation}\label{TDJB}
\mathrm{cl}_{\beta X}A\subseteq\mathrm{int}_{\beta X}\mathrm{cl}_{\beta X}C_1\cup\cdots\cup\mathrm{int}_{\beta X}\mathrm{cl}_{\beta X}C_n,
\end{equation}
where $C_i\in\mathrm{Coz}(X)$ for each $i=1,\ldots,n$, and $\mathrm{cl}_XC_i$ has a null neighborhood $U_i$ in $X$. Intersecting both sides of (\ref{TDJB}) with $X$ we have
\[\mathrm{cl}_XA\subseteq\mathrm{cl}_XC_1\cup\cdots\cup\mathrm{cl}_XC_n.\]
Now
\[\mathrm{cl}_XC_1\cup\cdots\cup\mathrm{cl}_XC_n\subseteq\mathrm{int}_X U_1\cup\cdots\cup\mathrm{int}_X U_n\subseteq\mathrm{int}_XW\]
where $W=U_1\cup\cdots\cup U_n$. Therefore $\mathrm{cl}_XA$ has a null neighborhood in $X$, namely $W$.
\end{proof}

\begin{lemma}\label{HGBV}
Let $X$ be a completely regular space locally null with respect to an ideal $\mathfrak{I}$ of $X$. The following are equivalent:
\begin{itemize}
\item[\rm(1)] $\lambda_{\mathfrak I} X$ is compact.
\item[\rm(2)] ${\mathfrak I}$ is non-proper.
\item[\rm(3)] $C^{\mathfrak I}_{00}(X)$ is unital.
\end{itemize}
\end{lemma}

\begin{proof}
(1) \emph{implies} (2). Note that $X\subseteq\lambda_{\mathfrak I} X$ by Lemma \ref{BBV}, as $X$ is locally null. Since $\lambda_{\mathfrak I} X$ is compact we have $\mathrm{cl}_{\beta X}X\subseteq\lambda_{\mathfrak I} X$. Therefore $X$ is null by Lemma \ref{HDHD}.

(2) \emph{implies} (3). Suppose that $X$ is null. Then the mapping $\mathbf{1}$ is the unit element of $C^{\mathfrak I}_{00}(X)$.

(3) \emph{implies} (1). Suppose that $C^{\mathfrak I}_{00}(X)$ has a unit element $u$. Let $x\in X$. Since $X$ is locally null, there is a null neighborhood $U_x$ of $x$ in $X$. Let $V_x$ be an open neighborhoods of $x$ in $X$ such that $\mathrm{cl}_XV_x\subseteq\mathrm{int}_X U_x$. Let $f_x:X\rightarrow[0,1]$ be a continuous mapping such that $f_x(x)=1$ and $f_x|_{X\setminus V_x}=\mathbf{0}$. Then $U_x$ is a neighborhood of $\mathrm{supp}(f_x)$ in $X$, as $\mathrm{supp}(f_x)\subseteq\mathrm{cl}_XV_x$. Therefore $f_x\in C^{\mathfrak I}_{00}(X)$. We have
\[u(x)=u(x)f_x(x)=f_x(x)=1\]
Thus $u=\mathbf{1}$ and therefore $X=\mathrm{supp}(u)$ is null. Since $X\in\mathrm{Coz}(X)$ trivially, it follows that $\lambda_{\mathfrak I} X=\beta X$ is compact.
\end{proof}

\begin{lemma}\label{KPJG}
Let $X$ be a normal space and let $\mathfrak{I}$ be an ideal in $X$. Then
\[\lambda_{\mathfrak I} X=\bigcup\{\mathrm{cl}_{\beta X}A:\mathrm{cl}_XA\mbox{ has a null neighborhood in }X\}.\]
\end{lemma}

\begin{proof}
Let
\[T=\bigcup\{\mathrm{cl}_{\beta X}A:\mathrm{cl}_XA\mbox{ has a null neighborhood in }X\}.\]
By the definition of $\lambda_{\mathfrak I} X$ it is clear that $\lambda_{\mathfrak I} X\subseteq T$. We check that $T\subseteq\lambda_{\mathfrak I} X$.

Let $A$ be a subset of $X$ whose closure $\mathrm{cl}_XA$ has a null neighborhood in $X$, say $U$. Since $X$ is normal, by the Urysohn lemma, there exists a continuous mapping $g:X\rightarrow[0,1]$ such that
\[g|_{\mathrm{cl}_XA}=\mathbf{0}\quad\mbox{and}\quad g|_{X\setminus\mathrm{int}_X U}=\mathbf{1}.\]
Let
\[C=g^{-1}\big([0,1/2)\big)\in\mathrm{Coz}(X).\]
Then $\mathrm{cl}_XC$ has a null neighborhood in $X$, namely $U$, as
\[\mathrm{cl}_XC\subseteq g^{-1}\big([0,1/2]\big)\subseteq\mathrm{int}_X U.\]
Therefore $\mathrm{int}_{\beta X}\mathrm{cl}_{\beta X}C\subseteq\lambda_{\mathfrak I} X$. But $g_\beta^{-1}([0,1/2))\subseteq\mathrm{int}_{\beta X}\mathrm{cl}_{\beta X}C$ by Lemma \ref{LKG}, and thus
\[\mathrm{cl}_{\beta X}A\subseteq\mathrm{z}(g_\beta)\subseteq g_\beta^{-1}\big([0,1/2)\big)\subseteq\lambda_{\mathfrak I} X.\]
\end{proof}

\begin{definition}\label{WWA}
Let $X$ be a completely regular space locally null with respect to an ideal $\mathfrak{I}$ of $X$. For any $f\in C_B(X)$ denote
\[f_\lambda=f_\beta|_{\lambda_{\mathfrak I} X}.\]
\end{definition}

Observe that by Lemma \ref{BBV} the mapping $f_\lambda$ extends $f$.

\begin{lemma}\label{TES}
Let $X$ be a normal space locally null with respect to an ideal $\mathfrak{I}$ of $X$. For any $f\in C_B(X)$ the following are equivalent:
\begin{itemize}
\item[\rm(1)] $f\in C^{\mathfrak I}_{00}(X)$.
\item[\rm(2)] $f_\lambda\in C_{00}(\lambda_{\mathfrak I} X)$.
\end{itemize}
\end{lemma}

\begin{proof}
Note that $X\subseteq\lambda_{\mathfrak I} X$ by Lemma \ref{BBV}, as $X$ is locally null.

(1) \emph{implies} (2). Since $\mathrm{cl}_X\mathrm{coz}(f)=\mathrm{supp}(f)$ has a null neighborhood in $X$ we have $\mathrm{cl}_{\beta X}\mathrm{coz}(f)\subseteq\lambda_{\mathfrak I} X$ by Lemma \ref{KPJG}. Thus
\begin{eqnarray*}
\mathrm{supp}(f_\lambda)=\mathrm{cl}_{\lambda_{\mathfrak I} X}\mathrm{coz}(f_\lambda)&=&\mathrm{cl}_{\lambda_{\mathfrak I} X}\big(X\cap \mathrm{coz}(f_\lambda)\big)\\&=&\mathrm{cl}_{\lambda_{\mathfrak I} X}\mathrm{coz}(f)=\lambda_{\mathfrak I}X\cap\mathrm{cl}_{\beta X}\mathrm{coz}(f)=\mathrm{cl}_{\beta X}\mathrm{coz}(f)
\end{eqnarray*}
is closed in $\beta X$ and is therefore compact.

(2) \emph{implies} (1). Note that $\mathrm{cl}_{\beta X}\mathrm{coz}(f)\subseteq\mathrm{supp}(f_\lambda)$, as $\mathrm{coz}(f)\subseteq\mathrm{coz}(f_\lambda)\subseteq\mathrm{supp}(f_\lambda)$ and the latter is compact. In particular $\mathrm{cl}_{\beta X}\mathrm{coz}(f)\subseteq\lambda_{\mathfrak I}X$. By Lemma \ref{HDHD} it follows that $\mathrm{supp}(f)$ has a null neighborhood in $X$.
\end{proof}

A version of the Banach--Stone theorem (Theorem 7.1 of \cite{Be}) states that for any locally compact spaces $X$ and $Y$, the rings $C_{00}(X)$ and $C_{00}(Y)$ are isomorphic if and only if the spaces $X$ and $Y$ are homeomorphic. (See \cite{A}.) This will be used in the proof of the following main result of this section.

\begin{theorem}\label{UUS}
Let $X$ be a normal space locally null with respect to an ideal $\mathfrak{I}$ of $X$. Then $C^{\mathfrak I}_{00}(X)$ is an ideal of $C_B(X)$ isometrically isomorphic to $C_{00}(Y)$ for some unique (up to homeomorphism) locally compact space $Y$, namely, for $Y=\lambda_{\mathfrak I} X$. Furthermore,
\begin{itemize}
\item[\rm(1)] $X$ is dense in $Y$.
\item[\rm(2)] $C^{\mathfrak I}_{00}(X)$ is non-vanishing.
\item[\rm(3)] $C^{\mathfrak I}_{00}(X)$ is unital if and only if ${\mathfrak I}$ is non-proper if and only if $Y$ is compact.
\end{itemize}
\end{theorem}

\begin{proof}
Observe that $C^{\mathfrak I}_{00}(X)$ is an ideal of $C_B(X)$ by Lemma \ref{TTG}. Define a mapping
\[\psi:C^{\mathfrak I}_{00}(X)\longrightarrow C_{00}(\lambda_{\mathfrak I} X)\]
by
\[\psi(f)=f_\lambda\]
for any $f\in C^{\mathfrak I}_{00}(X)$. By Lemma \ref{TES} the mapping $\psi$ is well defined. It is clear that $\psi$ is an algebra homomorphism and that it is injective. (Again, note that $X\subseteq\lambda_{\mathfrak I} X$, and use the fact that any two scalar-valued continuous mapping on $\lambda_{\mathfrak I} X$ coincide, provided that they agree on the dense subspace $X$ of $\lambda_{\mathfrak I} X$.) To show that $\psi$ is surjective, let $g\in C_{00}(\lambda_{\mathfrak I} X)$. Then $(g|_X)_\lambda=g$ and thus  $g|_X\in C^{\mathfrak I}_{00}(X)$ by Lemma \ref{TES}. Note that $\psi(g|_X)=g$. To show that $\psi$ is an isometry, let $h\in C^{\mathfrak I}_{00}(X)$. Then
\[|h_\lambda|(\lambda_{\mathfrak I} X)=|h_\lambda|(\mathrm{cl}_{\lambda_{\mathfrak I} X}X)\subseteq\overline{|h_\lambda|(X)}=\overline{|h|(X)}\subseteq\big[0,\|h\|\big]\]
where the bar denotes the closure in $\mathbb{R}$. This yields $\|h_\lambda\|\leq\|h\|$. That $\|h\|\leq\|h_\lambda\|$ is clear, as $h_\lambda$ extends $h$.

Note that $\lambda_{\mathfrak I} X$ is locally compact, as it is open in the compact space $\beta X$.

The uniqueness of $\lambda_{\mathfrak I} X$ follows from the fact that for any locally compact space $T$ the ring $C_{00}(T)$ determines the topology of $T$.

(1). By Lemma \ref{BBV} we have $X\subseteq\lambda_{\mathfrak I} X$. That $X$ is dense in $\lambda_{\mathfrak I} X$ is then obvious.

(2). This follows from Lemma \ref{BBV}.

(3). This follows from Lemma \ref{HGBV}.
\end{proof}

\section{The normed closed ideal $C^{\mathfrak I}_0(X)$ of $C_B(X)$}\label{JGGG}

In this section for a space $X$ and an ideal $\mathfrak{I}$ of $X$ we introduce the closed ideal $C^{\mathfrak I}_0(X)$ of $C_B(X)$. The definition of  $C^{\mathfrak I}_0(X)$ is modeled on (and generalizes) the definition of the closed ideal $C_0(X)$ of $C_B(X)$ which consists of all elements in $C_B(X)$ which vanishes at infinity.

\begin{definition}\label{HHLG}
Let $X$ be a space and let $\mathfrak{I}$ be an ideal in $X$. Define
\[C^{\mathfrak I}_0(X)=\big\{f\in C_B(X):|f|^{-1}\big([1/n,\infty)\big)\mbox{ is null for each }n\big\}.\]
\end{definition}

The following is to justify our use of the notation $C^{\mathfrak I}_0(X)$.

\begin{example}\label{GHUJ}
Let $X$ be a locally compact space. Consider the ideal
\[\mathfrak{I}=\{A\subseteq X:\mathrm{cl}_X A \mbox{ is compact}\}\]
of $X$. Then, an argument analogous to the one given in Example \ref{HUJ} shows that
\[C^{\mathfrak I}_0(X)=C_0(X).\]
\end{example}

Under certain conditions the representation of $C^{\mathfrak I}_0(X)$ given in Definition \ref{HHLG} simplifies. This is the context of the next result.

\begin{proposition}\label{JFJG}
Let $X$ be a space and let $\mathfrak{I}$ be a $\sigma$-ideal in $X$. Then
\[C^{\mathfrak I}_0(X)=\big\{f\in C_B(X):\mathrm{coz}(f)\mbox{ is null}\big\}.\]
\end{proposition}

\begin{proof}
Let $f\in C^{\mathfrak I}_0(X)$. Let $n$ be a positive integer. Then $|f|^{-1}([1/n,\infty))$ is null and thus, since $\mathfrak{I}$ is a $\sigma$-ideal in $X$, so is the countable union
\[\mathrm{coz}(f)=\bigcup_{n=1}^\infty|f|^{-1}\big([1/n,\infty)\big).\]

The converse is trivial, as if $\mathrm{coz}(f)$ is null, where $f\in C_B(X)$, then so is its subset $|f|^{-1}([1/n,\infty))$ for each positive integer $n$.
\end{proof}

In order to prove our representation theorem in this section we will prove a number of lemmas first.

\begin{lemma}\label{DGH}
Let $X$ be a space and let $\mathfrak{I}$ be an ideal in $X$. Then
$C^{\mathfrak I}_0(X)$ is a closed ideal in $C_B(X)$. Furthermore,
\[C^{\mathfrak I}_{00}(X)\subseteq C^{\mathfrak I}_0(X).\]
\end{lemma}

\begin{proof}
First we verify the second statement. Observe that if $f\in C^{\mathfrak I}_{00}(X)$ then $|f|^{-1}([1/n,\infty))$ is null for each positive integer $n$, as it is contained in $\mathrm{supp}(f)$ and the latter is so. Thus $f\in C^{\mathfrak I}_0(X)$.

Next, we show that $C^{\mathfrak I}_0(X)$ is an ideal in the algebra $C_B(X)$. Note that $C^{\mathfrak I}_0(X)$ is non-empty, as it contains $C^{\mathfrak I}_{00}(X)$, and the latter is so (since it is an ideal in $C_B(X)$). To show that $C^{\mathfrak I}_0(X)$ is closed under addition, let $f,g\in C^{\mathfrak I}_0(X)$. Let $n$ be a positive integer. The two sets $|f|^{-1}([1/(2n),\infty))$ and $|g|^{-1}([1/(2n),\infty))$ are null. But
\[|f+g|^{-1}\big([1/n,\infty)\big)\subseteq|f|^{-1}\Big(\Big[\frac{1}{2n},\infty\Big)\Big)\cup|g|^{-1}\Big(\Big[\frac{1}{2n},\infty\Big)\Big).\]
Thus $|f+g|^{-1}([1/n,\infty))$ is null. Therefore $f+g\in C^{\mathfrak I}_0(X)$. Next, let $f\in C^{\mathfrak I}_0(X)$ and $g\in C_B(X)$. Let $m$ be a positive integer such that $|g(x)|\leq m$ for each $x\in X$. Then
\[|fg|^{-1}\big([1/n,\infty)\big)\subseteq|f|^{-1}\Big(\Big[\frac{1}{mn},\infty\Big)\Big),\]
and $|f|^{-1}([1/(mn),\infty))$ is null. Thus $|fg|^{-1}([1/n,\infty))$ is null. Therefore $fg\in C^{\mathfrak I}_0(X)$. That $C^{\mathfrak I}_0(X)$ is closed under scalar multiplication follows analogously.

Finally, we show that $C^{\mathfrak I}_0(X)$ is closed in $C_B(X)$. Let $f$ be in the closure in $C_B(X)$ of $C^{\mathfrak I}_0(X)$. Let $n$ be a positive integer. There exists some $g\in C^{\mathfrak I}_0(X)$ with $\|f-g\|<1/(2n)$. Let $t\in |f|^{-1}([1/n,\infty))$. Then
\[\frac{1}{n}\leq\big|f(t)\big|\leq\big|f(t)-g(t)\big|+\big|g(t)\big|\leq\|f-g\|+\big|g(t)\big|\leq\frac{1}{2n}+\big|g(t)\big|\]
and thus $|g(t)|\geq 1/(2n)$. That is $t\in |g|^{-1}([1/(2n),\infty))$. Therefore
\[|f|^{-1}\big([1/n,\infty)\big)\subseteq|g|^{-1}\Big(\Big[\frac{1}{2n},\infty\Big)\Big).\]
Since the latter is null, so is $|f|^{-1}([1/n,\infty))$. Thus $f\in C^{\mathfrak I}_0(X)$.
\end{proof}

\begin{lemma}\label{KGV}
Let $X$ be a completely regular space and let $\mathfrak{I}$ be an ideal in $X$. The following are equivalent:
\begin{itemize}
\item[\rm(1)] $X\subseteq\lambda_{\mathfrak I} X$.
\item[\rm(2)] $X$ is locally null.
\item[\rm(3)] $C^{\mathfrak I}_0(X)$ is non-vanishing.
\item[\rm(4)] $C^{\mathfrak I}_{00}(X)$ is non-vanishing.
\end{itemize}
\end{lemma}

\begin{proof}
The equivalence of (1), (2) and (4) follows from Lemma \ref{BBV}.

(4) \emph{implies} (3). Note that $C^{\mathfrak I}_0(X)$ contains $C^{\mathfrak I}_{00}(X)$ by Lemma \ref{DGH}. Thus $C^{\mathfrak I}_0(X)$ is non-vanishing if $C^{\mathfrak I}_{00}(X)$ is so.

(3) \emph{implies} (2). Let $x\in X$. Then $f(x)\neq0$ for some $f\in C^{\mathfrak I}_0(X)$. Let $n$ be a positive integer such that $|f(x)|>1/n$. Then $|f|^{-1}([1/n,\infty))$ is a null neighborhood of $x$ in $X$.
\end{proof}

\begin{lemma}\label{HFBY}
Let $X$ be a completely regular space locally null with respect to an ideal $\mathfrak{I}$ of $X$. The following are equivalent:
\begin{itemize}
\item[\rm(1)] $\lambda_{\mathfrak I} X$ is compact.
\item[\rm(2)] ${\mathfrak I}$ is non-proper.
\item[\rm(3)] $C^{\mathfrak I}_0(X)$ is unital.
\item[\rm(4)] $C^{\mathfrak I}_{00}(X)$ is unital.
\end{itemize}
\end{lemma}

\begin{proof}
The equivalence of (1), (2) and (4) follows from Lemma \ref{HGBV}.

(2) \emph{implies} (3). Suppose that $X$ is null. Then the mapping $\mathbf{1}$ is the unit element of $C^{\mathfrak I}_0(X)$.

(3) \emph{implies} (2).  Suppose that $C^{\mathfrak I}_0(X)$ has a unit element $u$. Let $U_x$, $V_x$ and $f_x$ be as defined in the proof of Lemma \ref{HGBV}. Note that $f_x\in C^{\mathfrak I}_0(X)$, as $f_x\in C^{\mathfrak I}_{00}(X)$ and $C^{\mathfrak I}_{00}(X)\subseteq C^{\mathfrak I}_0(X)$ by Lemma \ref{DGH}. Arguing as in the proof of Lemma \ref{HGBV} we have $u=\mathbf{1}$. This implies that $X=u^{-1}([1,\infty))$ is null.
\end{proof}

Recall that in any space any two disjoint zero-sets are \textit{completely separated}, in the sense that, there exists a continuous $[0,1]$-valued mapping on the space which equals to $\mathbf{0}$ on one and $\mathbf{1}$ on the other. (To see this, let $S$ and $T$ be a pair of disjoint zero-sets in a space $X$. Let $S=\mathrm{z}(f)$ and $T=\mathrm{z}(g)$ for some continuous mappings $f,g:X\rightarrow[0,1]$. The mapping
\[h=\frac{f}{f+g}:X\longrightarrow[0,1]\]
is well defined and continuous with $h|_S=\mathbf{0}$ and $h|_T=\mathbf{1}$.)

\begin{lemma}\label{TTES}
Let $X$ be a completely regular space locally null with respect to an ideal $\mathfrak{I}$ of $X$. For any $f\in C_B(X)$ the following are equivalent:
\begin{itemize}
\item[\rm(1)] $f\in C^{\mathfrak I}_0(X)$.
\item[\rm(2)] $f_\lambda\in C_0(\lambda_{\mathfrak I} X)$.
\end{itemize}
\end{lemma}

\begin{proof}
Note that $X\subseteq\lambda_{\mathfrak I} X$ by Lemma \ref{BBV}, as $X$ is locally null.

(1) \emph{implies} (2). Let $k$ be a positive integer. Observe that
\[A=|f|^{-1}\big([1/k,\infty)\big)\quad\mbox{and}\quad B=|f|^{-1}\big(\big[0, 1/(k+1)\big]\big)\]
are zero-sets in $X$ and they are disjoint. Thus, they are completely separated in $X$. Let $g_k:X\rightarrow[0,1]$ be a continuous mapping such that $g_k|_A=\mathbf{0}$ and $g_k|_B=\mathbf{1}$. Let
\[C_k=g_k^{-1}\big([0,1/2)\big)\in\mathrm{Coz}(X).\]
Then $\mathrm{cl}_XC_k$ has a null neighborhood in $X$, as
\[\mathrm{cl}_XC_k\subseteq g_k^{-1}\big([0,1/2]\big)\subseteq|f|^{-1}\Big(\Big(\frac{1}{k+1},\infty\Big)\Big),\]
and the latter is null. Therefore $\mathrm{int}_{\beta X}\mathrm{cl}_{\beta X}C_k\subseteq\lambda_{\mathfrak I} X$. Arguing as in the proof of Lemma \ref{LKG} we have
\[|f_\beta|^{-1}\big((1/k,\infty)\big)\subseteq\mathrm{cl}_{\beta X}|f|^{-1}\big((1/k,\infty)\big).\]
Since
\[\mathrm{cl}_{\beta X}|f|^{-1}\big((1/k,\infty)\big)\subseteq\mathrm{cl}_{\beta X}\mathrm{z}(g_k)\subseteq\mathrm{z}(g_k^\beta)\subseteq(g_k^\beta)^{-1}\big([0,1/2)\big)\]
and
\[(g_k^\beta)^{-1}\big([0,1/2)\big)\subseteq\mathrm{int}_{\beta X}\mathrm{cl}_{\beta X}C_k\]
by Lemma \ref{LKG}, it follows that
\begin{equation}\label{DDSD}
|f_\beta|^{-1}\big((1/k,\infty)\big)\subseteq\lambda_{\mathfrak I} X.
\end{equation}
Now, let $n$ be a positive integer. Using (\ref{DDSD}), we have
\[|f_\beta|^{-1}\big([1/n,\infty)\big)\subseteq|f_\beta|^{-1}\Big(\Big(\frac{1}{n+1},\infty\Big)\Big)\subseteq\lambda_{\mathfrak I} X.\]
Therefore
\[|f_\lambda|^{-1}\big([1/n,\infty)\big)=\lambda_{\mathfrak I} X\cap|f_\beta|^{-1}\big([1/n,\infty)\big)=|f_\beta|^{-1}\big([1/n,\infty)\big)\]
is compact, as it is closed in $\beta X$.

(2) \emph{implies} (1). Let $n$ be a positive integer. Then since $|f_\lambda|^{-1}([1/n,\infty))$ contains $|f|^{-1}([1/n,\infty))$ and it is compact, we have
\[\mathrm{cl}_{\beta X}|f|^{-1}\big([1/n,\infty)\big)\subseteq|f_\lambda|^{-1}\big([1/n,\infty)\big)\subseteq\lambda_{\mathfrak I} X.\]
But then $|f|^{-1}([1/n,\infty))$ has a null neighborhood in $X$ by Lemma \ref{HDHD}, and is therefore itself null.
\end{proof}

There is a version of the Banach--Stone theorem which states that for any locally compact spaces $X$ and $Y$, the rings $C_0(X)$ and $C_0(Y)$ are isomorphic if and only if the spaces $X$ and $Y$ are homeomorphic. (See \cite{A}.) This will be used in the proof of the following main result of this section.

\begin{theorem}\label{UDR}
Let $X$ be a completely regular space locally null with respect to an ideal $\mathfrak{I}$ of $X$. Then $C^{\mathfrak I}_0(X)$ is a closed ideal of $C_B(X)$ isometrically isomorphic to $C_0(Y)$ for some locally compact space $Y$, namely, for $Y=\lambda_{\mathfrak I} X$. The space $Y$ is unique up to homeomorphism and coincides with the spectrum of $C^{\mathfrak I}_0(X)$. Furthermore,
\begin{itemize}
\item[\rm(1)] $X$ is dense in $Y$.
\item[\rm(2)] $C^{\mathfrak I}_0(X)$ is non-vanishing.
\item[\rm(3)] $C^{\mathfrak I}_{00}(X)$ is dense in $C^{\mathfrak I}_0(X)$, if $X$ is moreover normal.
\item[\rm(4)] $C^{\mathfrak I}_0(X)$ is unital if and only if ${\mathfrak I}$ is non-proper if and only if $Y$ is compact.
\end{itemize}
\end{theorem}

\begin{proof}
Observe that $C^{\mathfrak I}_0(X)$ is a closed ideal of $C_B(X)$ by Lemma \ref{DGH}. Define a mapping
\[\phi:C^{\mathfrak I}_0(X)\longrightarrow C_0(\lambda_{\mathfrak I} X)\]
by
\[\phi(f)=f_\lambda\]
for any $f\in C^{\mathfrak I}_0(X)$. By Lemma \ref{TTES} the mapping $\phi$ is well defined, and arguing as in the proof of Theorem \ref{UUS}, it follows that $\phi$ is an isometric algebra isomorphism.

Note that $\lambda_{\mathfrak I} X$ is locally compact, as it is open in the compact space $\beta X$, and $\lambda_{\mathfrak I} X$ contains $X$ (as a dense subspace) by Lemma \ref{KGV}.

The uniqueness of $\lambda_{\mathfrak I} X$ follows from the fact that the topology of any locally compact space $T$ is determined by the algebraic structure of the ring $C_0(T)$. Note that by the commutative Gelfand--Naimark theorem the Banach algebra $C^{\mathfrak I}_0(X)$ is isometrically isomorphic to $C_0(Y')$ where $Y'$ is the spectrum of $C^{\mathfrak I}_0(X)$. Since $Y'$ is a locally compact space the uniqueness implies that $Y'=\lambda_{\mathfrak I} X$.

(1). By Lemma \ref{KGV} we have $X\subseteq\lambda_{\mathfrak I} X$. That $X$ is dense in $\lambda_{\mathfrak I} X$ is then obvious.

(2). This follows from Lemma \ref{KGV}.

(3). Let $\phi$ be as defined in the above and let $\psi=\phi|_{C^{\mathfrak I}_{00}(X)}$. Then
\[\psi:C^{\mathfrak I}_{00}(X)\longrightarrow C_{00}(\lambda_{\mathfrak I} X),\]
and $\psi$ is surjective by the proof of Theorem \ref{UUS}. The result now follows from the well known fact that $C_{00}(T)$ is dense in $C_0(T)$ for any locally compact space $T$.

(4). This follows from Lemma \ref{HFBY}.
\end{proof}

\begin{remark}\label{HFJD}
Assuming that $C^{\mathfrak I}_0(X)$ is a Banach algebra, it follows from the commutative Gelfand--Naimark theorem that $C^{\mathfrak I}_0(X)$ is isometrically isomorphic to $C_0(Y)$ for some locally compact space $Y$. Our approach here in Theorem \ref{UDR} (apart from its independence of the proof) has the advantage that it provides certain information about either the Banach algebra $C^{\mathfrak I}_0(X)$ or the space $Y$ that is not generally expected to be deducible from the standard Gelfand theory. This fact is particularly highlighted in the second Part \ref{KHGJ} where we consider specific examples of the space $X$ or the ideal ${\mathfrak I}$.
\end{remark}

For a space $X$ and an ideal $\mathfrak{I}$ in $X$ it might be of some interest to determine when $C^{\mathfrak I}_{00}(X)=C^{\mathfrak I}_0(X)$. This will be done in the next theorem.

Let $X$ be a locally compact space. It is known that $C_{00}(X)=C_0(X)$ if and only if every $\sigma$-compact subspace of $X$ is contained in a compact subspace of $X$  (see Problem 7G.2 of \cite{GJ}); in particular, $C_{00}(X)=C_0(X)$ implies that $X$ is countably compact. (Recall that a space $X$ is \textit{countably compact} if every countable open cover of $X$ has a finite subcover, equivalently, if every countable infinite subspace of $X$ has an accumulation point in $X$; see Theorem 3.10.3 of \cite{E}.) This will be used in the proof of the following.

\begin{theorem}\label{ADEFR}
Let $X$ be a normal space and let $\mathfrak{I}$ be an ideal in $X$ such that the closure in $X$ of each null subset of $X$ has a null neighborhood in $X$. The following are equivalent:
\begin{itemize}
\item[\rm(1)] $C^{\mathfrak I}_{00}(X)=C^{\mathfrak I}_0(X)$.
\item[\rm(2)] $\mathfrak{I}$ is a $\sigma$-ideal in $X$.
\end{itemize}
Also, if $X$ is moreover locally null, then $\mathrm{(1)}$ and $\mathrm{(2)}$ are equivalent to the following:
\begin{itemize}
\item[\rm(3)] Every $\sigma$-compact subspace of $\lambda_{\mathfrak I} X$ is contained in a compact subspace of $\lambda_{\mathfrak I} X$; in particular,  $\lambda_{\mathfrak I} X$ is countably compact.
\end{itemize}
\end{theorem}

\begin{proof}
(1) \emph{implies} (2). Let $A=\bigcup_{n=1}^\infty A_n$ where $A_n$ is a null subset of $X$ for each positive integer $n$. Fix some positive integer $n$. By our assumption $\mathrm{cl}_XA_n$ has a null neighborhood $U_n$ in $X$. Since $X$ is normal, by the Urysohn lemma, there exists a continuous mapping $f:X\rightarrow[0,1]$ such that
\[f|_{\mathrm{cl}_XA_n}=\mathbf{1}\quad\mbox{and}\quad f|_{X\setminus\mathrm{int}_XU_n}=\mathbf{0}.\]
Observe that $\mathrm{coz}(f_n)$ is contained in $U_n$ and is therefore null, and thus, using our assumption, its closure $\mathrm{supp}(f_n)$ has a null neighborhood in $X$. Therefore $f_n\in C^{\mathfrak I}_{00}(X)$ and thus $f_n\in C^{\mathfrak I}_0(X)$, as $C^{\mathfrak I}_{00}(X)\subseteq C^{\mathfrak I}_0(X)$ by Lemma \ref{DGH}. Let
\[f=\sum_{n=1}^\infty\frac{f_n}{2^n}:X\longrightarrow[0,1].\]
Note that $f$ is well defined by the Weierstrass $M$-test. Also, $f\in C^{\mathfrak I}_0(X)$, as $f$ is the limit of a sequence in $C^{\mathfrak I}_0(X)$ and $C^{\mathfrak I}_0(X)$ is closed in $C_B(X)$ by Lemma \ref{DGH}. But then $f\in C^{\mathfrak I}_{00}(X)$ by (1). In particular, $\mathrm{supp}(f)$ and therefore its subset $\mathrm{coz}(f)=\bigcup_{n=1}^\infty\mathrm{coz}(f_n)$ is null. But $A\subseteq\mathrm{coz}(f)$, as $A_n\subseteq\mathrm{coz}(f_n)$ for each positive integer $n$. Therefore $A$ is null.

(2) \emph{implies} (1). Let $f\in C^{\mathfrak I}_0(X)$. Then $\mathrm{coz}(f)=\bigcup_{n=1}^\infty|f|^{-1}([1/n,\infty))$ is null, as it is a countable union of null subsets of $X$. Therefore, using our assumption, its closure $\mathrm{supp}(f)$ has a null neighborhood in $X$. Thus $f\in C^{\mathfrak I}_{00}(X)$. Therefore $C^{\mathfrak I}_0(X)\subseteq C^{\mathfrak I}_{00}(X)$. That $C^{\mathfrak I}_{00}(X)\subseteq C^{\mathfrak I}_0(X)$ follows from Lemma \ref{DGH}.

Now suppose that $X$ is moreover locally null. Let
\[\phi:C^{\mathfrak I}_0(X)\longrightarrow C_0(\lambda_{\mathfrak I} X)\]
be the isomorphism defined in the proof of Theorem \ref{UDR}. Then $C^{\mathfrak I}_{00}(X)$ and $C^{\mathfrak I}_0(X)$ are identical if and only if their images  $\phi(C^{\mathfrak I}_{00}(X))$ and $\phi(C^{\mathfrak I}_0(X))$ are so. But $\phi(C^{\mathfrak I}_{00}(X))=C_{00}(\lambda_{\mathfrak I} X)$ and $\phi(C^{\mathfrak I}_0(X))=C_0(\lambda_{\mathfrak I} X)$, as it is observed in the proof of Theorem \ref{UDR}. Therefore, (1) holds if and only if $C_{00}(\lambda_{\mathfrak I} X)=C_0(\lambda_{\mathfrak I} X)$, while the latter is equivalent to (3).
\end{proof}

The following theorem is a complement to the above theorem; for a normal space $X$ and an ideal $\mathfrak{I}$ in $X$ satisfying the required assumptions, the above theorem gives a necessary and sufficient condition such that every $\sigma$-compact subspace of $\lambda_{\mathfrak I} X$ is contained in a compact subspace while the following theorem provides a necessary and sufficient condition such that $\lambda_{\mathfrak I} X$ is $\sigma$-compact. The space $\lambda_{\mathfrak I} X$ is therefore compact if it satisfies the conditions in both theorems.

Recall that a subset $A$ of a partially ordered set $(P,\leq)$ is called \textit{cofinal} if for every $p\in P$ there is some $a\in A$ such that $p\leq a$.

\begin{theorem}\label{FGFF}
Let $X$ be a normal space and let $\mathfrak{I}$ be an ideal in $X$ such that the closure in $X$ of each null subset of $X$ has a null neighborhood in $X$. The following are equivalent:
\begin{itemize}
\item[\rm(1)] $\lambda_{\mathfrak I} X$ is $\sigma$-compact.
\item[\rm(2)] $\mathfrak{I}$ contains a countable cofinal subset consisting of open subspaces of $X$.
\end{itemize}
\end{theorem}

\begin{proof}
(1) \emph{implies} (2). Let $\lambda_{\mathfrak I} X=\bigcup_{n=1}^\infty K_n$ where $K_n$ is compact for each positive integer $n$. Since $\beta X$ is normal (and $\lambda_{\mathfrak I} X$ is open in $\beta X$), for each positive integer $n$ there is an open subspaces $U_n$ of $\beta X$ such that $K_n\subseteq U_n\subseteq\mathrm{cl}_{\beta X}U_n\subseteq\lambda_{\mathfrak I} X$. We may assume our choices are so that $U_1\subseteq U_2\subseteq\cdots$. Let $V_n=X\cap U_n$ for each positive integer $n$. We check that the countable set ${\mathfrak V}=\{V_n:n=1,2,\ldots\}$ of open subspaces of $X$ is a cofinal subset of $\mathfrak{I}$. Note that $V_n$ is null for any positive integer $n$ by Lemma \ref{HDHD}, as $\mathrm{cl}_{\beta X}V_n\subseteq\lambda_{\mathfrak I}X$. Let $A$ be a null subset of $X$. By our assumption $\mathrm{cl}_XA$ has a null neighborhood in $X$. Therefore, since $X$ is normal, $\mathrm{cl}_{\beta X}A\subseteq\lambda_{\mathfrak I}X$ by Lemma \ref{KPJG}. Obverse that $\lambda_{\mathfrak I} X=\bigcup_{n=1}^\infty U_n$. Therefore, by compactness and since $U_n$'s are ascending we have $\mathrm{cl}_{\beta X}A\subseteq U_k$ for some positive integer $k$. Now, if we intersect the two sides of the latter with $X$ it gives $A\subseteq X\cap U_k=V_k$.

(2) \emph{implies} (1). Let ${\mathfrak V}=\{V_n:n=1,2,\ldots\}$ be a cofinal subset of $\mathfrak{I}$ consisting of open subspaces of $X$. Observe that $\mathrm{cl}_{\beta X}V_n\subseteq\lambda_{\mathfrak I} X$ for any positive integer $n$ by Lemma \ref{KPJG} (as $X$ is normal and $\mathrm{cl}_XV_n$ has a null neighborhood in $X$ by our assumption). That is
\[\bigcup_{n=1}^\infty\mathrm{cl}_{\beta X}V_n\subseteq\lambda_{\mathfrak I} X.\]
We show that the reverse inclusion also holds in the latter. But this follows from Lemma \ref{KPJG}, as (by cofinality of ${\mathfrak V}$) for every null subset $A$ of $X$ we have $\mathrm{cl}_{\beta X}A\subseteq\mathrm{cl}_{\beta X}V_n$ for some positive integer $n$.
\end{proof}

The following provides an example of a space $X$ and an ideal $\mathfrak{I}$ in $X$ which satisfy the assumptions of Theorems \ref{ADEFR} and \ref{FGFF}.

Recall that for a space $X$ and open covers $\mathscr{U}$ and $\mathscr{V}$ of $X$ it is said that $\mathscr{U}$ is a \textit{refinement} of $\mathscr{V}$ if each element of $\mathscr{U}$ is contained in an element of $\mathscr{V}$. An open cover $\mathscr{U}$ of a space $X$ is called \textit{locally finite} if each point of $X$ has an open neighborhood in $X$ intersecting only a finite number of the elements of $\mathscr{U}$. A regular space $X$ is called \textit{paracompact} if for every open cover $\mathscr{U}$ of $X$ there is an open cover of $X$ which refines $\mathscr{U}$. Paracompact spaces are generally considered as the simultaneous generalizations of compact Hausdorff spaces and metrizable spaces. Every metrizable space as well as every compact Hausdorff space is paracompact and every paracompact space is normal. (See Theorems 5.1.1, 5.1.3 and 5.1.5 of \cite{E}.) Every closed subspace of a paracompact is paracompact. (See Corollary 5.1.29 of \cite{E}.) Also, a paracompact space with a dense Lindel\"{o}f subspace is Lindel\"{o}f. (See Theorem 5.1.25 of \cite{E}.)

\begin{example}
Let $X$ be a paracompact space. Suppose that $X$ is locally Lindel\"{o}f, that is, every $x\in X$ has a Lindel\"{o}f neighborhood in $X$. Let
\[\mathfrak{L}=\{L\subseteq X:\mathrm{cl}_XL\mbox{ is Lindel\"{o}f}\}.\]

We check that $\mathfrak{L}$ is a $\sigma$-ideal in $X$. Clearly, if $A\subseteq L$ with $L\in\mathfrak{L}$, then $A\in\mathfrak{L}$, as $\mathrm{cl}_XL$ is Lindel\"{o}f and so is its closed subspace $\mathrm{cl}_XA$. Let $L_n\in\mathfrak{L}$ for each positive integer $n$. Let
\[K=\mathrm{cl}_X\Big(\bigcup_{n=1}^\infty L_n\Big).\]
Note that $K$ is closed in the paracompact space $X$ and is therefore itself paracompact. But then $K$ is Lindel\"{o}f, as it contains $\bigcup_{n=1}^\infty\mathrm{cl}_X L_n$ as a dense subspace and the latter is so (since it is a countable union of Lindel\"{o}f subspaces). That is $\bigcup_{n=1}^\infty L_n\in\mathfrak{L}$. This shows that $\mathfrak{L}$ is a $\sigma$-ideal in $X$.

Note that $X$ is locally null, as it is locally Lindel\"{o}f by our assumption. We now verify that the closure in $X$ of each null subset of $X$ has a null neighborhood in $X$. Let $L\in\mathfrak{L}$. For each $x\in X$ let $U_x$ be a Lindel\"{o}f neighborhood of $x$ in $X$. The collection $\{\mathrm{int}_XU_x:x\in X\}$ is an open cover for the Lindel\"{o}f space $\mathrm{cl}_XL$. Therefore
\[\mathrm{cl}_XL\subseteq\bigcup_{n=1}^\infty\mathrm{int}_XU_{x_n}\]
for some $x_1,x_2,\ldots\in X$. Since $X$ is normal, there is an open subspace $V$ of $X$ with $\mathrm{cl}_XL\subseteq V\subseteq\mathrm{cl}_XV\subseteq W$. Note that $\mathrm{cl}_XV$ is Lindel\"{o}f, as it is contained in $\bigcup_{n=1}^\infty U_{x_n}$ as a closed subspace and the letter is so (since it is a countable union of Lindel\"{o}f subspaces). Therefore $V$ is a null neighborhood of $\mathrm{cl}_XL$ in $X$.
\end{example}

\begin{remark}
In \cite{T}, for a completely regular space $X$ and a filter base ${\mathscr B}$ of open subspaces of $X$, the author defined $C_{\mathscr B}(X)$ to be the set of all $f\in C(X)$ whose support is contained in $X\setminus B$ for some $B\in{\mathscr B}$, and $C_{\infty{\mathscr B}}(X)$ to be the set of all $f\in C(X)$ such that $|f|^{-1}([1/n,\infty))$ is contained in $X\setminus B$ for some $B\in{\mathscr B}$ for each positive integer $n$. (See \cite{AN} for certain special cases.) Also, if ${\mathfrak I}$ is an ideal of closed subspaces of $X$, in \cite{AG}, the authors defined $C_{\mathfrak I}(X)$ to be the set of all $f\in C(X)$ whose support is contained in ${\mathfrak I}$, and $C_\infty^{\mathfrak I}(X)$ to be the set of all $f\in C(X)$ such that $|f|^{-1}([1/n,\infty))$ is contained in ${\mathfrak I}$ for each positive integer $n$. Despite certain similarities between our definitions and the definitions given in \cite{T} or \cite{AG}, the existing differences between definitions have left this work with little in common with either \cite{T} or \cite{AG}.
\end{remark}

\part{Examples}\label{KHGJ}

\bigskip

In this part we study specific examples of the ideals $C^{\mathfrak I}_{00}(X)$ and $C^{\mathfrak I}_0(X)$ of $C_B(X)$ for various choices of the space $X$ and the ideal ${\mathfrak I}$ of $X$. Furthermore, we show how certain known ideals of $C_B(X)$ may be thought of as being of the form $C^{\mathfrak I}_{00}(X)$ or $C^{\mathfrak I}_0(X)$ for an appropriate choice of the ideal ${\mathfrak I}$ of $X$. Our general representation theorems in Part \ref{HFPG} may then be applied to obtain information about the structure of such ideals of $C_B(X)$.

\section{Ideals in $\ell_\infty$ arising from ideals in $\mathbb{N}$}\label{HFLH}

In this section we consider certain ideals in $\mathbb{N}$ (when $\mathbb{N}$ is endowed with the discrete topology). This leads to the introduction of certain ideals of $\ell_\infty$.

By $\ell_\infty$, $c_0$ and $c_{00}$, respectively, we denote the set of all bounded sequences in $\mathbb{C}$, the set of all vanishing sequences in $\mathbb{C}$, and the set of all sequences in $\mathbb{C}$ with only finitely many non-zero terms. Note that $\ell_\infty=C_B(\mathbb{N})$, $c_0=C_0(\mathbb{N})$ and $c_{00}=C_{00}(\mathbb{N})$, if $\mathbb{N}$ is given the discrete topology.

Let
\[{\mathfrak S}=\bigg\{A\subseteq\mathbb{N}:\sum_{n\in A}\frac{1}{n}\mbox{ converges}\bigg\}.\]
Then ${\mathfrak S}$ is an ideal in $\mathbb{N}$, called the \textit{summable ideal} in $\mathbb{N}$. A subset of $\mathbb{N}$ is called \textit{small} if it is null (with respect to ${\mathfrak S}$).

Note that there exists a family $\{A_i:i<2^\omega\}$ consisting of infinite subsets of $\mathbb{N}$ such that the intersection $A_i\cap A_j$ is finite for any distinct $i,j<2^\omega$. To see this, arrange the rational numbers into a sequence $q_1,q_2,\ldots$ and for each $i\in\mathbb{R}$ define $A_i=\{n_1,n_2,\ldots\}$ where $q_{n_1},q_{n_2},\ldots$ is a subsequence of $q_1,q_2,\ldots$ which converges to $i$. This known fact will be used in the proof of the following.

Recall that for a collection $\{X_i:i\in I\}$ of algebras the direct sum $\bigoplus_{i\in I}X_i$ is the set of all sequences $\{x_i\}_{i\in I}$ where $x_i\in X_i$ for each $i\in I$ such that $x_i=0$ for all but a finite number of indices $i\in I$. The set $\bigoplus_{i\in I}X_i$ is an algebra with addition, multiplication and scalar multiplication defined component-wise. We denote the sequence $\{x_i\}_{i\in I}$ by the sum $\sum_{i\in I}x_i$. The direct sum $\bigoplus_{i\in I}X_i$ of a collection $\{X_i:i\in I\}$ of normed spaces is defined analogously and is a normed space with the norm given by
\[\bigg\|\sum_{i\in I}x_i\bigg\|=\sup\big\{\|x_i\|_{X_i}:i\in I\big\}\]
for any $\sum_{i\in I}x_i\in\bigoplus_{i\in I}X_i$.

The proof of the following makes use of certain standard properties of the Stone--\v{C}ech compactification as stated in the concluding part of Part  \ref{JHGG}.

\begin{theorem}\label{TDK}
Let
\[\mathfrak{s}_{00}=\bigg\{\mathbf{x}\in\ell_\infty:\sum_{\mathbf{x}(n)\neq0}\frac{1}{n}\mbox{ converges}\bigg\}.\]
Then
\begin{itemize}
\item[\rm(1)] $\mathfrak{s}_{00}$ is an ideal in $\ell_\infty$.
\item[\rm(2)] $\mathfrak{s}_{00}$ is non-unital.
\item[\rm(3)] $\mathfrak{s}_{00}$ is isometrically isomorphic to $C_{00}(Y)$ where
\[Y=\bigcup\bigg\{\mathrm{cl}_{\beta\mathbb{N}}A:A\subseteq\mathbb{N}\mbox{ and }\sum_{n\in A}\frac{1}{n}\mbox{ converges}\bigg\}.\]
In particular, $Y$ is locally compact non-compact and contains $\mathbb{N}$ densely.
\item[\rm(4)] $\mathfrak{s}_{00}$ contains an isometric copy of the normed algebra $\bigoplus_{n=1}^\infty\ell_\infty$.
\item[\rm(5)] $\mathfrak{s}_{00}/c_{00}$ contains an isomorphic copy of the algebra
\[\bigoplus_{i<2^\omega}\frac{\ell_\infty}{c_{00}}.\]
\end{itemize}
\end{theorem}

\begin{proof}
Conditions (1)--(3) follow from Theorem \ref{UUS} and the following observation. Consider the ideal ${\mathfrak S}$ of $\mathbb{N}$. Note that if $\mathbf{x}\in\ell_\infty$ (since $\mathbb{N}$ is discrete) then
\[\mathrm{supp}(\mathbf{x})=\big\{n\in\mathbb{N}:\mathbf{x}(n)\neq0\big\},\]
and $\mathrm{supp}(\mathbf{x})$ is null if and only if it has a null neighborhood in $\mathbb{N}$. Thus $\mathfrak{s}_{00}=C^{\mathfrak S}_{00}(\mathbb{N})$. Note that $\mathbb{N}$ is locally null (indeed, $\{n\}$ is a null neighborhood of $n$ in $\mathbb{N}$ for each $n\in\mathbb{N}$) and that ${\mathfrak S}$ is non-proper (as $\sum 1/n$ diverges). Also, note that (since $\mathbb{N}$ is discrete) every subset $A$ of $\mathbb{N}$ is closed and open in $\mathbb{N}$ and therefore has an closed and open closure $\mathrm{cl}_{\beta\mathbb{N}}A$ in $\beta\mathbb{N}$. Therefore $Y=\lambda_{\mathfrak S}\mathbb{N}$ by Proposition \ref{HWA}.

(4). By (3), we may consider $C_{00}(Y)$ in place of $\mathfrak{s}_{00}$. Let $A$ be an infinite subset of $\mathbb{N}$ such that $\sum_{n\in A}1/n$ converges (which exists, for example, let $A=\{2^n:n\in \mathbb{N}\}$). Let $A_1, A_2,\dots$ be a partition of $A$ into pairwise disjoint infinite subsets. Let $n$ be a positive integer. We may assume that $C(\mathrm{cl}_{\beta\mathbb{N}}A_n)$ is a subalgebra of $C_{00}(Y)$. (Since $A_n$ is closed and open in $\mathbb{N}$ it has a closed and open closure $\mathrm{cl}_{\beta\mathbb{N}}A_n$ in $\beta\mathbb{N}$. Thus each element of $C(\mathrm{cl}_{\beta\mathbb{N}}A_n)$ may be continuously extended over $Y$ by defining it to be identically $0$ elsewhere.) Note that $\mathrm{cl}_{\beta\mathbb{N}}A_i$ and $\mathrm{cl}_{\beta\mathbb{N}}A_j$ are disjoint for any distinct positive integers $i$ and $j$, as $A_i$ and $A_j$ are disjoint closed and open subspaces (and thus zero-sets) of $\mathbb{N}$. Thus, the inclusion mapping
\[\iota:\bigoplus_{n=1}^\infty C(\mathrm{cl}_{\beta\mathbb{N}}A_n)\longrightarrow C_{00}(Y)\]
is an algebra isomorphism (onto its image) and it preserves norms. (For the latter, use the fact that $\mathrm{cl}_{\beta\mathbb{N}}A_i$'s are disjoint for distinct indices.) Note that $\mathrm{cl}_{\beta\mathbb{N}}A_n$ coincides with $\beta A_n$, as $A_n$ is closed in the normal space $\mathbb{N}$. Finally, observe that
\[C(\mathrm{cl}_{\beta\mathbb{N}}A_n)=C(\beta A_n)=C(\beta\mathbb{N})=C_B(\mathbb{N})=\ell_\infty.\]

(5). By (3), we may consider $C_{00}(Y)$ in place of $\mathfrak{s}_{00}$. Let $A$ be an infinite subset of $\mathbb{N}$ such that $\sum_{n\in A}1/n$ converges. Consider a family $\{A_i:i<2^\omega\}$ consisting of infinite subsets of $A$ such that $A_i\cap A_j$ is finite for any distinct $i,j<2^\omega$. Let
\[H=\big\{f\in C_{00}(Y):\mathrm{supp}(f)\subseteq A\big\}\]
and let
\[H_i=\big\{f\in C(\mathrm{cl}_{\beta\mathbb{N}}A_i):\mathrm{supp}(f)\subseteq A_i\big\}\]
for each $i<2^\omega$. As in (4), we may assume that $C(\mathrm{cl}_{\beta\mathbb{N}}A_i)$ is a subalgebra of $C_{00}(Y)$ for each $i<2^\omega$, and thus, we may assume that $H_i\subseteq H$. Note that if $f\in H$ then $\mathrm{supp}(f)$ is finite, as it is a compact subspace of $\mathbb{N}$.

Define a mapping
\[\Theta:\bigoplus_{i<2^\omega}\frac{C(\mathrm{cl}_{\beta\mathbb{N}}A_i)}{H_i}\longrightarrow\frac{C_{00}(Y)}{H}\]
by
\[\sum_{i<2^\omega}(f_i+H_i)\longmapsto\sum_{i<2^\omega}f_i+H\]
where $f_i\in C(\mathrm{cl}_{\beta\mathbb{N}}A_i)$ for each $i<2^\omega$. We show that $\Theta$ is an isometric isomorphism onto its image; since
\[\frac{C(\mathrm{cl}_{\beta\mathbb{N}}A_i)}{H_i}=\frac{\ell_\infty}{c_{00}}\]
for each $i<2^\omega$, this completes the proof.

First, note that $\Theta$ is well defined; to show this, let
\[\sum_{i<2^\omega}(f_i+H_i)=\sum_{i<2^\omega}(g_i+H_i)\]
where $f_i, g_i\in C(\mathrm{cl}_{\beta\mathbb{N}}A_i)$ for each $i<2^\omega$. For each $i<2^\omega$ then $f_i+H_i=g_i+H_i$, or equivalently $f_i-g_i\in H_i$; in particular $f_i-g_i\in H$. Thus
\[\sum_{i<2^\omega}(f_i-g_i)\in H\]
and therefore
\[\Theta\bigg(\sum_{i<2^\omega}(f_i+H_i)\bigg)=\sum_{i<2^\omega}f_i+H=\sum_{i<2^\omega}g_i+H=\Theta\bigg(\sum_{i<2^\omega}(g_i+H_i)\bigg).\]

Now, we show that $\Theta$ preserves product. Let $f_i,g_i\in C(\mathrm{cl}_{\beta\mathbb{N}}A_i)$ for each $i<2^\omega$. Then
\[\Theta\bigg(\sum_{i<2^\omega}(f_i+H_i)\cdot\sum_{i<2^\omega}(g_i+H_i)\bigg)=\Theta\bigg(\sum_{i<2^\omega}(f_ig_i+H_i)\bigg)
=\sum_{i<2^\omega}f_ig_i+H.\]
Note that if $k,l<2^\omega$ with $k\neq l$ then
\begin{eqnarray*}
\mathrm{supp}(f_kg_l)&\subseteq&\mathrm{supp}(f_k)\cap\mathrm{supp}(g_l)\\&\subseteq&\mathrm{cl}_{\beta\mathbb{N}}A_k\cap\mathrm{cl}_{\beta\mathbb{N}}A_l=
\mathrm{cl}_{\beta\mathbb{N}}(A_k\cap A_l)=A_k\cap A_l\subseteq A
\end{eqnarray*}
and thus $f_kg_l\in H$. We have
\begin{eqnarray*}
&&\Theta\bigg(\sum_{i<2^\omega}(f_i+H_i)\bigg)\cdot\Theta\bigg(\sum_{i<2^\omega}(g_i+H_i)\bigg)\\&=&\bigg(\sum_{i<2^\omega}f_i+H\bigg)\cdot\bigg(\sum_{i<2^\omega}g_i+H\bigg)
\\&=&\bigg(\sum_{i<2^\omega}f_i\sum_{i<2^\omega}g_i\bigg)+H\\&=&\bigg(\sum_{i<2^\omega}f_ig_i+\sum_{k\neq l}f_kg_l\bigg)+H\\&=&\sum_{i<2^\omega}f_ig_i+H.
\end{eqnarray*}
This together with the above relations proves that
\[\Theta\bigg(\sum_{i<2^\omega}(f_i+H_i)\cdot\sum_{i<2^\omega}(g_i+H_i)\bigg)=\Theta\bigg(\sum_{i<2^\omega}(f_i+H_i)\bigg)\cdot\Theta\bigg(\sum_{i<2^\omega}(g_i+H_i)\bigg).\]
That $\Theta$ preserves addition and scalar multiplication follows analogously.

Next, we show that $\Theta$ is injective. Let
\[\Theta\bigg(\sum_{i<2^\omega}(f_i+H_i)\bigg)=0\]
where $f_i\in C(\mathrm{cl}_{\beta\mathbb{N}}A_i)$ for each $i<2^\omega$. Then
\[\sum_{i<2^\omega}f_i+H=0,\]
or, equivalently
\[h=\sum_{i<2^\omega}f_i\in H.\]
Suppose that $f_{i_j}$, where $j=1,\ldots,n$, are the possibly non-zero terms. Fix some $k=1,\ldots,n$. Then
\[f_{i_k}=h-\sum_{1\leq j\neq k\leq n}f_{i_j}.\]
We have
\[\mathrm{coz}(f_{i_k})\subseteq\mathrm{coz}(h)\cup\bigcup_{1\leq j\neq k\leq n}\mathrm{coz}(f_{i_j})\subseteq\mathrm{supp}(h)\cup\bigcup_{1\leq j\neq k\leq n}\mathrm{cl}_{\beta\mathbb{N}}A_{i_j}.\]
Note that the latter set is compact, as $\mathrm{supp}(h)$ is finite, since $h\in H$. Therefore
\begin{equation}\label{PJUF}
\mathrm{supp}(f_{i_k})\subseteq\mathrm{supp}(h)\cup\bigcup_{1\leq j\neq k\leq n}\mathrm{cl}_{\beta\mathbb{N}}A_{i_j}.
\end{equation}
Intersecting both sides of (\ref{PJUF}) with $\mathrm{cl}_{\beta\mathbb{N}}A_{i_k}$ it yields
\begin{eqnarray*}
\mathrm{supp}(f_{i_k})&\subseteq&\mathrm{supp}(h)\cup\bigcup_{1\leq j\neq k\leq n}(\mathrm{cl}_{\beta\mathbb{N}}A_{i_j}\cap\mathrm{cl}_{\beta\mathbb{N}}A_{i_k})\\&=&\mathrm{supp}(h)\cup\bigcup_{1\leq j\neq k\leq n}\mathrm{cl}_{\beta\mathbb{N}}(A_{i_j}\cap A_{i_k})=\mathrm{supp}(h)\cup\bigcup_{1\leq j\neq k\leq n}(A_{i_j}\cap A_{i_k})
\end{eqnarray*}
with the latter being a subset of $\mathbb{N}$. Thus
\[\mathrm{supp}(f_{i_k})\subseteq\mathbb{N}\cap\mathrm{cl}_{\beta\mathbb{N}} A_{i_k}=A_{i_k}\]
and therefore $f_{i_k}\in H_{i_k}$. This implies that
\[\sum_{j=1}^n(f_{i_j}+H_{i_j})=0.\]
Thus $\Theta$ is injective.
\end{proof}

\begin{theorem}\label{JHF}
Let
\[\mathfrak{s}_0=\bigg\{\mathbf{x}\in\ell_\infty:\sum_{|\mathbf{x}(n)|\geq\epsilon}\frac{1}{n}\mbox{ converges for each }\epsilon>0\bigg\}.\]
Then
\begin{itemize}
\item[\rm(1)] $\mathfrak{s}_0$ is a closed ideal in $\ell_\infty$.
\item[\rm(2)] $\mathfrak{s}_0$ is non-unital.
\item[\rm(3)] $\mathfrak{s}_0$ contains $\mathfrak{s}_{00}$ densely.
\item[\rm(4)] $\mathfrak{s}_0$ is isometrically isomorphic to $C_0(Y)$ where
\[Y=\bigcup\bigg\{\mathrm{cl}_{\beta\mathbb{N}}A:A\subseteq\mathbb{N}\mbox{ and }\sum_{n\in A}\frac{1}{n}\mbox{ converges}\bigg\}.\]
In particular, $Y$ is the spectrum of $\mathfrak{s}_0$, is locally compact non-compact and contains $\mathbb{N}$ densely.
\item[\rm(5)] $\mathfrak{s}_0/c_0$ contains a copy of the normed algebra
\[\bigoplus_{i<2^\omega}\frac{\ell_\infty}{c_0}.\]
\end{itemize}
\end{theorem}

\begin{proof}
The proofs for (1)--(4) are analogous to the proofs for the corresponding parts in Theorem \ref{TDK}, making use of Lemma \ref{DGH}. Note that if $\mathbf{x}\in\ell_\infty$ and $\epsilon>0$ then $|\mathbf{x}|^{-1}([\epsilon,\infty))$ is null (with respect to ${\mathfrak S}$) if and only if
\[\sum_{|\mathbf{x}(n)|\geq\epsilon}\frac{1}{n}\]
converges. Thus, in particular $\mathfrak{s}_0=C^{\mathfrak S}_0(\mathbb{N})$.

(5). The proof of this part is analogous to the proof of the corresponding part in Theorem \ref{TDK}; we will highlight only the differences.

Let $A$ and $\{A_i:i<2^\omega\}$ be as chosen in the proof of Theorem \ref{TDK}. Let
\[H=\big\{f\in C_0(Y):|f|^{-1}\big([\epsilon,\infty)\big)\subseteq A\mbox{ for each }\epsilon>0\big\}\]
and let
\[H_i=\big\{f\in C(\mathrm{cl}_{\beta\mathbb{N}}A_i):|f|^{-1}\big([\epsilon,\infty)\big)\subseteq A_i\mbox{ for each }\epsilon>0\big\}\]
for each $i<2^\omega$. We consider $C(\mathrm{cl}_{\beta\mathbb{N}}A_i)$ as a normed subalgebra of $C_0(Y)$, and thus, we may assume that $H_i\subseteq H$ for each $i<2^\omega$. Note that if $i<2^\omega$ and $f\in H_i$ then $|f|^{-1}([\epsilon,\infty))$ is a finite subset of $A_i$, as it is compact (since it is  closed in $\mathrm{cl}_{\beta\mathbb{N}}A_i$). Define a mapping
\[\Theta:\bigoplus_{i<2^\omega}\frac{C(\mathrm{cl}_{\beta\mathbb{N}}A_i)}{H_i}\longrightarrow\frac{C_0(Y)}{H}\]
by
\[\sum_{i<2^\omega}(f_i+H_i)\mapsto\sum_{i<2^\omega}f_i+H\]
where $f_i\in C(\mathrm{cl}_{\beta\mathbb{N}}A_i)$ for each $i<2^\omega$. Arguing as in the proof of Theorem \ref{TDK} it follows that $\Theta$ is well defined. The proofs that $\Theta$ preserves product is analogous to the corresponding part in the proof of Theorem \ref{TDK}; simply observe that if $f_i,g_i\in C(\mathrm{cl}_{\beta\mathbb{N}}A_i)$ for each $i<2^\omega$ and $\epsilon>0$, then for any $k,l<2^\omega$ with $k\neq l$ we have
\begin{eqnarray*}
|f_kg_l|^{-1}\big([\epsilon,\infty)\big)&\subseteq&\mathrm{coz}(f_kg_l)\\&\subseteq&\mathrm{coz}(f_k)\cap\mathrm{coz}(g_l)
\\&\subseteq&\mathrm{cl}_{\beta\mathbb{N}}A_k\cap\mathrm{cl}_{\beta\mathbb{N}}A_l=\mathrm{cl}_{\beta\mathbb{N}}(A_k\cap A_l)=A_k\cap A_l\subseteq A
\end{eqnarray*}
and thus $f_kg_l\in H$. The proofs that $\Theta$ preserves addition and scalar multiplication are analogous.

Now, we show that $\Theta$ is injective. Let
\[\Theta\bigg(\sum_{i<2^\omega}(f_i+H_i)\bigg)=0\]
where $f_i\in C(\mathrm{cl}_{\beta\mathbb{N}}A_i)$ for each $i<2^\omega$. Suppose that $f_{i_j}$, where $j=1,\ldots,n$, are the possibly non-zero terms. Fix some $k=1,\ldots,n$. Then, as in the proof of Theorem \ref{TDK} we have
\[f_{i_k}=h-\sum_{1\leq j\neq k\leq n}f_{i_j}\]
for some $h\in H$. Let $\epsilon>0$. Then
\begin{eqnarray*}
|f_{i_k}|^{-1}\big([\epsilon,\infty)\big)&\subseteq&|h|^{-1}\big([\epsilon/n,\infty)\big)\cup\bigcup_{1\leq j\neq k\leq n}|f_{i_j}|^{-1}\big([\epsilon/n,\infty)\big)\\&\subseteq& A\cup\bigcup_{1\leq j\neq k\leq n}\mathrm{cl}_{\beta\mathbb{N}}A_{i_j}.
\end{eqnarray*}
Arguing as in the proof of Theorem \ref{TDK}, intersecting both sides of the above relation with $\mathrm{cl}_{\beta\mathbb{N}}A_{i_k}$ it yields
\begin{eqnarray*}
|f_{i_k}|^{-1}\big([\epsilon,\infty)\big)\subseteq A_{i_k}
\end{eqnarray*}
and therefore $f_{i_k}\in H_{i_k}$. This implies that
\[\sum_{j=1}^n(f_{i_j}+H_{i_j})=0.\]
Thus $\Theta$ is injective.

Next, we show that $\Theta$ is an isometry. First, we need to show the following.

\begin{xclaim}
Let $i<2^\omega$ and $f\in C(\mathrm{cl}_{\beta\mathbb{N}}A_i)$. Then
\[\|f+H_i\|=\big\|f|_{Y\setminus\mathbb{N}}\big\|_\infty.\]
\end{xclaim}

\noindent \emph{Proof of the claim.} Suppose to the contrary that
\begin{equation}\label{FSD}
\|f+H_i\|<\big\|f|_{Y\setminus\mathbb{N}}\big\|_\infty.
\end{equation}
Then
\[\alpha=\|f+h\|_\infty<\big\|f|_{Y\setminus\mathbb{N}}\big\|_\infty\]
for some $h\in H_i$. Let $\alpha<\gamma<\|f|_{Y\setminus\mathbb{N}}\|_\infty$. Then
\[B=|f|^{-1}\big([\gamma,\infty)\big)\subseteq|h|^{-1}\big([\gamma-\alpha,\infty)\big)=C;\]
as if $y\in Y$ such that $|f(y)|\geq\gamma$, then
\[\alpha\geq\big|f(y)+h(y)\big|\geq\big|f(y)\big|-\big|h(y)\big|\geq\gamma-\big|h(y)\big|\]
and thus $|h(y)|\geq\gamma-\alpha$. Note that $B$ is a finite subset of $\mathbb{N}$, as $C$ is so, since $h\in H_i$. We have
\[Y=\mathrm{cl}_Y \mathbb{N}=B\cup\mathrm{cl}_Y (\mathbb{N}\setminus B)\]
and thus
\begin{eqnarray*}
|f|(Y\setminus\mathbb{N})\subseteq|f|\big(\mathrm{cl}_Y (\mathbb{N}\setminus B)\big)\subseteq\overline{|f|(\mathbb{N}\setminus B)}&=&\overline{|f|\big(\mathbb{N}\cap|f|^{-1}\big([0,\gamma)\big)\big)}\\&\subseteq&\overline{|f|\big(|f|^{-1}\big([0,\gamma)\big)\big)}
\subseteq\overline{[0,\gamma)}=[0,\gamma],
\end{eqnarray*}
where the bar denotes the closure in $\mathbb{R}$. This implies that $\|f|_{Y\setminus\mathbb{N}}\|_\infty\leq\gamma$, which contradicts the choice of $\gamma$. Thus (\ref{FSD}) is false, that is
\begin{equation}\label{GD}
\|f+H_i\|\geq\big\|f|_{Y\setminus\mathbb{N}}\big\|_\infty.
\end{equation}

Now, we prove the reverse inequality in (\ref{GD}). Let $\delta=\|f|_{Y\setminus\mathbb{N}}\|_\infty$. We first show that
\[D_\epsilon=\mathbb{N}\cap|f|^{-1}\big((\delta+\epsilon,\infty)\big)\]
is finite for every $\epsilon>0$. Suppose the contrary, that is, suppose that $D_\epsilon$ is infinite for some $\epsilon>0$. Note that $\mathrm{cl}_{\beta\mathbb{N}}D_\epsilon\setminus\mathbb{N}$ is non-empty, as $\mathrm{cl}_{\beta\mathbb{N}}D_\epsilon\subseteq\mathbb{N}$ implies that $D_\epsilon=\mathrm{cl}_{\beta\mathbb{N}}D_\epsilon$ is compact and is then finite. Let $p\in\mathrm{cl}_{\beta\mathbb{N}}D_\epsilon\setminus\mathbb{N}$. Note that
\[\mathrm{cl}_{\beta\mathbb{N}}D_\epsilon=\mathrm{cl}_{\beta\mathbb{N}}|f|^{-1}\big((\delta+\epsilon,\infty)\big).\]
We have
\begin{eqnarray*}
|f|(p)\in|f|(\mathrm{cl}_{\beta\mathbb{N}}D_\epsilon)&=&|f|\big(\mathrm{cl}_{\beta\mathbb{N}}|f|^{-1}\big((\delta+\epsilon,\infty)\big)\big)
\\&\subseteq&\overline{|f|\big(|f|^{-1}\big((\delta+\epsilon,\infty)\big)\big)}\subseteq\overline{(\delta+\epsilon,\infty)}=[\delta+\epsilon,\infty),
\end{eqnarray*}
where the bar denotes the closure in $\mathbb{R}$. Therefore
\[\delta=\big\|f|_{Y\setminus\mathbb{N}}\big\|_\infty\geq\big|f(p)\big|\geq\delta+\epsilon,\]
which is not possible. This shows that $D_\epsilon$ is finite for every $\epsilon>0$. Now, let $\epsilon>0$. Define a mapping  $h:Y\rightarrow\mathbb{C}$ such that $h(x)=-f(x)$ if $x\in D_\epsilon$ and $h(x)=0$ otherwise. Note that $D_\epsilon$ is closed in $Y$, as it is finite, and $D_\epsilon$ is open in $Y$, as it is open in $\mathbb{N}$ and $\mathbb{N}$ is open in $\beta\mathbb{N}$ (and thus in $Y$), since $\mathbb{N}$ is locally compact. Therefore $h$ is continuous. Observe that
\[\mathrm{coz}(h)\subseteq D_\epsilon\subseteq\mathbb{N}\cap|f|^{-1}\big([\delta+\epsilon,\infty)\big)\subseteq\mathbb{N}\cap\mathrm{cl}_{\beta\mathbb{N}}A_i=A_i.\]
Thus $h\in H_i$. Note that $f+h=\mathbf{0}$ on $D_\epsilon$ and $f+h= f$ on $Y\setminus D_\epsilon$. Also,
\[\big\|f|_{\mathbb{N}\setminus D_\epsilon}\big\|_\infty\leq\delta+\epsilon\]
by the way we have defined $D_\epsilon$. Now
\[\|f+H_i\|\leq\|f+h\|_\infty=\big\|f|_{Y\setminus D_\epsilon}\big\|_\infty=\max\big\{\big\|f|_{\mathbb{N}\setminus D_\epsilon}\big\|_\infty, \big\|f|_{Y\setminus \mathbb{N}}\big\|_\infty\big\}\leq\delta+\epsilon.\]
Since $\epsilon>0$ is arbitrary, it follows that
\[\|f+H_i\|\leq\delta=\big\|f|_{Y\setminus\mathbb{N}}\big\|_\infty.\]
This together with (\ref{GD}) proves the claim.

\begin{xclaim}
Let $i_1,\ldots,i_n<2^\omega$ and $f_{i_j}\in C(\mathrm{cl}_{\beta\mathbb{N}}A_{i_j})$ for each $j=1,\ldots,n$. Then
\[\bigg\|\sum_{j=1}^nf_{i_j}+H\bigg\|=\bigg\|\bigg(\sum_{j=1}^nf_{i_j}\bigg)\bigg|_{Y\setminus\mathbb{N}}\bigg\|_\infty.\]
\end{xclaim}

\noindent \emph{Proof of the claim.}
This follows by an argument similar to the one we have given in the first claim; one simply needs to replace $f$ by $\sum_{j=1}^nf_{i_j}$, $H_i$ by $H$, and $A_i$ by $A$ throughout.

\medskip

Now, let $i_1,\ldots,i_n<2^\omega$ and $f_{i_j}\in C(\mathrm{cl}_{\beta\mathbb{N}}A_{i_j})$ for each $j=1,\ldots,n$. For simplicity of the notation let
\[f=\sum_{j=1}^nf_{i_j}.\]
Let $1\leq j\neq k\leq n$. Then $A_j\cap A_k$ is finite and thus
\[\mathrm{coz}(f_{i_k})\cap\mathrm{coz}(f_{i_l})\subseteq\mathrm{cl}_{\beta\mathbb{N}}A_{i_k}\cap\mathrm{cl}_{\beta\mathbb{N}}A_{i_l}=\mathrm{cl}_{\beta\mathbb{N}}(A_{i_k}\cap A_{i_l})=A_{i_k}\cap A_{i_l}\subseteq\mathbb{N}.\]
In particular,
\[\mathrm{coz}\big(f_{i_k}|_{Y\setminus\mathbb{N}}\big)\cap\mathrm{coz}\big(f_{i_l}|_{Y\setminus\mathbb{N}}\big)=\emptyset.\]
Now, by the second claim
\[\bigg\|\Theta\bigg(\sum_{j=1}^n(f_{i_j}+H_{i_j})\bigg)\bigg\|=\|f+H\|=\big\|f|_{Y\setminus\mathbb{N}}\big\|_\infty\]
and by the first claim
\begin{eqnarray*}
\big\|f|_{Y\setminus\mathbb{N}}\big\|_\infty&=&\max\big\{\big\|f|_{(\mathrm{cl}_{\beta\mathbb{N}}A_{i_j}\setminus\mathbb{N})}\big\|_\infty:j=1,\ldots,n\big\}
\\&=&\max\big\{\big\|f_{i_j}|_{Y\setminus\mathbb{N}}\big\|_\infty:j=1,\ldots,n\big\}
\\&=&\max\big\{\|f_{i_j}+H_{i_j}\|:j=1,\ldots,n\big\}=\bigg\|\sum_{j=1}^n(f_{i_j}+H_{i_j})\bigg\|.
\end{eqnarray*}
That is, $\Theta$ is an isometry.
\end{proof}

\begin{remark}
Any sequence $f:\mathbb{N}\rightarrow(0,\infty)$ such that $\sum_{n=1}^\infty f(n)$ diverges, determines an ideal
\[{\mathfrak S}_f=\bigg\{A\subseteq\mathbb{N}:\sum_{n\in A}f(n)\mbox{ converges}\bigg\}\]
in $\mathbb{N}$. This provides a more general setting to state and prove Theorems \ref{TDK} and \ref{JHF}.
\end{remark}

Let
\[{\mathfrak D}=\bigg\{A\subseteq\mathbb{N}:\limsup_{n\rightarrow\infty}\frac{|A\cap\{1,\ldots,n\}|}{n}=0\bigg\}.\]
Then ${\mathfrak D}$ also is an important ideal in $\mathbb{N}$, called the \textit{density ideal} in $\mathbb{N}$. In other words, ${\mathfrak D}$ consists of those subsets $D$ of $\mathbb{N}$ such that $D$ has an \textit{asymptotic density} zero. Note that every small set in $\mathbb{N}$ has an asymptotic density zero, that is ${\mathfrak S}\subseteq{\mathfrak D}$; the converse, however, does not hold in general. The set of all prime numbers has an asymptotic density zero but it is not small. (For more information on the subject, see \cite{Hr}.)

The following is dual to Theorems \ref{TDK} and \ref{JHF} and may be proved analogously, by replacing the ideal ${\mathfrak S}$ by the ideal ${\mathfrak D}$ throughout all proofs.

\begin{theorem}\label{DFDS}
Let
\[\mathfrak{d}_{00}=\bigg\{\mathbf{x}\in\ell_\infty:\limsup_{n\rightarrow\infty}\frac{|\{k\leq n:\mathbf{x}(k)\neq0\}|}{n}=0\bigg\}\]
and
\[\mathfrak{d}_0=\bigg\{\mathbf{x}\in\ell_\infty:\limsup_{n\rightarrow\infty}\frac{|\{k\leq n:|\mathbf{x}(k)|\geq\epsilon\}|}{n}=0\mbox{ for each }\epsilon>0\bigg\}.\]
Then
\begin{itemize}
\item[\rm(1)] $\mathfrak{d}_{00}$ is an ideal in $\ell_\infty$ and $\mathfrak{d}_0$ is a closed ideal in $\ell_\infty$.
\item[\rm(2)] $\mathfrak{d}_{00}$ is dense in $\mathfrak{d}_0$.
\item[\rm(3)] Neither $\mathfrak{d}_{00}$ nor $\mathfrak{d}_0$ is unital.
\item[\rm(4)] $\mathfrak{d}_{00}$ and $\mathfrak{d}_0$ are isometrically isomorphic to $C_{00}(Y)$ and $C_0(Y)$, respectively, where
\[Y=\bigcup\bigg\{\mathrm{cl}_{\beta\mathbb{N}}A:A\subseteq\mathbb{N}\mbox{ and }\limsup_{n\rightarrow\infty}\frac{|A\cap\{1,\ldots,n\}|}{n}=0\bigg\}.\]
In particular, $Y$ is the spectrum of $\mathfrak{d}_0$, is locally compact non-compact and contains $\mathbb{N}$ densely.
\item[\rm(5)] $\mathfrak{d}_{00}$ contains an isometric copy of the normed algebra $\bigoplus_{n=1}^\infty\ell_\infty$.
\item[\rm(6)] $\mathfrak{d}_{00}/c_{00}$ contains an isomorphic copy of the algebra
\[\bigoplus_{i<2^\omega}\frac{\ell_\infty}{c_{00}}.\]
\item[\rm(7)] $\mathfrak{d}_0/c_0$ contains an isometric copy of the normed algebra
\[\bigoplus_{i<2^\omega}\frac{\ell_\infty}{c_0}.\]
\end{itemize}
\end{theorem}

\section{The structure of non-vanishing closed ideals in $C_B(X)$}\label{KHDS}

In this section, for a completely regular space $X$, we consider non-vanishing closed ideals of $C_B(X)$. (Recall that an ideal $H$ of $C_B(X)$ is called \textit{non-vanishing} if for each $x\in H$ there is some $h\in H$ such that $h(x)\neq 0$.) We first show how a non-vanishing closed ideal $H$ of $C_B(X)$ can be thought of as being $C^{\mathfrak I}_0(X)$ for an appropriate choice of an ideal ${\mathfrak I}$ of $X$. Our representation theorems in Part \ref{HFPG} will then enable us to determine the spectrum of $H$ (denoted by $\mathfrak{sp}(H)$) as a subspace of $\beta X$. We will see how the known construction of $\mathfrak{sp}(H)$ allows a better study of its properties.

We start with the following definition.

\begin{definition}\label{HJGHA}
Let $X$ be a completely regular space and let $H$ be an ideal in $C_B(X)$. Define
\[\lambda_H X=\bigcup_{h\in H}\mathrm{cl}_{\beta X}|h|^{-1}\big((1,\infty)\big),\]
considered as a subspace of $\beta X$.
\end{definition}

As we will see the space $\lambda_H X$ defined in the above definition is nothing but $\lambda_{\mathfrak H} X$ for an appropriate choice of the ideal ${\mathfrak H}$ of $X$.

\begin{lemma}\label{JJHF}
Let $X$ be a completely regular space and let $f\in C_B(X)$. Let $f_1,f_2,\ldots$ be a sequence in $C_B(X)$ such that
\[|f|^{-1}\big([1/n,\infty)\big)\subseteq |f_n|^{-1}\big([1,\infty)\big)\]
for any positive integer $n$. Then, there is a sequence $g_1,g_2,\ldots$ in $C_B(X)$ with $g_nf_n\rightarrow f$.
\end{lemma}

\begin{proof}
Fix some positive integer $n$. We define a mapping $u_n:X\rightarrow\mathbb{C}$ by
\[u_n(x)=\left\{
             \begin{array}{ll}
               1/f_n(x), & \hbox{ if }x\in |f_n|^{-1}([1,\infty));  \\
               \overline{f_n(x)}, & \hbox{ if }x\in |f_n|^{-1}([0,1]).
             \end{array}
         \right.\]
The mapping $u_n$ is well defined, as $1/f_n(x)=\overline{f_n(x)}$ for any $x$ in the intersection
\[|f_n|^{-1}\big([1,\infty)\big)\cap |f_n|^{-1}\big([0,1]\big)=|f_n|^{-1}(1),\]
and $u_n$ is continuous, as it is continuous on each of the two closed subspaces $|f_n|^{-1}([1,\infty))$ and $|f_n|^{-1}([0,1])$ of $X$ whose union is the entire $X$. Note that $|u_n(x)|\leq 1$ for each $x\in X$. In particular, $u_n\in C_B(X)$. We now verify that
\begin{equation}\label{GFFF}
\big|u_nf_nf(x)-f(x)\big|<1/n
\end{equation}
for each $x\in X$. Let $x\in X$. We consider the following two cases:
\begin{description}
  \item[Case 1.] Suppose that $x\in|f_n|^{-1}([1,\infty))$. Then $u_n(x)f_n(x)=1$ by the definition of $u_n$. Therefore $u_nf_nf(x)-f(x)=0$, and thus (\ref{GFFF}) holds trivially in this case.

  \item[Case 2.] Suppose that $x\in|f_n|^{-1}([0,1))$. Then $f_n(x)=\overline{u_n(x)}$ by the definition of $u_n$. Therefore \[u_n(x)f_n(x)f(x)-f(x)=\big|u_n(x)\big|^2f(x)-f(x)=\big[\big|u_n(x)\big|^2-1\big]f(x).\]
      But $|u_n(x)|\leq 1$ and $|f(x)|<1/n$, as using our assumption
      \[|f_n|^{-1}\big([0,1)\big)\subseteq|f|^{-1}\big([0,1/n)\big).\]
      Therefore (\ref{GFFF}) holds in this case as well.
\end{description}
By (\ref{GFFF}) it follows that $\|u_nf_nf-f\|\leq 1/n$ for each positive integer $n$, and consequently $\|u_nf_nf-f\|\rightarrow 0$. For each positive integer $n$ let $g_n=u_nf$. Then $g_1,g_2,\ldots$ is a sequence in $C_B(X)$ such that $g_nf_n\rightarrow f$.
\end{proof}

\begin{lemma}\label{KLGS}
Let $X$ be a completely regular space and let $H$ be an ideal in $C_B(X)$. Then
\[\lambda_H X=\bigcup_{h\in H}\mathrm{int}_{\beta X}\mathrm{cl}_{\beta X}|h|^{-1}\big((1,\infty)\big).\]
\end{lemma}

\begin{proof}
Clearly, it suffices to show that
\[\lambda_H X\subseteq\bigcup_{h\in H}\mathrm{int}_{\beta X}\mathrm{cl}_{\beta X}|h|^{-1}\big((1,\infty)\big).\]
Let $h\in H$. Then
\[\mathrm{cl}_{\beta X}|h|^{-1}\big((1,\infty)\big)\subseteq|h_\beta|^{-1}\big([1,\infty)\big)\subseteq|h_\beta|^{-1}\big((\epsilon,\infty)\big)\]
where $0<\epsilon<1$. But
\[|h_\beta|^{-1}\big((\epsilon,\infty)\big)\subseteq\mathrm{int}_{\beta X}\mathrm{cl}_{\beta X}|h|^{-1}\big((\epsilon,\infty)\big)=\mathrm{int}_{\beta X}\mathrm{cl}_{\beta X}|h/\epsilon|^{-1}\big((1,\infty)\big)\]
by (a slight modification of) Lemma \ref{LKG}.
\end{proof}

\begin{lemma}\label{GFS}
Let $X$ be a completely regular space and let $H$ be an ideal in $C_B(X)$. Let $K$ be a compact subspace of $\lambda_H X$. Then
\[K\subseteq\mathrm{int}_{\beta X}\mathrm{cl}_{\beta X}h^{-1}\big((1,\infty)\big)\]
for some non-negative $h\in H$.
\end{lemma}

\begin{proof}
Using Lemma \ref{KLGS} and compactness of $K$ we have
\begin{equation}\label{JJHJS}
K\subseteq\bigcup_{i=1}^j\mathrm{int}_{\beta X}\mathrm{cl}_{\beta X}|h_i|^{-1}\big((1,\infty)\big)
\end{equation}
where $h_i\in H$ for each $i=1,\ldots,j$. Let
\[h=\sum_{i=1}^j|h_i|^2.\]
Then $h=\sum_{i=1}^j h_i\overline{h_i}\in H$, as $H$ is an ideal in $C_B(X)$. We have
\[\bigcup_{i=1}^j|h_i|^{-1}\big((1,\infty)\big)\subseteq h^{-1}\big((1,\infty)\big).\]
In particular,
\[\bigcup_{i=1}^j\mathrm{int}_{\beta X}\mathrm{cl}_{\beta X}|h_i|^{-1}\big((1,\infty)\big)\subseteq\mathrm{int}_{\beta X}\mathrm{cl}_{\beta X}\Big[\bigcup_{i=1}^j|h_i|^{-1}\big((1,\infty)\big)\Big]\subseteq\mathrm{int}_{\beta X}\mathrm{cl}_{\beta X}h^{-1}\big((1,\infty)\big).\]
This together with (\ref{JJHJS}) implies that
\[K\subseteq\mathrm{int}_{\beta X}\mathrm{cl}_{\beta X}h^{-1}\big((1,\infty)\big).\]
\end{proof}

The next theorem determines the spectrum of a non-vanishing closed ideal of $C_B(X)$ as a subspace of $\beta X$.

\begin{theorem}\label{TRES}
Let $X$ be a completely regular space and let $H$ be a non-vanishing closed ideal in $C_B(X)$. Then $H$ is isometrically isomorphic to $C_0(Y)$ for some unique locally compact space $Y$, namely, for
\[Y=\bigcup_{h\in H}\mathrm{cl}_{\beta X}|h|^{-1}\big((1,\infty)\big).\]
In particular, $Y$ is the spectrum of $H$. Furthermore, $Y$ contains $X$ as a dense subspace.
\end{theorem}

\begin{proof}
Let
\[\mathfrak{H}=\big\langle |h|^{-1}\big((1,\infty)\big):h\in H\big\rangle,\]
that is, the ideal generated by the sets $|h|^{-1}((1,\infty))$ for all $h\in H$. Observe that $X$ is locally null (with respect to $\mathfrak{H}$), as $H$ is non-vanishing. Then $C^\mathfrak{H}_0(X)$ is a closed ideal in $C_B(X)$ which is  isometrically isomorphic to $C_0(\lambda_\mathfrak{H}X)$ by Theorem \ref{UDR}. Furthermore, $\lambda_\mathfrak{H}X$ is a locally compact space which is unique up to homeomorphism, contains $X$ as a dense subspace and coincides with the spectrum of $C^\mathfrak{H}_0(X)$ (again by Theorem \ref{UDR}). Thus, to prove our theorem, it suffices to show that $H=C^\mathfrak{H}_0(X)$ and $Y=\lambda_\mathfrak{H}X$ (with $Y$ as given in the statement of the theorem).

First, we check that
\begin{equation}\label{HJGF}
\mathfrak{H}=\big\{A:A\subseteq |h|^{-1}\big((1,\infty)\big)\mbox{ for some } h\in H\big\}.
\end{equation}
It is clear that $|h|^{-1}((1,\infty))$, where $h\in H$, is null, and thus so is any of its subsets. Let $A\in\mathfrak{H}$. Then
\begin{equation}\label{LJD}
A\subseteq\bigcup_{i=1}^k|h_i|^{-1}\big((1,\infty)\big)
\end{equation}
where $h_i\in H$ for each $i=1,\ldots,k$. Let
\[h=\sum_{i=1}^k |h_i|^2.\]
Then $h=\sum_{i=1}^k h_i\overline{h_i}\in H$, as $H$ is an ideal in $C_B(X)$. Note that
\[|h_i|^{-1}\big((1,\infty)\big)\subseteq h^{-1}\big((1,\infty)\big)\]
for each $i=1,\ldots,k$. Therefore $A\subseteq h^{-1}((1,\infty))$ by (\ref{LJD}).

We check that $H=C^\mathfrak{H}_0(X)$. It is clear that $|h|^{-1}((1/n,\infty))=|nh|^{-1}((1,\infty))$ is null for each $h\in H$ and positive integer $n$. Therefore $H\subseteq C^\mathfrak{H}_0(X)$. To show the reverse inclusion, let $g\in C^\mathfrak{H}_0(X)$. Let $n$ be a positive integer. Then $|g|^{-1}([1/n,\infty))$ is null. Therefore
\[|g|^{-1}\big([1/n,\infty)\big)\subseteq |h_n|^{-1}\big([1,\infty)\big)\]
by (\ref{HJGF}) for some $h_n\in H$. Since this holds for every positive integer $n$, by Lemma \ref{JJHF}, there is a sequence $v_1,v_2,\ldots$ in $C_B(X)$ such that $h_nv_n\rightarrow g$. Observe that $h_1v_1,h_2v_2,\ldots$ is a sequence in $H$, as $H$ is an ideal in $C_B(X)$. This implies that $g\in H$, as $H$ is closed in $C_B(X)$. Thus $C^\mathfrak{H}_0(X)\subseteq H$ and therefore $H=C^\mathfrak{H}_0(X)$ by the above part.

Finally, we check that $Y=\lambda_\mathfrak{H}X$. Note that $\lambda_\mathfrak{H}X\subseteq Y$ by (\ref{HJGF}) and the definition of $\lambda_\mathfrak{H}X$. On the other hand, if $C=|h|^{-1}((1,\infty))$, where $h\in H$, then $C\in\mathrm{Coz}(X)$ (indeed, $C=\mathrm{coz}(g)$ where $g=\max\{|h|-\mathbf{1},\mathbf{0}\}$) and $\mathrm{cl}_XC$ has a null neighborhood in $X$, say $|2h|^{-1}((1,\infty))$. Therefore $\mathrm{int}_{\beta X}\mathrm{cl}_{\beta X}C\subseteq\lambda_\mathfrak{H}X$ (by the definition of $\lambda_\mathfrak{H}X$). Thus $Y\subseteq\lambda_\mathfrak{H}X$ by Lemma \ref{KLGS} (as $Y=\lambda_HX$).
\end{proof}

\begin{notation}\label{ADFJ}
Let $X$ be a completely regular space and let $H$ be a non-vanishing closed ideal in $C_B(X)$. We denote by $\mathfrak{sp}(H)$ the spectrum of $H$.
\end{notation}

Next, in a series of results, we use the construction of the spectrum given in Theorem \ref{TRES} to study its properties.

The following theorem introduces a necessary and sufficient condition for the spectrum of a non-vanishing closed ideal of $C_B(X)$ to be $\sigma$-compact. First, we prove the following lemma (which will be used sometimes without explicit reference to it).

\begin{lemma}\label{KJGFF}
Let $X$ be a completely regular space. Then, for any $f,g\in C_B(X)$ and $r,s\in\mathbb{C}$ we have
\begin{itemize}
\item[\rm(1)] $(rf+sg)_\beta=rf_\beta+sg_\beta$.
\item[\rm(2)] $(fg)_\beta=f_\beta g_\beta$.
\item[\rm(3)] $(\overline{f})_\beta=\overline{f_\beta}$.
\item[\rm(4)] $|f_\beta|=|f|_\beta$.
\item[\rm(5)] $\|f_\beta\|=\|f\|$.
\end{itemize}
\end{lemma}

\begin{proof}
To show (1) observe that $(rf+sg)_\beta$ and $rf_\beta+sg_\beta$ are identical, as they are continuous mappings which both coincide with $rf+sg$ on the dense subspace $X$ of $\beta X$. That (2), (3) and (4) hold follows analogously.

To show (5), note that
\[|f_\beta|(\beta X)=|f_\beta|(\mathrm{cl}_{\beta X}X)\subseteq\overline{|f_\beta|(X)}=\overline{|f|(X)}\subseteq\big[0,\|f\|\big],\]
where the bar denotes the closure in $\mathbb{R}$. This yields $\|f_\beta\|\leq\|f\|$. That $\|f\|\leq\|f_\beta\|$ is clear, as $f_\beta$ is an extension of $f$.
\end{proof}

For a ring $R$ and $A\subseteq R$ we denote by $\langle A\rangle$ the ideal in $R$ generated by $A$.

\begin{theorem}\label{JUHGH}
Let $X$ be a completely regular space. Let $H$ be a non-vanishing closed ideal in $C_B(X)$. The following are equivalent:
\begin{itemize}
\item[\rm(1)] $\mathfrak{sp}(H)$ is $\sigma$-compact.
\item[\rm(2)] $H$ is $\sigma$-generated, that is,
\[H=\overline{\langle f_1,f_2,\ldots\rangle}\]
for some $f_1,f_2,\ldots$ in $C_B(X)$.
\end{itemize}
Here the bar denotes the closure in $C_B(X)$.
\end{theorem}

\begin{proof}
(1) \emph{implies} (2). Suppose that the spectrum of $H$, which by Theorem \ref{TRES} equals to $\lambda_H X$, is $\sigma$-compact. Let
\begin{equation}\label{OJHD}
\lambda_H X=\bigcup_{n=1}^\infty K_n
\end{equation}
where $K_1,K_2,\ldots$ are compact. For each positive integer $n$, using the representation of $\lambda_H X$ given in Lemma \ref{KLGS} and compactness of $K_n$, we have
\begin{equation}\label{KIHS}
K_n\subseteq\bigcup_{i=1}^{k_n}\mathrm{int}_{\beta X}\mathrm{cl}_{\beta X}|h^n_i|^{-1}\big((1,\infty)\big)
\end{equation}
where $h^n_i\in H$ for each $i=1,\ldots,k_n$. Let
\[g_n=\sum_{i=1}^{k_n}|h^n_i|^2.\]
Then $g_n=\sum_{i=1}^{k_n}h^n_i\overline{h^n_i}\in H$, as $H$ is an ideal in $C_B(X)$. We may assume the choices are such that
\[\big\{h^n_1,\ldots,h^n_{k_n}\big\}\subseteq\big\{h^{n+1}_1,\ldots,h^{n+1}_{k_{n+1}}\big\}.\]
Thus
\begin{equation}\label{JHGD}
g_n\leq g_{n+1}.
\end{equation}
Arguing as in the proof of Lemma \ref{GFS} we have
\begin{equation}\label{KOJUS}
K_n\subseteq\mathrm{int}_{\beta X}\mathrm{cl}_{\beta X}g_n^{-1}\big((1,\infty)\big).
\end{equation}
On the other hand
\[\mathrm{int}_{\beta X}\mathrm{cl}_{\beta X}g_n^{-1}\big((1,\infty)\big)\subseteq\lambda_H X\]
by the definition of $\lambda_H X$. This, together with (\ref{OJHD}) and (\ref{KOJUS}), implies that
\begin{equation}\label{JGS}
\lambda_H X=\bigcup_{n=1}^\infty\mathrm{int}_{\beta X}\mathrm{cl}_{\beta X}g_n^{-1}\big((1,\infty)\big).
\end{equation}

Next, we show that
\begin{equation}\label{PIHD}
H=\overline{\langle g_1,g_2,\ldots\rangle};
\end{equation}
this will conclude the proof. Let $h\in H$. Note that by Lemma \ref{LKG} we have
\[|h^\beta|^{-1}\big((1/2,\infty)\big)\subseteq\mathrm{int}_{\beta X}\mathrm{cl}_{\beta X}|h|^{-1}\big((1/2,\infty)\big),\]
where the latter set is contained in $\lambda_H X$ by the definition of $\lambda_H X$. In particular,
\[|h^\beta|^{-1}\big([1,\infty)\big)\subseteq\lambda_H X.\]
By compactness of $|h^\beta|^{-1}([1,\infty))$ and the representation given in (\ref{JGS}) we have
\[|h^\beta|^{-1}\big([1,\infty)\big)\subseteq\bigcup_{i=1}^{l_n}\mathrm{int}_{\beta X}\mathrm{cl}_{\beta X}g_i^{-1}\big((1,\infty)\big),\]
for some positive integer $l_n$. Thus, intersecting the two sides of the above relation with $X$, it gives
\begin{equation}\label{KJH}
|h|^{-1}\big([1,\infty)\big)\subseteq\bigcup_{i=1}^{l_n}g_i^{-1}\big([1,\infty)\big)
\end{equation}
But, by (\ref{JHGD}) we have
\[g_i^{-1}\big([1,\infty)\big)\subseteq g_{l_n}^{-1}\big([1,\infty)\big).\]
for each $i=1,\ldots,l_n$. Therefore, (\ref{KJH}) yields
\[|h|^{-1}\big([1,\infty)\big)\subseteq g_{l_n}^{-1}\big([1,\infty)\big).\]
By Lemma \ref{JJHF} there is a sequence $f_1,f_2,\ldots$ in $C_B(X)$ such that $g_{l_n}f_n\rightarrow h$. That is $h\in\overline{\langle g_1,g_2,\ldots\rangle}$. This shows that $H\subseteq\overline{\langle g_1,g_2,\ldots\rangle}$. The reverse inclusion follows trivially, as $H$ is a closed ideal in $C_B(X)$ which contains the generators $g_n$ for all positive integer $n$. This proves (\ref{PIHD}).

(2) \emph{implies} (1). Let
\[H=\overline{\langle f_1,f_2,\ldots\rangle}\]
where $f_1,f_2,\ldots$ are in $C_B(X)$. We prove that $\lambda_H X=Y$, where
\[Y=\bigcup\big\{\mathrm{cl}_{\beta X}|g|^{-1}\big((1/n,\infty)\big):g\in\Omega\mbox{ and }n=1,2,\ldots\big\},\]
in which
\[\Omega=\{e^{\mathbf{i}\theta_1}f_1+\cdots+e^{\mathbf{i}\theta_k}f_k:\theta_1,\ldots,\theta_k\in\mathbb{Q}\mbox{ and }k=1,2,\ldots\}.\]
Note that $\Omega$ is a countable subset of $H$. This will imply that the spectrum of $H$ (which equals to $\lambda_H X$ by Theorem \ref{TRES}) is $\sigma$-compact.

Let $t\in\lambda_H X$. Then, by the definition, $t\in\mathrm{cl}_{\beta X}|h|^{-1}((1,\infty))$ for some $h\in H$. Note that
\[\mathrm{cl}_{\beta X}|h|^{-1}\big((1,\infty)\big)\subseteq|h^\beta|^{-1}\big([1,\infty)\big).\]
Thus $t\in|h^\beta|^{-1}([1,\infty))$ and therefore
\begin{equation}\label{OJAS}
\big|h^\beta(t)\big|>1/2.
\end{equation}
Let $h_1,h_2,\ldots$ be a sequence in $\langle f_1,f_2,\ldots\rangle$ such that $h_n\rightarrow h$. Then $h^\beta_n\rightarrow h^\beta$, as \[\|h^\beta_n-h^\beta\|=\big\|(h_n-h)^\beta\big\|=\|h_n-h\|\]
by Lemma \ref{KJGFF}. Thus $h_n^\beta(t)\rightarrow h^\beta(t)$ and therefore $|h_n^\beta(t)|\rightarrow|h^\beta(t)|$. Because of (\ref{OJAS}) there is some positive integer $k$ such that $|h_k^\beta(t)|>1/2$. We may write
\[h_k=\sum_{i=1}^lu_if_i,\]
where $u_1,\ldots,u_l$ are in $C_B(X)$. Then
\[|h_k|^2=h_k\overline{h_k}=\sum_{i=1}^l\overline{h_k}u_if_i=\sum_{i=1}^lv_if_i\]
where $v_i=\overline{h_k}u_i$ for each $i=1,\ldots,l$. Note that $v_i\in C_B(X)$ for each $i=1,\ldots,l$. Choose some integer $M$ such that $M\geq\|v_i\|$ ($=\|v^\beta_i\|$ by Lemma \ref{KJGFF}) for each $i=1,\ldots,l$. Using Lemma \ref{KJGFF}
\[\big(|h_k|^2\big)^\beta=(h_k\overline{h_k})^\beta=h_k^\beta\overline{h_k}^\beta=h_k^\beta\overline{h_k^\beta}=|h^\beta_k|^2.\]
We have
\begin{eqnarray*}
\frac{1}{4}<|h_k^\beta|^2(t)=\big(|h_k|^2\big)^\beta(t)&=&
\Big(\sum_{i=1}^lv_if_i\Big)^\beta(t)\\&=&\sum_{i=1}^l(v_i f_i)^\beta(t)\\&=&\sum_{i=1}^lv_i^\beta f^\beta_i(t)\\&\leq&\sum_{i=1}^l\big|v_i^\beta f^\beta_i(t)\big|\leq\sum_{i=1}^l\|v_i^\beta\|\big|f^\beta_i(t)\big|\leq M\sum_{i=1}^l\big|f^\beta_i(t)\big|.
\end{eqnarray*}
For each $i=1,\ldots,l$ let $\zeta_i$ be a real number such that
\[\big|f^\beta_i(t)\big|=e^{\mathbf{i}\zeta_i}f^\beta_i(t).\]
Then, by the above relation
\begin{equation}\label{PKHS}
\frac{1}{4}<M\sum_{i=1}^le^{\mathbf{i}\zeta_i}f^\beta_i(t).
\end{equation}
For each $i=1,\ldots,l$ let $\zeta^i_1,\zeta^i_2,\ldots$ be a sequence of rational numbers such that $\zeta^i_n\rightarrow\zeta_i$. Then
\[\sum_{i=1}^le^{\mathbf{i}\zeta^i_n}f^\beta_i(t)\longrightarrow\sum_{i=1}^le^{\mathbf{i}\zeta_i}f^\beta_i(t).\]
By (\ref{PKHS}) there is some positive integer $n$ such that
\begin{equation}\label{OKH}
\frac{1}{4M}<\sum_{i=1}^le^{\mathbf{i}\zeta^i_n}f^\beta_i(t)=g^\beta(t)
\end{equation}
where
\[g=\sum_{i=1}^le^{\mathbf{i}\zeta^i_n}f_i.\]
It is clear that $g\in\Omega$. Observe that
\[g_\beta^{-1}\Big(\Big(\frac{1}{4M},\infty\Big)\Big)\subseteq\mathrm{cl}_{\beta X}g^{-1}\Big(\Big(\frac{1}{4M},\infty\Big)\Big)\]
by Lemma \ref{LKG}. But $t\in g_\beta^{-1}((1/(4M),\infty))$ by (\ref{OKH}). Thus $t\in Y$ by the definition of $Y$. This shows that $\lambda_H X\subseteq Y$. That $Y\subseteq\lambda_H X$ is clear and follows from the definition of $\lambda_H X$ (and the fact that $\Omega\subseteq H$).
\end{proof}

\begin{corollary}\label{JKJH}
Let $X$ be a completely regular space. Let $H$ be a non-vanishing closed ideal in $C_B(X)$. Then, if $\mathfrak{sp}(H)$ is $\sigma$-compact then $H$ has an element vanishing nowhere on $X$.
\end{corollary}

\begin{proof}
Suppose that the spectrum of $H$ is $\sigma$-compact. By Theorem \ref{JUHGH} then $H$ is $\sigma$-generated by some $f_1,f_2,\ldots$ in $C_B(X)$, that is
\begin{equation}\label{OYKH}
H=\overline{\langle f_1,f_2,\ldots\rangle}.
\end{equation}
We may assume that $f_n\neq\mathbf{0}$ for each positive integer $n$. For each positive integer $n$ let $g_n=|f_n|^2$. Define
\[g=\sum_{n=1}^\infty\frac{g_n}{2^n\|g_n\|}.\]
Note that $g$ is well defined by the Weierstrass $M$-test. Also, $g_n=f_n\overline{f_n}\in H$ for each positive integer $n$, and thus $g\in H$, as $g$ is the limit of a sequence in $H$ and $H$ is closed in $C_B(X)$. We show that $g$ does not vanish anywhere on $X$. Let $x\in X$. Since $H$ is non-vanishing, there is some $h\in H$ with $h(x)\neq0$. By (\ref{OYKH}) there is a sequence $u_1,u_2,\ldots$ in $\langle f_1,f_2,\ldots\rangle$ such that $u_n\rightarrow h$. In particular, $u_n(x)\rightarrow h(x)$. Thus, $u_k(x)\neq0$ for some positive integer $k$. But, $u_k=v_1f_1+\cdots+v_mf_m$ for some $v_1,\ldots,v_m$ in $C_B(X)$. Therefore $f_i(x)\neq0$ for some $i=1,\ldots,m$ and thus $g_i(x)\neq0$. Therefore $g(x)\neq0$ by the definition of $g$.
\end{proof}

The following theorem (which is a further application of Theorem \ref{TRES}) may be considered as a complement to Theorem \ref{JUHGH}; while Theorem \ref{JUHGH} gives a necessary and sufficient condition such that the spectrum of a non-vanishing closed ideal of $C_B(X)$ is $\sigma$-compact, the following theorem provides a necessary and sufficient condition such that every $\sigma$-compact subspace of the spectrum is contained in a compact subspace. Clearly, the spectrum is compact if it satisfies the conditions in both theorems. (Compare Theorems \ref{JUHGH} and \ref{OYSD} to Theorems \ref{ADEFR} and \ref{FGFF}, respectively. Note that Theorems \ref{JUHGH} and \ref{OYSD} require the spaces under consideration to be only completely regular, while the spaces under consideration in Theorems \ref{ADEFR} and \ref{FGFF} are normal. In particular, Theorems \ref{JUHGH} and \ref{OYSD} are not deducible from Theorems \ref{ADEFR} and \ref{FGFF}.)

\begin{theorem}\label{OYSD}
Let $X$ be a completely regular space. Let $H$ be a non-vanishing closed ideal in $C_B(X)$. The following are equivalent:
\begin{itemize}
\item[\rm(1)] Every $\sigma$-compact subspace of $\mathfrak{sp}(H)$ is contained in a compact subspace; in particular, $\mathfrak{sp}(H)$ is countably compact.
\item[\rm(2)] $C_{00}(\mathfrak{sp}(H))=C_0(\mathfrak{sp}(H))$.
\item[\rm(3)] $H=\bigcup_{h\in H}\mathrm{Ann}(h-\mathbf{1})$.
\end{itemize}
Here $\mathrm{Ann}(h)$ denotes the annihilator of $h$.
\end{theorem}

\begin{proof}
Conditions (1) and (2) are known to be equivalent. (See the observation made prior to Theorem \ref{ADEFR}.)

(1) \emph{implies} (3). Let $h\in H$. For any $r>0$ we have
\[|h_\beta|^{-1}\big((r,\infty)\big)\subseteq\mathrm{int}_{\beta X}\mathrm{cl}_{\beta X}|h|^{-1}\big((r,\infty)\big)\]
by (a slight modification of) Lemma \ref{LKG} (and Lemma \ref{KJGFF}), thus $|h_\beta|^{-1}((r,\infty))\subseteq\lambda_H X$. In particular, the $\sigma$-compact subspace $\mathrm{coz}(h_\beta)=\bigcup_{n=1}^\infty|h_\beta|^{-1}([1/n,\infty))$ of $\beta X$ is contained in $\lambda_H X$. Note that the spectrum of $H$ equals to $\lambda_H X$ by Theorem \ref{TRES}. Thus, by our assumption, $\mathrm{coz}(h_\beta)$ is contained in a compact subspace $K$ of $\lambda_H X$. By Lemma \ref{GFS} we have $K\subseteq\mathrm{cl}_{\beta X}f^{-1}((1,\infty))$ for some non-negative $f\in H$. Therefore $\mathrm{coz}(h_\beta)\subseteq\mathrm{cl}_{\beta X}f^{-1}((1,\infty))$. Now, if we intersect the two sides of the latter with $X$ we obtain
\begin{equation}\label{LKOHS}
\mathrm{coz}(h)\subseteq f^{-1}\big([1,\infty)\big).
\end{equation}
Let $u:X\rightarrow\mathbb{R}$ be the mapping defined by
\[u(x)=\left\{
             \begin{array}{ll}
               1/f(x), & \hbox{ if }x\in f^{-1}([1,\infty));  \\
               f(x), & \hbox{ if }x\in f^{-1}([0,1]).
             \end{array}
         \right.\]
It is clear that $u$ is well defined, continuous, and bounded. That is $u\in C_B(X)$. Let $g=uf$. Note that $g\in H$ (as $f\in H$ and $H$ is an ideal in $C_B(X)$). It follows from the definitions that $f^{-1}([1,\infty))\subseteq g^{-1}(1)$. This combined with (\ref{LKOHS}) gives $\mathrm{coz}(h)\subseteq g^{-1}(1)$, which implies that $hg=h$. Therefore $h\in\mathrm{Ann}(g-\mathbf{1})$. This shows that $H\subseteq\bigcup_{h\in H}\mathrm{Ann}(h-\mathbf{1})$. The reverse inclusion is trivial, as $H$ is an ideal in $C_B(X)$.

(3) \emph{implies} (1). Let $K$ be a $\sigma$-compact subspace of the spectrum of $H$ (which equals to $\lambda_H X$ by Theorem \ref{TRES}). Let $K=\bigcup_{n=1}^\infty K_n$ where $K_n$ is compact for each positive integer $n$. By Lemma \ref{GFS} for each positive integer $n$ there is some non-negative $f_n\in H$ such that
\begin{equation}\label{ASDHS}
K_n\subseteq\mathrm{cl}_{\beta X}f_n^{-1}\big((1,\infty)\big).
\end{equation}
Let
\[f=\sum_{n=1}^\infty\frac{f_n}{2^n\|f_n\|}:X\longrightarrow\mathbb{R}.\]
Then $f$ is a well defined continuous bounded mapping. Also, $f\in H$, as $f$ is the limit a sequence in $H$ (and $H$ is closed in $C_B(X)$). By our assumption there is some $g\in H$ such that $f\in\mathrm{Ann}(g-\mathbf{1})$. That is $fg=f$, or equivalently, $\mathrm{coz}(f)\subseteq g^{-1}(1)$. In particular,
\[\mathrm{cl}_{\beta X}\mathrm{coz}(f)\subseteq\mathrm{cl}_{\beta X} g^{-1}(1)\subseteq g_\beta^{-1}(1).\]
Note that
\[g_\beta^{-1}\big((1/2,\infty)\big)\subseteq\mathrm{int}_{\beta X}\mathrm{cl}_{\beta X}g^{-1}\big((1/2,\infty)\big)\]
by (a slight modification of) Lemma \ref{LKG}, and
\[\mathrm{int}_{\beta X}\mathrm{cl}_{\beta X}g^{-1}\big((1/2,\infty)\big)\subseteq\lambda_H X.\]
Therefore $\mathrm{cl}_{\beta X}\mathrm{coz}(f)\subseteq\lambda_H X$. But $f_n^{-1}((1,\infty))\subseteq\mathrm{coz}(f)$ for each positive integer $n$ (by the definition of $f$) and thus $K_n\subseteq\mathrm{cl}_{\beta X}\mathrm{coz}(f)$ by (\ref{ASDHS}). That is the compact subspace $\mathrm{cl}_{\beta X}\mathrm{coz}(f)$ of $\lambda_H X$ contains $K$.
\end{proof}

In Theorems \ref{JUHGH} and \ref{OYSD} for a non-vanishing closed ideal of $C_B(X)$ we considered compactness like (such as $\sigma$-compactness and countable compactness) properties of its spectrum. In the next few results we consider connectedness like properties of the spectrum. This includes connectedness and local connectedness in particular.

For a collection $\{H_i:i\in I\}$ of ideals of a ring $R$ we denote the ideal $\langle\bigcup_{i\in I}H_i\rangle$ by
$\bigoplus_{i\in I}H_i$ provided that $H_i\cap\langle\bigcup_{i\neq j\in I}H_j\rangle=\mathbf{0}$ for each $i\in I$.

\begin{theorem}\label{HJGS}
Let $X$ be a completely regular space. Let $H$ be a non-vanishing closed ideal in $C_B(X)$. The following are equivalent:
\begin{itemize}
\item[\rm(1)] $\mathfrak{sp}(H)$ is connected.
\item[\rm(2)] $H$ is indecomposable, that is,
\[H\neq I\oplus J\]
for any non-zero ideals $I$ and $J$ of $C_B(X)$.
\end{itemize}
\end{theorem}

\begin{proof}
(1) \emph{implies} (2). Suppose that $H$ is not indecomposable, that is, $H=I\oplus J$ for some non-zero ideals $I$ and $J$ of $C_B(X)$. We show that the pair
\[A=\bigcup_{h\in I}\mathrm{int}_{\beta X}\mathrm{cl}_{\beta X}|h|^{-1}\big((1,\infty)\big)\quad\mbox{and}\quad B=\bigcup_{h\in J}\mathrm{int}_{\beta X}\mathrm{cl}_{\beta X}|h|^{-1}\big((1,\infty)\big)\]
form a separation for $\lambda_HX$ ($=\mathfrak{sp}(H)$ by Theorem \ref{TRES}). Clearly $A$ and $B$ are open subspaces of $\lambda_HX$. Also, $A$ and $B$ are non-empty, as $I$ and $J$ are non-zero. (To see this, let $\mathbf{0}\neq f\in I$. Let $x\in X$ such that $f(x)\neq0$. Let $n$ be a positive integer with $1/n<|f(x)|$. Let $g=nf$. Then $x\in A$, as $x\in|g_\beta|^{-1}((1,\infty))$ and $|g_\beta|^{-1}((1,\infty))\subseteq\mathrm{int}_{\beta X}\mathrm{cl}_{\beta X}|g|^{-1}((1,\infty))$ by (a slight modification of) Lemma \ref{LKG}. This shows that $A\neq\emptyset$. Similarly $B\neq\emptyset$.) We now prove that $\lambda_HX=A\cup B$. Let $t\in\lambda_HX$. Then $t\in\mathrm{cl}_{\beta X}|h|^{-1}((1,\infty))$ for some $h\in H$. Then $t\in|h_\beta|^{-1}([1,\infty))$. Let $f\in I$ and $g\in J$ such that $h=f+g$. Note that $h_\beta=f_\beta+g_\beta$ by Lemma \ref{KJGFF}. Since $|h_\beta(t)|\geq1$, we have either $|f_\beta(t)|>1/3$ or $|g_\beta(t)|>1/3$. But then, as argued above, either $t\in\mathrm{int}_{\beta X}\mathrm{cl}_{\beta X}|3f|^{-1}((1,\infty))$ or $t\in\mathrm{int}_{\beta X}\mathrm{cl}_{\beta X}|3g|^{-1}((1,\infty))$. Consequently, either $t\in A$ or $t\in B$. That is $\lambda_HX\subseteq A\cup B$. The reverse inclusion holds trivially. Finally, we prove that $A\cap B=\emptyset$. Suppose to the contrary that there is some $s\in A\cap B$. Then $s\in\mathrm{int}_{\beta X}\mathrm{cl}_{\beta X}|f|^{-1}((1,\infty))$ and $s\in\mathrm{int}_{\beta X}\mathrm{cl}_{\beta X}|g|^{-1}((1,\infty))$ for some $f\in I$ and $g\in J$. Therefore $s\in|f_\beta|^{-1}([1,\infty))$ and $s\in|g_\beta|^{-1}([1,\infty))$. Thus $f_\beta g_\beta(s)\neq 0$. But (using Lemma \ref{KJGFF}) we have $f_\beta g_\beta=(fg)_\beta=\mathbf{0}$, as $fg\in I\cap J=\mathbf{0}$. This contraction proves that $A\cap B=\emptyset$. That is, the pair $A$ and $B$ form a separation for $\lambda_HX$ and $\lambda_HX$ is therefore disconnected.

(2) \emph{implies} (1). Suppose that $\lambda_HX$ is disconnected. Let $C$ and $D$ be a separation for $\lambda_HX$. Let
\[K=\{h\chi_{(X\cap C)}:h\in H\}\quad\mbox{and}\quad L=\{h\chi_{(X\cap D)}:h\in H\},\]
where $\chi$ is used to denote the characteristic function. It is easy to see that $K$ and $L$ are ideals of $C_B(X)$ which are contained in $H$. We check that $K$ and $L$ are non-zero. Note that $C$ is open in $\beta X$ (as $\lambda_HX$  is so by Lemma \ref{KLGS}, and $C$ is open in $\lambda_HX$) and is non-empty. Therefore $X\cap C\neq\emptyset$. Let $x\in X\cap C$. Let $h\in H$ such that $h(x)\neq0$. Then $h\chi_{(X\cap C)}\in K$ and $h\chi_{(X\cap C)}\neq\mathbf{0}$, as $h\chi_{(X\cap C)}(x)\neq0$. That is $K$ is non-zero. Similarly, $L$ is non-zero. It is clear that for every $h\in H$ we have $h=h\chi_{(X\cap C)}+h\chi_{(X\cap D)}$, as the sets $X\cap C$ and $X\cap D$ cover $X$ and are disjoint. Thus $H\subseteq K+L$ and therefore $H=K+L$. It is also clear that $K\cap L=\mathbf{0}$. Thus $H=K\oplus L$ and $H$ is therefore not indecomposable.
\end{proof}

Let $\{Y_i:i\in I\}$ be a collection of topological spaces. We may assume that the spaces $Y_i$'s are pairwise disjoint. The (\textit{topological}) \textit{direct sum} of $\{Y_i:i\in I\}$, denoted by $\bigoplus_{i\in I}Y_i$, is the set $Y=\bigcup_{i\in I}Y_i$ together with the family of open subsets $\mathscr{O}$ consisting of all $U\subseteq Y$ such that $U\cap Y_i$ is open in $Y_i$ for every $i\in I$. Clearly, a space $Y$ is the direct sum of a collection  $\{Y_i:i\in I\}$ of its subspaces if and only if $\{Y_i:i\in I\}$ form a collection of pairwise disjoint open subspaces of $Y$ whose union is the whole $Y$.

\begin{theorem}\label{JFDF}
Let $X$ be a completely regular space. Let $H$ be a non-vanishing closed ideal in $C_B(X)$. The following are equivalent:
\begin{itemize}
\item[\rm(1)] Components of $\mathfrak{sp}(H)$ are open in $\mathfrak{sp}(H)$; in particular, $\mathfrak{sp}(H)$ is locally connected.
\item[\rm(2)] There is a representation
\[H=\overline{\bigoplus_{i\in I}H_i},\]
where $H_i$ is an indecomposable closed ideal in $C_B(X)$ for any $i\in I$.
\end{itemize}
Here the bar denotes the closure in $C_B(X)$.
\end{theorem}

\begin{proof}
(1) \emph{implies} (2). Let $\mathscr{U}=\{U_i:i\in I\}$ be the collection of all components of $\lambda_HX$ ($=\mathfrak{sp}(H)$ by Theorem \ref{TRES}) which are faithfully indexed. Note that components of any space are always closed. Thus, in particular, $U_i$'s are both open and closed in $\lambda_HX$.

For each $i\in I$ let
\[H_i=\{h\chi_{X_i}:h\in H\},\]
where $X_i=X\cap U_i$ and $\chi$ is used to denote the characteristic function. We prove that $H_i$ is an indecomposable closed ideal of $C_B(X)$ for any $i\in I$, and
\[H=\overline{\bigoplus_{i\in I}H_i}.\]

It is easy to see that $H_i$ is an ideal in $C_B(X)$ for any $i\in I$, and $H_i$ is contained in $H$. Let $i_0\in I$. Let $f\in H_{i_0}\cap\langle\bigcup_{i_0\neq i\in I}H_i\rangle$. Then
\[f=h_0\chi_{X_{i_0}}=h_1\chi_{X_{i_1}}+\cdots+h_n\chi_{X_{i_n}},\]
where $h_0,\ldots,h_n\in H$, in which $i_0,\ldots,i_n\in I$ are distinct. This implies that $f=\mathbf{0}$, as $U_{i_0},\ldots,U_{i_n}$ (and therefore $X_{i_0},\ldots,X_{i_n}$) are pairwise disjoint. That is the ideal $\langle\bigcup_{i\in I}H_i\rangle$ generated by $H_i$'s (which is clearly contained in $H$) is the direct sum $\bigoplus_{i\in I}H_i$. Since $H$ is closed in $C_B(X)$, this in particular implies that
\begin{equation}\label{JJHG}
\overline{\bigoplus_{i\in I}H_i}\subseteq H.
\end{equation}

We now check that the reverse inclusion also holds in the above relation. Let $h\in H$. Let $\epsilon>0$. Observe that $|h_\beta|^{-1}((\epsilon,\infty))$ is contained in $\lambda_HX$, as it is contained in $\mathrm{cl}_{\beta X}|h/\epsilon|^{-1}((1,\infty))$ (by Lemma \ref{LKG}) and the latter is so. Let $n$ be a positive integer. Then $|h_\beta|^{-1}([1/n,\infty))$ is contained in $\lambda_HX$ and it is compact. Since  $\mathscr{U}$ is an open cover for $\lambda_HX$ it then follows that
\begin{equation}\label{KHDF}
|h_\beta|^{-1}\big([1/n,\infty)\big)\subseteq U_{j_1}\cup\cdots\cup U_{j_n}
\end{equation}
for some distinct $j_1,\ldots,j_n\in I$. Let
\[h_n=h\chi_{X_{j_1}}+\cdots+h\chi_{X_{j_n}}.\]
It is clear that $h_n\in\bigoplus_{i\in I}H_i$. We prove that $h_n\rightarrow h$. This will imply that $h\in\overline{\bigoplus_{i\in I}H_i}$. But this follows trivially, as by (\ref{KHDF}) we have
\[h_n(x)=h(x),\quad\mbox{if}\quad x\in X\cap(U_{j_1}\cup\cdots\cup U_{j_n})\]
and
\[|h_n(x)-h(x)|\leq 1/n,\quad\mbox{if}\quad x\in X\setminus(U_{j_1}\cup\cdots\cup U_{j_n})\]
for any positive integer $n$, and thus $\|h_n-h\|\leq 1/n$. This proves the reverse inclusion in (\ref{JJHG}).

To conclude the proof we need to show that $H_i$'are indecomposable. Let $i\in I$. Note that the elements of $H_i$ vanish outside of $X_i$. We can therefore identify each element $h\chi_{X_i}$ of $H_i$ (where $h\in H$) with its restriction on $X_i$. In particular, $H_i$ will be a non-vanishing closed ideal in $C_B(X_i)$, as one can easily check. Observe that $\beta X_i=\mathrm{cl}_{\beta X}X_i$, as $X_i$ is open and closed in $X$. Also, \[|h\chi_{X_i}|^{-1}\big([1,\infty)\big)=X_i\cap|h|^{-1}\big([1,\infty)\big).\]
Note that $X_i$ and $|h|^{-1}([1,\infty))$ are both zero-sets in $X$. (Indeed, $|h|^{-1}([1,\infty))=\mathrm{z}(g)$ where $g=\min\{|h|-\mathbf{1},\mathbf{0}\}$.) Therefore
\[\mathrm{cl}_{\beta X}\big[X_i\cap|h|^{-1}\big([1,\infty)\big)\big]=\mathrm{cl}_{\beta X}X_i\cap\mathrm{cl}_{\beta X}|h|^{-1}\big([1,\infty)\big).\]
We combine these and use Theorem \ref{TRES} to obtain
\begin{eqnarray*}
\mathfrak{sp}(H_i)&=&\bigcup_{h\in H}\mathrm{cl}_{\beta X_i}|h\chi_{X_i}|^{-1}\big([1,\infty)\big)\\&=&\bigcup_{h\in H}\big[\mathrm{cl}_{\beta X}X_i\cap\mathrm{cl}_{\beta X}|h|^{-1}\big([1,\infty)\big)\big]\\&=&\mathrm{cl}_{\beta X}X_i\cap\bigcup_{h\in H}\mathrm{cl}_{\beta X}|h|^{-1}\big([1,\infty)\big)=\mathrm{cl}_{\beta X}X_i\cap\lambda_HX=\mathrm{cl}_{\lambda_HX}X_i.
\end{eqnarray*}
But
\[\mathrm{cl}_{\lambda_HX}X_i=\mathrm{cl}_{\lambda_HX}(X\cap U_i)=\mathrm{cl}_{\lambda_HX}U_i,\]
as $U_i$ is open in $\lambda_HX$ (and $X$ is dense in $\lambda_HX$). Therefore $\mathfrak{sp}(H_i)=\mathrm{cl}_{\lambda_HX}U_i$. But $U_i$ is connected and thus so is its closure in $\lambda_HX$. That is $\mathfrak{sp}(H_i)$ is connected. It now follows from Theorem \ref{HJGS} that $H_i$ is indecomposable.

(2) \emph{implies} (1). For each $i\in I$ let
\[X_i=\bigcup_{f\in H_i}\mathrm{coz}(f).\]
We show that
\begin{equation}\label{OPHS}
X=\bigoplus_{i\in I}X_i,
\end{equation}
that is, $\{X_i:i\in I\}$ forms a collection of pairwise disjoint open subspaces of $X$ whose union is the whole $X$. It is clear that $X_i$ is open in $X$ for each $i\in I$. Since $H$ is non-vanishing, for each $x\in X$ there is some $f\in H$ such that $f(x)\neq0$. But $f$ is the limit of a sequence $f_1,f_2,\ldots$ in $\bigoplus_{i\in I}H_i$. Therefore $f_n(x)\neq0$ for some positive integer $n$. Since $f_n$ is a finite sum of elements from $\bigcup_{i\in I}H_i$ it follows that $h(x)\neq0$ for some $h\in H_i$ and $i\in I$. Thus $x\in X_i$. That is $X_i$'s cover $X$. To show that $X_i$'s are pairwise disjoint, suppose otherwise that $x\in X_i\cap X_j$ for some distinct $i,j\in I$. Then $x\in\mathrm{coz}(f_i)\cap\mathrm{coz}(f_j)$ where $f_i\in H_i$ and $f_j\in H_j$. But $f_if_j\in H_i\cap H_j=\mathbf{0}$, which is a contradiction. This proves (\ref{OPHS}).

Next, we shows that
\begin{equation}\label{RRC}
\mathfrak{sp}(H)=\bigoplus_{i\in I}\mathfrak{sp}(H_i),
\end{equation}
that is, $\{\mathfrak{sp}(H_i):i\in I\}$ is a collection of pairwise disjoint open subspaces of $\mathfrak{sp}(H)$ whose union is $\mathfrak{sp}(H)$.

First, we check that
\begin{equation}\label{JOKIBH}
\mathfrak{sp}(H_i)=\bigcup_{f\in H_i}\mathrm{cl}_{\beta X}|f|^{-1}\big((1,\infty)\big)
\end{equation}
for any $i\in I$. Let $i\in I$. Note that the elements of $H_i$ vanish outside of $X_i$. We can therefore identify the elements of $H_i$ with their restriction on $X_i$. In particular, $H_i$ will be a non-vanishing closed ideal in $C_B(X_i)$. By Theorem \ref{TRES} we therefore have
\begin{equation}\label{JUOKBH}
\mathfrak{sp}(H_i)=\bigcup_{f\in H_i}\mathrm{cl}_{\beta X_i}|f|^{-1}\big((1,\infty)\big).
\end{equation}
But $\beta X_i=\mathrm{cl}_{\beta X}X_i$, therefore, we have
\[\mathrm{cl}_{\beta X_i}|f|^{-1}\big((1,\infty)\big)=\mathrm{cl}_{\beta X}X_i\cap\mathrm{cl}_{\beta X}|f|^{-1}\big((1,\infty)\big)=\mathrm{cl}_{\beta X}|f|^{-1}\big((1,\infty)\big)\]
for any $f\in H_i$. This together with (\ref{JUOKBH}) proves (\ref{JOKIBH}).

We check that $\mathfrak{sp}(H_i)$'s are open in $\beta X$ (and thus in $\mathfrak{sp}(H)$). To see this, let $i\in I$. Let $f\in H_i$. Then
\[\mathrm{cl}_{\beta X}|f|^{-1}\big((1,\infty)\big)\subseteq|f_\beta|^{-1}\big([1,\infty)\big)\subseteq|2f_\beta|^{-1}\big((1,\infty)\big),\]
and
\[|f_\beta|^{-1}\big((1,\infty)\big)\subseteq\mathrm{cl}_{\beta X}|f|^{-1}\big((1,\infty)\big)\]
by Lemma \ref{LKG}. Thus, using (\ref{JOKIBH}), we have
\[\mathfrak{sp}(H_i)=\bigcup_{f\in H_i}|f_\beta|^{-1}\big((1,\infty)\big).\]

We now check that
\begin{equation}\label{KLDSD}
\mathfrak{sp}(H)=\bigcup_{i\in I}\mathfrak{sp}(H_i).
\end{equation}
Let $t\in\mathfrak{sp}(H)$. Then, by Theorem \ref{TRES} we have $t\in\mathrm{cl}_{\beta X}|f|^{-1}((1,\infty))$ for some $f\in H$. In particular $t\in|f^\beta|^{-1}([1,\infty))$. Let $f_1,f_2,\ldots$ be a sequence in $\bigoplus_{i\in I}H_i$ such that $f_n\rightarrow f$. Then $f^\beta_n\rightarrow f^\beta$, as
\[\|f^\beta_n-f^\beta\|=\big\|(f_n-f)^\beta\big\|=\|f_n-f\|\]
by Lemma \ref{KJGFF}. In particular $f_n^\beta(t)\rightarrow f^\beta(t)$, and therefore $|f_n^\beta(t)|\rightarrow|f^\beta(t)|\geq 1$. Let $0<\epsilon<1$. Let $m$ be a positive integer with $|f_m^\beta(t)|>\epsilon$. By the choice of the sequence $f_1,f_2,\ldots$ we have $f_m=\sum_{j=1}^kh_{i_j}$, where $h_{i_j}\in H_{i_j}$ for each $j=1,\ldots,k$. In particular $f^\beta_m=\sum_{j=1}^kh^\beta_{i_j}$ by Lemma \ref{KJGFF}. Observe that
\begin{eqnarray*}
|f^\beta_m|^{-1}\big((\epsilon,\infty)\big)&=&\bigg|\sum_{j=1}^kh^\beta_{i_j}\bigg|^{-1}\big((\epsilon,\infty)\big)\\&\subseteq&
\bigg(\sum_{j=1}^k|h^\beta_{i_j}|\bigg)^{-1}\big((\epsilon,\infty)\big)\subseteq \bigcup_{j=1}^k|h^\beta_{i_j}|^{-1}\big((\epsilon/k,\infty)\big).
\end{eqnarray*}
But, using (\ref{JOKIBH}), we have
\[|h^\beta_{i_j}|^{-1}\big((\epsilon/k,\infty)\big)\subseteq\mathrm{cl}_{\beta X}|h_{i_j}|^{-1}\big((\epsilon/k,\infty)\big)=\mathrm{cl}_{\beta X}\big|(k/\epsilon) h_{i_j}\big|^{-1}\big((1,\infty)\big)\subseteq\mathfrak{sp}(H_{i_j})\]
for each $j=1,\ldots,k$. Combining the above relations we therefore have
\[|f^\beta_m|^{-1}\big((\epsilon,\infty)\big)\subseteq\bigcup_{j=1}^k\mathfrak{sp}(H_{i_j})\subseteq\bigcup_{i\in I}\mathfrak{sp}(H_i).\]
Since $t\in|f^\beta_m|^{-1}((\epsilon,\infty))$ we have $t\in\bigcup_{i\in I}\mathfrak{sp}(H_i)$. This shows that $\mathfrak{sp}(H)\subseteq\bigcup_{i\in I}\mathfrak{sp}(H_i)$. The reverse inclusion is trivial and follows from (\ref{JOKIBH}) and Theorem \ref{TRES}.

Finally, we check that $\mathfrak{sp}(H_i)$'s are pairwise disjoint. Let $i,j\in I$ be distinct. Then $X_i$ and $X_j$ are disjoint closed and open subspaces of $X$ and therefore they have disjoint closures in $\beta X$. But $\mathfrak{sp}(H_i)$ and $\mathfrak{sp}(H_j)$ are then disjoint, as they are contained in $\mathrm{cl}_{\beta X}X_i$ ($=\beta X_i$) and $\mathrm{cl}_{\beta X}X_j$ ($=\beta X_j$), respectively. This proves (\ref{RRC}).

To conclude the proof, observe that $\mathfrak{sp}(H_i)$ is connected for each $i\in I$ by Theorem \ref{HJGS}, as $H_i$ is indecomposable. Therefore, the collection $\{\mathfrak{sp}(H_i):i\in I\}$ coincides with the collection of all components of $\mathfrak{sp}(H)$. In particular, components of $\mathfrak{sp}(H)$ are all open in $\mathfrak{sp}(H)$.
\end{proof}

In the next theorem, for a collection $\{H_i:i\in I\}$ of ideals in $C_B(X)$ we study the relation between the spectrum of certain closed ideals of $C_B(X)$ generated by $H_i$'s and the individual spectrums $\mathfrak{sp}(\overline{H_i})$'s.

Let $(P,\leq)$ be a partially ordered set. For a subset $A$ of $P$ we denote the least upper bound of $A$ by $\bigvee A$ (provided it exists).

Note that for a space $X$ the collection $\mathscr{H}(X)$ of all ideals of $C_B(X)$ (partially ordered with set theoretic inclusion $\subseteq$) is a partially ordered set in which for any subcollection $\{H_i:i\in I\}$ of $\mathscr{H}(X)$ we have
\[\bigvee_{i\in I}H_i=\bigg\langle\bigcup_{i\in I}H_i\bigg\rangle,\]
that is, the ideal in $C_B(X)$ generated by $\bigcup_{i\in I}H_i$.

\begin{theorem}\label{OJFS}
Let $X$ be a completely regular space. Let $\{H_i:i\in I\}$ be a non-empty collection of ideals in $C_B(X)$.
\begin{itemize}
\item[\rm(1)] Suppose that $H_i$ is non-vanishing for each $i\in I$. Then
  \[\mathfrak{sp}\bigg(\overline{\bigvee_{i\in I}H_i}\bigg)=\bigcup_{i\in I}\mathfrak{sp}\big(\overline{H_i}\big).\]
\item[\rm(2)] Suppose that ${\bigoplus_{i\in I}H_i}$ is non-vanishing. Then
  \[\mathfrak{sp}\bigg(\overline{\bigoplus_{i\in I}H_i}\bigg)=\bigoplus_{i\in I}\mathfrak{sp}\big(\overline{H_i}\big).\]
\item[\rm(3)] Suppose that $\bigcap_{i\in I}H_i$ is non-vanishing. Then
  \[\mathfrak{sp}\bigg(\overline{\bigcap_{i\in I}H_i}\bigg)=\mathrm{int}_{\,\mathfrak{sp}(C_B(X))}\bigg(\bigcap_{i\in I}\mathfrak{sp}\big(\overline{H_i}\big)\bigg).\]
\end{itemize}
Here the bar denotes the closure in $C_B(X)$.
\end{theorem}

\begin{proof}
(1). Suppose that $H_i$ is non-vanishing for each $i\in I$. Denote $H=\bigvee_{i\in I}H_i$. Note that $\overline{H}$ contains $H_i$ for some $i\in I$ (as it contains $H$) and $H_i$ is non-vanishing. Thus $\overline{H}$ is also non-vanishing. Now, an argument similar to the one we used to prove (\ref{KLDSD}) applies to shows that
\[\mathfrak{sp}\big(\overline{H}\big)=\bigcup_{i\in I}\mathfrak{sp}\big(\overline{H_i}\big).\]

(2). Suppose that $H=\bigoplus_{i\in I}H_i$ is non-vanishing. Then, an argument similar to the one we used to prove (\ref{RRC}) proves that
\[\mathfrak{sp}\big(\overline{H}\big)=\bigoplus_{i\in I}\mathfrak{sp}\big(\overline{H_i}\big).\]

(3). Suppose that $H=\bigcap_{i\in I}H_i$ is non-vanishing. First, note that $\mathfrak{sp}(C_B(X))=\beta X$. (This standard fact may also be deduced from Theorem \ref{TRES}.) Note that $\mathfrak{sp}(\overline{H})$ is an open subspace of $\beta X$ by Lemma \ref{KLGS} and Theorem \ref{TRES}, and it is contained in $\mathfrak{sp}(\overline{H_i})$ for each $i\in I$ (again by Theorem \ref{TRES}), as $\overline{H}\subseteq\overline{H_i}$. Thus
\[\mathfrak{sp}\big(\overline{H}\big)\subseteq\mathrm{int}_{\beta X}\bigg(\bigcap_{i\in I}\mathfrak{sp}\big(\overline{H_i}\big)\bigg).\]
To show the reverse inclusion, denote
\[T=\mathrm{int}_{\beta X}\bigg(\bigcap_{i\in I}\mathfrak{sp}\big(\overline{H_i}\big)\bigg)\]
and let $t\in T$. By the Urysohn lemma there is a continuous mapping $G:\beta X\rightarrow[0,1]$ such that $G(t)=1$ and $G|_{\beta X\setminus T}=\mathbf{0}$. Let $g=G|_X$. We verify that $g\in H$. Fix some $i\in I$. Let $n$ be a positive integer. Then $G^{-1}([1/n,1])\subseteq T\subseteq\mathfrak{sp}(\overline{H_i})$. For each $s\in G^{-1}([1/n,1])$ (using the representation of $\mathfrak{sp}(\overline{H_i})$ given in Theorem \ref{TRES}) we have
\[s\in\mathrm{cl}_{\beta X}|f_{n,s}|^{-1}\big((1,\infty)\big)\]
for some $f_{n,s}\in\overline{H_i}$. Thus, by an argument similar to the one given in the proof of (\ref{KLDSD}) we have $|h_{n,s}^\beta(s)|>1/2$ for some $h_{n,s}\in H_i$. Therefore
\[\big\{|h_{n,s}^\beta|^{-1}\big((1/2,\infty)\big):s\in G^{-1}\big([1/n,1]\big)\big\}\]
is an open cover for $G^{-1}([1/n,1])$. Now, since $G^{-1}([1/n,1])$ is compact we have
\begin{equation}\label{HDDGF}
G^{-1}\big([1/n,1]\big)\subseteq\bigcup_{j=1}^{k_n}|h_{n,s_j}^\beta|^{-1}\big((1/2,\infty)\big)
\end{equation}
for some $s_1,\ldots,s_{k_n}\in G^{-1}([1/n,1])$. Now, if we intersect the two sides of (\ref{HDDGF}) with $X$ it yields
\begin{equation}\label{HPKF}
g^{-1}\big([1/n,1]\big)\subseteq\bigcup_{j=1}^{k_n}|h_{n,s_j}|^{-1}\big((1/2,\infty)\big)\subseteq|h_n|^{-1}\big([1/4,\infty)\big)
=|u_n|^{-1}\big([1,\infty)\big)
\end{equation}
where $u_n=4h_n$,
\[h_n=\sum_{j=1}^{k_n}|h_{n,s_j}|^2.\]
Note that $h_n=\sum_{j=1}^{k_n}h_{n,s_j}\overline{h_{n,s_j}}\in H_i$, as $H_i$ is an ideal $C_B(X)$. In particular $u_n\in H_i$. Since (\ref{HPKF}) holds for every positive integer $n$, using Lemma \ref{JJHF}, there is a sequence $v_1,v_2,\ldots$ in $C_B(X)$ such that $u_nv_n\rightarrow g$. Observe that $u_1v_1,u_2v_2,\ldots$ is a sequence in $H_i$ (as $H_i$ is an ideal in $C_B(X)$). Therefore $g\in H_i$, as $H_i$ is closed in $C_B(X)$. But this holds for every $i\in I$, and thus $g\in H$. By Theorem \ref{TRES} we have
\[\mathrm{cl}_{\beta X}|2g|^{-1}\big((1,\infty)\big)\subseteq\mathfrak{sp}\big(\overline{H}\big)\]
In particular, this implies that $t\in\mathfrak{sp}(\overline{H})$, as $t\in|G|^{-1}((1/2,\infty))$ and \[|G|^{-1}\big((1/2,\infty)\big)\subseteq\mathrm{cl}_{\beta X}|g|^{-1}\big((1/2,\infty)\big)\]
by Lemma \ref{LKG}. (Observe that $G=g_\beta$.)
\end{proof}

A partially ordered set $(P,\leq)$ is called an \textit{upper semi-lattice} (\textit{complete upper semi-lattice}, respectively) if together with any two elements $a$ and $b$ in $P$ (any subset $A$ of $P$, respectively) it has their least upper bound $a\vee b$ ($\bigvee A$, respectively). \textit{Lower semi-lattices} and \textit{complete lower semi-lattices} are defined  analogously (with the greatest lower bounds $a\wedge b$ and $\bigwedge A$ in places of $a\vee b$ and $\bigvee A$, respectively). A partially ordered set is called a \textit{lattice} if it is both an upper semi-lattice and a lower semi-lattice. A lattice $L$ which has the largest element $\mathbf{1}$ is called \textit{compact} if for any subset $A$ of $L$ with $\bigvee A=\mathbf{1}$ we have $a_1\vee\cdots\vee a_n=\mathbf{1}$ for some $a_1,\ldots,a_n\in L$.

\begin{theorem}\label{ASHK}
Let $X$ be a completely regular space. Let $\mathscr{H}(X)$ be the collection of all non-vanishing closed ideals of $C_B(X)$ partially ordered with $\subseteq$. Then $\mathscr{H}(X)$ is a lattice which is also a complete upper semi-lattice. Indeed, for a subcollection $\{H_i:i\in I\}$ of $\mathscr{H}(X)$ we have
\[\bigvee_{i\in I}H_i=\overline{\bigg\langle\bigcup_{i\in I}H_i\bigg\rangle},\]
and for any two elements $G,H\in\mathscr{H}(X)$ we have
\[G\wedge H=G\cap H.\]
Moreover, the lattice $\mathscr{H}(X)$ is compact, that is, for any subcollection $\{H_i:i\in I\}$ of $\mathscr{H}(X)$ such that
\[\bigvee_{i\in I}H_i=C_B(X)\]
there are $i_1,\ldots,i_n\in I$ such that
\[\bigvee_{k=1}^nH_{i_k}=C_B(X).\]
Here the bar denotes the closure in $C_B(X)$.
\end{theorem}

\begin{proof}
That $\mathscr{H}(X)$ is a lattice and also a complete upper semi-lattice is clear. (Note that the intersection $G\cap H$ of two non-vanishing ideals $G$ and $H$ of $C_B(X)$ is also non-vanishing. To check this, let $x\in X$. Let $g\in G$ and $h\in H$ such that $g(x)\neq 0$ and $h(x)\neq 0$. Then $gh\in G\cap H$ and $(gh)(x)\neq 0$.) We show that the lattice $\mathscr{H}(X)$ is compact.

Note that $\mathscr{H}(X)$ has the largest element, namely, $C_B(X)$. Let $\{H_i:i\in I\}$ be a subcollection of $\mathscr{H}(X)$ such that
\[\bigvee_{i\in I}H_i=C_B(X).\]
Then, using Theorem \ref{OJFS}, we have
\[\beta X=\mathfrak{sp}\big(C_B(X)\big)=\mathfrak{sp}\bigg(\bigvee_{i\in I}H_i\bigg)=\bigcup_{i\in I}\mathfrak{sp}(H_i).\]
Note that $\mathfrak{sp}(H_i)$ is an open subspace of $\beta X$ for all $i\in I$. Therefore, by compactness, we have
\[\beta X=\bigcup_{k=1}^n\mathfrak{sp}(H_{i_k})\]
for some $i_1,\ldots,i_n\in I$. Observe that
\[\bigcup_{k=1}^n\mathfrak{sp}(H_{i_k})=\mathfrak{sp}\bigg(\bigvee_{k=1}^nH_{i_k}\bigg)\]
by Theorem \ref{OJFS}. Thus $\beta X=\mathfrak{sp}(H)$ ($=\lambda_HX$ by Theorem \ref{TRES}) where
\[H=\bigvee_{k=1}^nH_{i_k}.\]
To conclude the proof it suffices to prove that $H=C_B(X)$. Let $f\in C_B(X)$. Let $n$ be a positive integer. Then $\mathrm{cl}_{\beta X}|f|^{-1}([1/n,\infty))$ is a compact subspace of $\lambda_HX$, and therefore by Lemma \ref{GFS},
\[\mathrm{cl}_{\beta X}|f|^{-1}\big([1/n,\infty)\big)\subseteq\mathrm{cl}_{\beta X}|h_n|^{-1}\big((1,\infty)\big)\]
for some $h_n\in H$. Now, if we intersect the two sides of the above relation with $X$ we obtain
\[|f|^{-1}\big([1/n,\infty)\big)\subseteq|h_n|^{-1}\big([1,\infty)\big).\]
By Lemma \ref{JJHF} there is a sequence $g_1,g_2,\ldots$ in $C_B(X)$ with $g_nh_n\rightarrow f$. But then $f\in H$, as $H$ is closed in $C_B(X)$ and $g_1h_1,g_2h_2,\ldots$ is a sequence in $H$.
\end{proof}

Our concluding result in this section shows how the collection of all non-vanishing closed ideals of $C_B(X)$ can be made into one-to-one correspondence with the collection of all closed open bornologies of $X$.

Recall that a bornology $\mathfrak{B}$ in a set $X$ is an ideal in $X$ whose elements cover the whole $X$, that is, $X=\bigcup\mathfrak{B}$. Let $X$ be a space. A bornology $\mathfrak{B}$ in $X$ is called \textit{closed} (\textit{open}, respectively) if each element of $\mathfrak{B}$ is contained in a closed (open, respectively) element of $\mathfrak{B}$.

Let $(P,\leq)$ and $(Q,\leq)$ be partially ordered sets. A mapping $f:P\rightarrow Q$ is called an \textit{order-isomorphism} if $f$ is a bijection such that $f(x)\leq f(y)$, where $x,y\in P$, if and only if $x\leq y$. Also, $P$ and $Q$ are \textit{order-isomorphic} if there is an order-isomorphism between them.

\begin{theorem}\label{HGJJ}
Let $X$ be a normal space. Denote by $\mathscr{B}(X)$ the set of all closed open bornologies in $X$ and denote by $\mathscr{H}(X)$ the set of all non-vanishing closed ideals in $C_B(X)$, both partially ordered with $\subseteq$. Then $\mathscr{B}(X)$ and $\mathscr{H}(X)$ are order-isomorphic.
\end{theorem}

\begin{proof}
Define a mapping
\[\phi:\mathscr{H}(X)\longrightarrow\mathscr{B}(X)\]
by
\[\phi(H)=\big\langle |h|^{-1}\big((1,\infty)\big):h\in H\big\rangle\]
for any $H\in\mathscr{H}(X)$. We prove that $\phi$ is an order-isomorphism.

First, we need to check that $\phi$ is well defined. Let $H\in\mathscr{H}(X)$. By an argument similar to the one given in the proof of Theorem \ref{TRES} we have
\begin{eqnarray*}
\phi(H)&=&\big\{A:A\subseteq|h|^{-1}\big((1,\infty)\big)\mbox{ where }h\in H\big\}\\&=&\big\{A:A\subseteq|h|^{-1}\big((1/n,\infty)\big)\mbox{ where }h\in H\mbox{ and }n=1,2,\ldots\big\}.
\end{eqnarray*}
It now follows from the above representation that the ideal $\phi(H)$ of $X$ is a bornology in $X$ (as $H$ is non-vanishing) and $\phi(H)$ is both closed and open.

Observe that $\phi(G)\subseteq\phi(H)$ whenever $G,H\in\mathscr{H}(X)$ and $G\subseteq H$. That is, $\phi$ preserves order. Now, let $\phi(G)\subseteq\phi(H)$ where $G,H\in\mathscr{H}(X)$. Let $g\in G$. Let $n$ be a positive integer. Then $|g|^{-1}([1/n,\infty))\in\phi(H)$, as $|g|^{-1}([1/n,\infty))\in\phi(G)$, and thus
\[|g|^{-1}\big([1/n,\infty)\big)\subseteq|h_n|^{-1}\big([1,\infty)\big)\]
for some $h_n\in H$. Since this holds for each positive integer $n$, by Lemma \ref{JJHF} there is a sequence $f_1,f_2,\ldots$ in $C_B(X)$ such that $f_nh_n\rightarrow g$. But then $g\in H$, as $f_1h_1,f_2h_2,\ldots$ is a sequence in $H$ (since $H$ is an ideal in $C_B(X)$) and $H$ is closed in $C_B(X)$. That is $G\subseteq H$. This in particular shows that $\phi$ is injective.

To conclude the proof it suffices to show that $\phi$ is surjective. Let $\mathfrak{B}\in\mathscr{B}(X)$. Denote
\[H=\big\{f|_X:f\in C_0(\lambda_\mathscr{B}X)\big\}.\]
Note that the above notation makes sense, as $X\subseteq\lambda_\mathfrak{B}X$ by Lemma \ref{BBV}, since $X$ is locally null with respect to $\mathfrak{B}$ (because $\mathfrak{B}$ is a bornology in $X$ and it is open). We show that $H\in\mathscr{H}(X)$ and $\phi(H)=\mathfrak{B}$.

\begin{xclaim}
$H\in\mathscr{H}(X)$.
\end{xclaim}

\noindent \emph{Proof of the claim.} Let $f\in C_B(X)$ and $g,h\in H$. Then $g=g'|_X$ and $h=h'|_X$ for some $g',h'\in C_0(\lambda_\mathfrak{B}X)$. Note that $g-h=(g'-h')|_X\in H$ as $g'-h'\in C_0(\lambda_\mathfrak{B}X)$, and $gf=(g'\cdot f_\beta|_{\lambda_\mathfrak{B}X})|_X\in H$, as $g'\cdot f_\beta|_{\lambda_\mathfrak{B}X}\in C_0(\lambda_\mathfrak{B}X)$. That is, $H$ is an ideal in $C_B(X)$.

To show that $H$ is closed in $C_B(X)$, let $h_n\rightarrow f$ for some sequence $h_1,h_2,\ldots$ in $H$ and $f\in C_B(X)$. Then $h_n=h'_n|_X$ where $h'_n\in C_0(\lambda_\mathfrak{B}X)$ for each positive integer $n$. Note that $h'_1,h'_2,\ldots$ is a Cauchy sequence in $C_0(\lambda_\mathfrak{B}X)$, as $h_1,h_2,\ldots$ is a Cauchy sequence in $C_B(X)$ and $\|h'_n-h'_m\|=\|h_n-h_m\|$ for any positive integers $n$ and $m$ by (an argument similar to the one in the proof of) Lemma \ref{KJGFF}. Thus $h'_n\rightarrow h'$ for some $h'\in C_0(\lambda_\mathfrak{B}X)$ and therefore $h'_n|_X\rightarrow h'|_X$ by Lemma \ref{KJGFF}. This implies that $f=h'|_X\in H$.

To show that $H$ is non-vanishing, let $x\in X$. Then $x\in\lambda_\mathfrak{B}X$ (as $X\subseteq\lambda_\mathfrak{B}X$). Note that $\lambda_\mathfrak{B}X$ is an open subspace of $\beta X$. Let $f':\beta X\rightarrow[0,1]$ be a continuous mapping with $f'(x)=1$ and $f'|_{\beta X\setminus\lambda_\mathfrak{B}X}=\textbf{0}$. Let $f=f'|_{\lambda_\mathfrak{B}X}$. Then $f\in C_0(\lambda_\mathfrak{B}X)$, as
$f^{-1}([\epsilon,\infty))=f'^{-1}([\epsilon,\infty))$ is closed in $\beta X$ for each $\epsilon>0$, and is therefore compact. Thus $f|_X\in H$. Observe that $f|_X(x)=f'(x)\neq 0$.

\begin{xclaim}
$\phi(H)=\mathfrak{B}$.
\end{xclaim}

\noindent \emph{Proof of the claim.} Let $A\in\phi(H)$. Then $A\subseteq|h|^{-1}((1,\infty))$ for some $h\in H$. Let $h=f|_X$ where $f\in C_0(\lambda_\mathfrak{B}X)$. Then $|f|^{-1}([1,\infty))$ is a compact subspace of $\lambda_\mathfrak{B}X$ and is therefore closed in $\beta X$. Thus $\mathrm{cl}_{\beta X}|h|^{-1}([1,\infty))\subseteq|f|^{-1}([1,\infty))$, which implies that $\mathrm{cl}_{\beta X}|h|^{-1}([1,\infty))\subseteq\lambda_\mathfrak{B}X$. By Lemma \ref{HDHD} then $|h|^{-1}([1,\infty))$ is null, and therefore so is its subset $A$. This shows that $\phi(H)\subseteq\mathfrak{B}$.

Next, let $B\in\mathfrak{B}$. Then $\mathrm{cl}_XB\in\mathfrak{B}$, as $\mathfrak{B}$ is closed, and $\mathrm{cl}_XB\subseteq U$ for some null open subspace of $X$, as $\mathfrak{B}$ is open. Since $X$ is normal, there is a continuous mapping $f:X\rightarrow[0,1]$ such that $f|_{\mathrm{cl}_XB}=\mathbf{1}$ and $f|_{X\setminus U}=\mathbf{0}$. Let $g=f_\beta|_{\lambda_\mathfrak{B}X}$. We check that $g\in C_0(\lambda_\mathfrak{B}X)$. Note that by Lemma \ref{LKG} and Proposition \ref{KJHF} we have
\[f_\beta^{-1}\big((\epsilon,\infty)\big)\subseteq\mathrm{int}_{\beta X}\mathrm{cl}_{\beta X}f^{-1}\big((\epsilon,\infty)\big)\subseteq\mathrm{int}_{\beta X}\mathrm{cl}_{\beta X}U\subseteq\lambda_\mathfrak{B}X\]
for each $\epsilon>0$. In particular,
\[g^{-1}\big([\epsilon,\infty)\big)=\lambda_\mathfrak{B}X\cap f_\beta^{-1}\big([\epsilon,\infty)\big)=f_\beta^{-1}\big([\epsilon,\infty)\big)\]
is closed in $\beta X$ for each $\epsilon>0$, and is therefore compact. This implies that $f=g|_X\in H$. Thus $f^{-1}((1/2,\infty))\in\phi(H)$. In particular $B\in\phi(H)$, as $B\subseteq f^{-1}((1/2,\infty))$. This shows that $\mathfrak{B}\subseteq\phi(H)$.
\end{proof}

\section{Ideals in $C_B(X)$ arising from ideals in $X$ related to its topology}\label{KIJGD}

In this section for a space $X$ we consider the ideals $C^{\mathfrak I}_{00}(X)$ and $C^{\mathfrak I}_0(X)$ of $C_B(X)$ where the ideal ${\mathfrak I}$ arises from the topology of $X$. More specifically, for a space $X$ (which is sometimes required to satisfy certain separation axioms) and a topological property $\mathfrak{P}$ (which is subject to some mild requirements) we consider the ideal
\[{\mathfrak I}=\{A\subseteq X:\mathrm{cl}_XA\mbox{ has }\mathfrak{P}\}\]
of $X$. We then use the representation theorems which we proved in Part \ref{HFPG} to obtain information about the ideals $C^{\mathfrak I}_{00}(X)$ and $C^{\mathfrak I}_0(X)$ of $C_B(X)$. The specification we have made in this section will enable us to examine a few more properties from $C^{\mathfrak I}_{00}(X)$ and $C^{\mathfrak I}_0(X)$. As we will see, under certain conditions on the space $X$ (such as metrizability) and the topological property $\mathfrak{P}$ the ideals $C^{\mathfrak I}_{00}(X)$ and $C^{\mathfrak I}_0(X)$ coincide, the spectrum of $C^{\mathfrak I}_0(X)$ is countably compact and non-normal, and the vector space dimensions of $C^{\mathfrak I}_0(X)$ turns out to be computable in terms of a certain topological characteristic (the density) of the space $X$.

Now, we proceed with the formal treatment of the subject. Theorems \ref{JGF} and \ref{GKGF} improve results from \cite{Ko10} and \cite{Ko11}, respectively, Proposition \ref{HJGL} is from \cite{Ko10}, Theorem \ref{HGFLK} improves a result in \cite{Ko11}, Theorem \ref{CTTF} is actually the main result of \cite{Ko6} which is rephrased in the new context, Theorem \ref{GGFD} improves a result in \cite{Ko10}, and Theorems \ref{PTF} and \ref{RRGY} reproves results from \cite{Ko4} and \cite{Ko12}, respectively.

Topological properties in this section are all assumed to be non-empty, in a sense that, for a topological property  $\mathfrak{P}$ there always exists a space with $\mathfrak{P}$.

\begin{definition}\label{HHI}
Let $\mathfrak{P}$ be a topological property. Then
\begin{itemize}
  \item $\mathfrak{P}$ is \textit{closed hereditary}, if any closed subspace of a space having $\mathfrak{P}$ also has $\mathfrak{P}$.
  \item $\mathfrak{P}$ is \textit{preserved under finite} (resp. \textit{countable}) \textit{closed sums}, if any space which is expressible as a finite (resp. countable) union of its closed subspaces each having $\mathfrak{P}$ also has $\mathfrak{P}$.
\end{itemize}
\end{definition}

\begin{notation}\label{GGH}
Let $X$ be a space and let $\mathfrak{P}$ be a topological property. Denote
\[{\mathfrak I}_\mathfrak{P}=\{A\subseteq X:\mathrm{cl}_XA\mbox{ has }\mathfrak{P}\}.\]
\end{notation}

We may use the following lemma without explicitly referring to it.

\begin{lemma}\label{KGF}
Let $X$ be a space and let $\mathfrak{P}$ be a closed hereditary topological property preserved under finite closed sums.
Then $\mathfrak{I}_\mathfrak{P}$ is an ideal in $X$.
\end{lemma}

\begin{proof}
Note that $\mathfrak{I}_\mathfrak{P}$ is non-empty, as $\mathfrak{P}$ is so, and thus $\emptyset$ has $\mathfrak{P}$, as it is a closed subspace of any space with $\mathfrak{P}$ and $\mathfrak{P}$ is closed hereditary. Let $B\subseteq A$ with $A\in\mathfrak{I}_\mathfrak{P}$. Then $\mathrm{cl}_XB\subseteq\mathrm{cl}_XA$. Since $\mathrm{cl}_XA$ has $\mathfrak{P}$ and $\mathfrak{P}$ is closed hereditary, the closed subspace $\mathrm{cl}_XB$ of $\mathrm{cl}_XA$ also has $\mathfrak{P}$. That is $B\in\mathfrak{I}_\mathfrak{P}$. Next, let $A_i\in\mathfrak{I}_\mathfrak{P}$ for each $i=1,\ldots,n$. Note that
\[\mathrm{cl}_X(A_1\cup\cdots\cup A_n)=\mathrm{cl}_XA_1\cup\cdots\cup\mathrm{cl}_XA_n\]
and the latter has $\mathfrak{P}$, as $\mathrm{cl}_XA_i$ has $\mathfrak{P}$ for each $i=1,\ldots,n$ and $\mathfrak{P}$ is preserved under finite closed sums. Thus $A_1\cup\cdots\cup A_n\in\mathfrak{I}_\mathfrak{P}$.
\end{proof}

\begin{example}\label{GLL}
The list of topological properties $\mathfrak{P}$ which are closed hereditary and preserved under finite closed sums (thus satisfying the assumption of Lemma \ref{KGF}) is quite long and include almost all major covering properties (that is, topological properties described in terms of the existence of certain kinds of open subcovers or refinements of a given open cover of a certain type); among them are: (1) compactness (2) $[\theta,\kappa]$-compactness (in particular, countable compactness and the Lindel\"{o}f property) (3) paracompactness (4) metacompactness (5) countable paracompactness (6) subparacompactness (7) $\theta$-refinability (or submetacompactness) (8) the $\sigma$-para-Lindel\"{o}f property (9) $\delta\theta$-refinability (or the submeta-Lindel\"{o}f property) (10) weak $\theta$-refinability, and (11) weak $\delta\theta$-refinability. (See \cite{Bu}, \cite{Steph} and \cite{Va} for definitions.) These topological properties are all closed hereditary (this is obvious for (1), see Theorem 7.1 of \cite{Bu} for (3)--(4) and (6)--(11), Theorem 3.1 of \cite{Steph} for (2) and Exercise 5.2.B of \cite{E} for (5)) and are preserved under finite closed sums (this follows from the definition for (2) and is obvious for (1), see Theorems 7.3 and 7.4 of \cite{Bu} for (3)--(4) and (6)--(11), and Theorem 3.7.22 and Exercises 5.2.B and 5.2.G of \cite{E} for (5)).

There are however examples of topological properties which are not generally considered to be a covering property, yet they are closed hereditary and preserved under finite closed sums. We only mention $\alpha$-boundedness. (A space is called \textit{$\alpha$-bounded}, where $\alpha$ is an infinite cardinal, if every subspace of cardinality $\leq\alpha$ has compact closure.) That $\alpha$-boundedness is closed hereditary and preserved under finite closed sums follows easily from its definition.
\end{example}

\begin{definition}\label{KI}
Let $\mathfrak{P}$ be a topological property. A space $X$ is called \textit{locally-$\mathfrak{P}$} if each $x\in X$ has a neighborhood in $X$ with $\mathfrak{P}$.
\end{definition}

\begin{lemma}\label{KFFF}
Let $X$ be a regular space and let $\mathfrak{P}$ be a closed hereditary topological property preserved under finite closed sums. Then $X$ is locally null (with respect to the ideal $\mathfrak{I}_\mathfrak{P}$) if and only if $X$ is locally-$\mathfrak{P}$.
\end{lemma}

\begin{proof}
Note that $\mathfrak{I}_\mathfrak{P}$ is an ideal in $X$ by Lemma \ref{KGF}. Let $X$ be locally null and let $x\in X$. Then $x\in\mathrm{int}_XU$ for some $U\in\mathfrak{I}_\mathfrak{P}$. Thus $\mathrm{cl}_XU$ is a neighborhood of $x$ in $X$ with $\mathfrak{P}$. For the converse, let $x\in X$ have a neighborhood $V$ in $X$ with $\mathfrak{P}$. There exists an open neighborhood $W$ of $x$ in $X$ with $\mathrm{cl}_XW\subseteq V$. Now, $\mathrm{cl}_XW$ has $\mathfrak{P}$, as it is closed in $V$, $V$ has $\mathfrak{P}$ and $\mathfrak{P}$ is closed hereditary. Thus $W$ is null.
\end{proof}

We introduce the following notation for convenience.

\begin{notation}\label{KHFG}
Let $X$ be a space and let $\mathfrak{P}$ be a closed hereditary topological property preserved under finite closed sums. Denote
\[C^\mathfrak{P}_{00}(X)=C^{\mathfrak I_\mathfrak{P}}_{00}(X)\quad\mbox{and}\quad C^\mathfrak{P}_0(X)=C^{\mathfrak I_\mathfrak{P}}_0(X).\]
Also, denote
\[\lambda_\mathfrak{P}X=\lambda_{\mathfrak I_\mathfrak{P}}X\]
whenever $X$ is completely regular.
\end{notation}

\begin{theorem}\label{JGF}
Let $X$ be a space and let $\mathfrak{P}$ be a closed hereditary topological property preserved under finite closed sums. Then
\[C^\mathfrak{P}_{00}(X)=\big\{f\in C_B(X):\mathrm{supp}(f)\mbox{ has a closed neighborhood in $X$ with }\mathfrak{P}\big\},\]
and in particular,
\[C^\mathfrak{P}_{00}(X)=\big\{f\in C_B(X):\mathrm{supp}(f)\mbox{ has a neighborhood in $X$ with }\mathfrak{P}\big\},\]
if $X$ is normal. Furthermore,
\begin{itemize}
\item[\rm(1)] $C^\mathfrak{P}_{00}(X)$ is an ideal in $C_B(X)$.
\item[\rm(2)] Let $X$ be completely regular. Then $C^\mathfrak{P}_{00}(X)$ is non-vanishing if and only if $X$ is locally-$\mathfrak{P}$.
\item[\rm(3)] Let $X$ be completely regular and locally-$\mathfrak{P}$. Then $C^\mathfrak{P}_{00}(X)$ is unital if and only if $X$ has $\mathfrak{P}$.
\item[\rm(4)] Let $X$ be normal and locally-$\mathfrak{P}$. Then $C^\mathfrak{P}_{00}(X)$ is an ideal of $C_B(X)$ isometrically isomorphic to $C_{00}(Y)$ for some unique locally compact space $Y$, namely, for $Y=\lambda_\mathfrak{P}X$. Furthermore,
    \begin{itemize}
    \item $X$ is dense in $Y$.
    \item $C^\mathfrak{P}_{00}(X)$ is unital if and only if $X$ has $\mathfrak{P}$ if and only if $Y$ is compact.
    \end{itemize}
\end{itemize}
\end{theorem}

\begin{proof}
Let $f\in C_B(X)$. Suppose that $\mathrm{supp}(f)$ has a closed neighborhood $U$ in $X$ with $\mathfrak{P}$. Then $\mathrm{supp}(f)$ has a null neighborhood in $X$, namely, $U$ itself. On the other hand, if $f\in C^\mathfrak{P}_{00}(X)$, then $\mathrm{supp}(f)$ has a null neighborhood $V$ in $X$. Thus $\mathrm{cl}_X V$ is a closed neighborhood of $\mathrm{supp}(f)$ in $X$ with $\mathfrak{P}$. This shows the first representation of $C^\mathfrak{P}_{00}(X)$. Note that if $X$ is normal, then any neighborhood of $\mathrm{supp}(f)$ in $X$ contains a closed neighborhood of $\mathrm{supp}(f)$ in $X$. The second representation of $C^\mathfrak{P}_{00}(X)$ then follows, as $\mathfrak{P}$ is closed hereditary. The remaining assertions of the theorem follow from Lemmas \ref{KGF}, \ref{KFFF}, \ref{TTG}, \ref{BBV} and \ref{HGBV}, and Theorem \ref{UUS}.
\end{proof}

\begin{theorem}\label{HJG}
Let $X$ be a space and let $\mathfrak{P}$ be a closed hereditary topological property preserved under finite closed sums. Then
\[C^\mathfrak{P}_0(X)=\big\{f\in C_B(X):|f|^{-1}\big([1/n,\infty)\big)\mbox{ has }\mathfrak{P}\mbox{ for each }n\big\},\]
and in particular,
\[C^\mathfrak{P}_0(X)=\big\{f\in C_B(X):\mathrm{coz}(f)\mbox{ has }\mathfrak{P}\big\},\]
if $\mathfrak{P}$ is also preserved under countable closed sums. Furthermore,
\begin{itemize}
\item[\rm(1)] $C^\mathfrak{P}_0(X)$ is a closed ideal in $C_B(X)$ which contains $C^\mathfrak{P}_{00}(X)$.
\item[\rm(2)] Let $X$ be completely regular. Then $C^\mathfrak{P}_0(X)$ is non-vanishing if and only if $X$ is locally-$\mathfrak{P}$.
\item[\rm(3)] Let $X$ be completely regular and locally-$\mathfrak{P}$. Then $C^\mathfrak{P}_0(X)$ is unital if and only if $X$ has $\mathfrak{P}$.
\item[\rm(4)] Let $X$ be completely regular and locally-$\mathfrak{P}$. Then $C^\mathfrak{P}_0(X)$ is a closed ideal of $C_B(X)$ isometrically isomorphic to $C_0(Y)$ for some unique locally compact space $Y$, namely, for $Y=\lambda_\mathfrak{P}X$. In particular, $\lambda_\mathfrak{P}X$ is the spectrum of $C^\mathfrak{P}_0(X)$. Furthermore,
    \begin{itemize}
    \item $X$ is dense in $Y$.
    \item $C^\mathfrak{P}_{00}(X)$ is dense in $C^\mathfrak{P}_0(X)$.
    \item $C^\mathfrak{P}_0(X)$ is unital if and only if $X$ has $\mathfrak{P}$ if and only if $Y$ is compact.
    \end{itemize}
\end{itemize}
\end{theorem}

\begin{proof}
Note that for each positive integer $n$ the set $|f|^{-1}([1/n,\infty))$ (since it is closed in $X$) is null if and only if it has $\mathfrak{P}$. The first representation of $C^\mathfrak{P}_0(X)$ then follows. For the second representation of $C^\mathfrak{P}_0(X)$, let $f\in C_B(X)$. Note that
\[\mathrm{coz}(f)=\bigcup_{n=1}^\infty|f|^{-1}\big([1/n,\infty)\big).\]
Thus $\mathrm{coz}(f)$ has $\mathfrak{P}$ if $|f|^{-1}([1/n,\infty))$ does for each positive integer $n$ and $\mathfrak{P}$ is preserved under countable closed sums. Trivially, if $\mathrm{coz}(f)$ has $\mathfrak{P}$ then so does its closed subspace $|f|^{-1}([1/n,\infty))$ for each positive integer $n$, as $\mathfrak{P}$ is closed hereditary. The remaining assertions of the theorem follow from Lemmas \ref{KGF}, \ref{KFFF}, \ref{DGH}, \ref{KGV} and \ref{HFBY}, and Theorem \ref{UDR}.
\end{proof}

The following theorem is the first in a series of consecutive corollaries to Theorems \ref{JGF} and \ref{HJG}. It introduces conditions under which $C^\mathfrak{P}_0(X)=C^\mathfrak{P}_{00}(X)$ and simplifies the representations of $C^\mathfrak{P}_0(X)$ and $C^\mathfrak{P}_{00}(X)$. It also derive certain properties of the spectrum of $C^\mathfrak{P}_0(X)$. Examples of topological properties $\mathfrak{P}$ and $\mathfrak{Q}$ satisfying the assumption of the following theorem are given in Example \ref{GGK}. (See also Theorems \ref{CTTF} and \ref{HGFLK}.)

Let $X$ be a locally compact space. Recall that $C_0(X)=C_{00}(X)$ if and only if every $\sigma$-compact subspace of $X$ is contained in a compact subspace of $X$. (See Problem 7G.2 of \cite{GJ}.) In particular, $C_0(X)=C_{00}(X)$ implies that $X$ is countably compact. This will be used in the proof of the following.

\begin{theorem}\label{GKGF}
Let $\mathfrak{P}$ and $\mathfrak{Q}$ be topological properties such that
\begin{itemize}
\item $\mathfrak{P}$ and $\mathfrak{Q}$ are closed hereditary.
\item A space with both $\mathfrak{P}$ and $\mathfrak{Q}$ is Lindel\"{o}f.
\item $\mathfrak{P}$ is preserved under countable closed sums.
\item A space with $\mathfrak{Q}$ having a dense subspace with $\mathfrak{P}$ has $\mathfrak{P}$.
\end{itemize}
Let $X$ be a locally-$\mathfrak{P}$ space with $\mathfrak{Q}$.
\begin{itemize}
\item[\rm(1)] Let $X$ be regular. Then
\[C^\mathfrak{P}_0(X)=\big\{f\in C_B(X):\mathrm{supp}(f)\mbox{ has }\mathfrak{P}\big\}=C^\mathfrak{P}_{00}(X).\]
\item[\rm(2)] Let $X$ be normal. Then $C^\mathfrak{P}_0(X)$ is a closed ideal of $C_B(X)$ isometrically isomorphic to $C_0(Y)$ for some unique locally compact space $Y$, namely, for $Y=\lambda_\mathfrak{P}X$. In particular, $\lambda_\mathfrak{P}X$ is the spectrum of $C^\mathfrak{P}_0(X)$. Furthermore,
    \begin{itemize}
    \item $Y$ is countably compact.
    \item $C_0(Y)=C_{00}(Y)$.
    \item $C^\mathfrak{P}_0(X)$ is unital if and only if $X$ has $\mathfrak{P}$ if and only if $Y$ is compact.
    \end{itemize}
\end{itemize}
\end{theorem}

\begin{proof}
We first verify that $\mathfrak{I}_\mathfrak{P}$ is a $\sigma$-ideal in $X$. Note that $\mathfrak{I}_\mathfrak{P}$ is an ideal in $X$ by Lemma \ref{KGF}, as $\mathfrak{P}$ is closed hereditary and preserved under countable closed sums. Thus, we need only check that $\mathfrak{I}_\mathfrak{P}$ is closed under formation of countable unions. Let $A_n\in\mathfrak{I}_\mathfrak{P}$ for each positive integer $n$. Then
\[G=\bigcup_{n=1}^\infty\mathrm{cl}_XA_n\]
has $\mathfrak{P}$, as it is a countable union of its closed subspaces each with $\mathfrak{P}$ and $\mathfrak{P}$ is preserved under countable closed sums. Note that
\[H=\mathrm{cl}_X\bigg(\bigcup_{n=1}^\infty A_n\bigg)\]
has $\mathfrak{Q}$, as it is closed in $X$, $X$ has $\mathfrak{Q}$ and $\mathfrak{Q}$ is closed hereditary. Since $H$ contains $G$ as a dense subspace, using our assumption it follows that $H$ has $\mathfrak{P}$. Therefore
\[\bigcup_{n=1}^\infty A_n\in\mathfrak{I}_\mathfrak{P}.\]

(1). We first show that
\begin{equation}\label{FGH}
C^\mathfrak{P}_{00}(X)=\big\{f\in C_B(X):\mathrm{supp}(f)\mbox{ has }\mathfrak{P}\big\}.
\end{equation}
Let $f\in C^\mathfrak{P}_{00}(X)$. Then $\mathrm{supp}(f)$ has a null neighborhood in $X$, that is, $\mathrm{supp}(f)\subseteq\mathrm{int}_XU$ for some subspace $U$ of $X$ such that $\mathrm{cl}_XU$ has $\mathfrak{P}$. Therefore $\mathrm{supp}(f)$ has $\mathfrak{P}$, as it is closed in $\mathrm{cl}_XU$.

Next, let $f\in C_B(X)$ such that $\mathrm{supp}(f)$ has $\mathfrak{P}$. Note that $X$ is locally null by Lemma \ref{KFFF}, as it is locally-$\mathfrak{P}$. For each $x\in X$ let $U_x$ be a null neighborhood of $x$ in $X$. Then
\[\mathrm{supp}(f)\subseteq\bigcup_{x\in X}\mathrm{int}_XU_x.\]
Note that $\mathrm{supp}(f)$ has $\mathfrak{Q}$, as it closed in $X$ (and $X$ has $\mathfrak{Q}$ and $\mathfrak{Q}$ is closed hereditary). Since $\mathrm{supp}(f)$ has $\mathfrak{P}$, it is then by our assumption Lindel\"{o}f. Let $x_1,x_2,\ldots\in X$ such that
\[\mathrm{supp}(f)\subseteq\bigcup_{n=1}^\infty\mathrm{int}_XU_{x_n}.\]
Note that $\mathfrak{I}_\mathfrak{P}$ is a $\sigma$-ideal. Therefore $\bigcup_{n=1}^\infty U_{x_n}$ is a null neighborhood of $\mathrm{supp}(f)$ in $X$. Thus $f\in C^\mathfrak{P}_{00}(X)$. This shows (\ref{FGH}).

Now, we show that $C^\mathfrak{P}_0(X)=C^\mathfrak{P}_{00}(X)$. Since $C^\mathfrak{P}_{00}(X)\subseteq C^\mathfrak{P}_0(X)$ by Lemma \ref{DGH}, we only check that $C^\mathfrak{P}_0(X)\subseteq C^\mathfrak{P}_{00}(X)$. Let $f\in C^\mathfrak{P}_0(X)$. Then $|f|^{-1}([1/n,\infty))$ has $\mathfrak{P}$ for each positive integer $n$. But then $\mathrm{coz}(f)$ has $\mathfrak{P}$, as it is a countable union of its closed subspaces $|f|^{-1}([1/n,\infty))$ each with $\mathfrak{P}$ (and $\mathfrak{P}$ is closed hereditary). Note that $\mathrm{supp}(f)$ has $\mathfrak{Q}$ (and $X$ has $\mathfrak{Q}$ and $\mathfrak{Q}$ is closed hereditary. Therefore $\mathrm{supp}(f)$ has $\mathfrak{P}$, as it contains $\mathrm{coz}(f)$ as a dense subspace and $\mathrm{coz}(f)$ has $\mathfrak{P}$. Thus $f\in C^\mathfrak{P}_{00}(X)$ by (\ref{FGH}).

(2). Let $Y=\lambda_\mathfrak{P} X$. Let
\[\phi:C^\mathfrak{P}_0(X)\longrightarrow C_0(Y)\]
denote the isometric isomorphism as defined in the proof of Theorem \ref{UDR}. As it is observed in the proof Theorem \ref{UDR}, the mapping $\phi$, when restricted to $C^\mathfrak{P}_{00}(X)$, induces an isometric isomorphism between $C^\mathfrak{P}_{00}(X)$ and $C_{00}(Y)$. But then $C_0(Y)=C_{00}(Y)$, as $C^\mathfrak{P}_0(X)=C^\mathfrak{P}_{00}(X)$ by (1). This in particular implies that $Y$ is countably compact. The remaining statements also follow from Theorem \ref{UDR}.
\end{proof}

\begin{remark}\label{PF}
Observe that in Theorem \ref{GKGF} we indeed have
\[C^\mathfrak{P}_0(X)=C^{\mathfrak{P}+\mathfrak{Q}}_0(X)=C^{\mathfrak{P}+\mathfrak{Q}}_{00}(X)=C^\mathfrak{P}_{00}(X).\]
To show this, it suffices to observe that $\mathfrak{I}_\mathfrak{P}=\mathfrak{I}_{\mathfrak{P}+\mathfrak{Q}}$. It is obvious that $\mathfrak{I}_{\mathfrak{P}+\mathfrak{Q}}\subseteq\mathfrak{I}_\mathfrak{P}$. Let $A\in\mathfrak{I}_\mathfrak{P}$. Then $\mathrm{cl}_XA$ has $\mathfrak{P}$. But $\mathrm{cl}_XA$ also has $\mathfrak{Q}$, as it is closed in $X$, $X$ has $\mathfrak{Q}$ and $\mathfrak{Q}$ is closed hereditary. Thus $A\in\mathfrak{I}_{\mathfrak{P}+\mathfrak{Q}}$.
\end{remark}

\begin{remark}\label{GGJ}
Note that in Theorem \ref{GKGF} the space $Y$ is non-compact if and only if $X$ is non-$\mathfrak{P}$. In particular, if $X$ is non-$\mathfrak{P}$, since $Y$ is countably compact, then $Y$ is non-$\mathfrak{R}$ for any topological property $\mathfrak{R}$ such that
\[\mbox{$\mathfrak{R}$ $+$ countable compactness $\rightarrow$ compactness}.\]
The list of such topological properties $\mathfrak{R}$ is quite wide, including topological properties such as the Lindel\"{o}f property, paracompactness, realcompactness, metacompactness, subparacompactness, submetacompactness (or $\theta$-refinability), the meta-Lindel\"{o}f property, the submeta-Lindel\"{o}f property (or $\delta\theta$-refinability), weak submetacompactness (or weak $\theta$-refinability) and the weak submeta-Lindel\"{o}f property (or weak $\delta\theta$-refinability) among others. (See Parts 6.1 and 6.2 of \cite{Va}.)
\end{remark}

In the following we give examples of pairs of topological properties $\mathfrak{P}$ and $\mathfrak{Q}$ satisfying the assumption of Theorem \ref{GKGF}.

\begin{example}\label{GGK}
Let $\mathfrak{P}$ be the Lindel\"{o}f property and let $\mathfrak{Q}$ be either metrizability or paracompactness. Note that the Lindel\"{o}f property is closed hereditary (see Theorem 3.8.4 of \cite{E}) and is preserved under countable closed sums, as obviously, any space which is the union of a countable number of its Lindel\"{o}f subspaces is Lindel\"{o}f. Note that any subspace of a metrizable space is metrizable and a closed subspace of a paracompact space is paracompact. (See Corollary 5.1.29 of \cite{E}.) Also, any paracompact space with a dense Lindel\"{o}f space is Lindel\"{o}f (see Theorem 5.1.25 of \cite{E}), since any metrizable space is paracompact, it then follows that any metrizable space with a dense Lindel\"{o}f space is Lindel\"{o}f. Therefore, the pair $\mathfrak{P}$ and $\mathfrak{Q}$ satisfy the assumption of Theorem \ref{GKGF}.

Note that in the realm of metrizable spaces the Lindel\"{o}f property coincides with separability and second countability; denote the latter two by $\mathfrak{S}$ and $\mathfrak{C}$, respectively, and denote the Lindel\"{o}f property by $\mathfrak{L}$. Thus, if $X$ is a metrizable space, then
\[\mathfrak{I}_\mathfrak{L}=\mathfrak{I}_\mathfrak{S}=\mathfrak{I}_\mathfrak{C},\]
which implies that
\[C^\mathfrak{L}_{00}(X)=C^\mathfrak{S}_{00}(X)=C^\mathfrak{C}_{00}(X)\quad\mbox{and}\quad C^\mathfrak{L}_0(X)=C^\mathfrak{S}_0(X)=C^\mathfrak{C}_0(X).\]
\end{example}

Observe that for a paracompact locally-$\mathfrak{P}$ space $X$, where $\mathfrak{P}$ is the Lindel\"{o}f property, we have
\[C^\mathfrak{P}_{00}(X)=\big\{f\in C_B(X):\mathrm{supp}(f)\mbox{ has }\mathfrak{P}\big\},\]
by Theorem \ref{GKGF} (and Example \ref{GGK}). As we see in the following, for the case when $\mathfrak{P}$ is the Lindel\"{o}f property the above representation remains true even if one requires $X$ to be only normal.

\begin{proposition}\label{HJGL}
Let $X$ be a normal locally-$\mathfrak{P}$ space, where $\mathfrak{P}$ is the Lindel\"{o}f property. Then
\[C^\mathfrak{P}_{00}(X)=\big\{f\in C_B(X):\mathrm{supp}(f)\mbox{ has }\mathfrak{P}\big\}.\]
\end{proposition}

\begin{proof}
It is trivial that $\mathrm{supp}(f)$ has $\mathfrak{P}$ for any $f\in C^\mathfrak{P}_{00}(X)$, as $\mathrm{supp}(f)$ has a neighborhood in $X$ with $\mathfrak{P}$, and $\mathfrak{P}$ is closed hereditary.

Next, let $f\in C_B(X)$ such that $\mathrm{supp}(f)$ has $\mathfrak{P}$. For each $x\in X$, let $U_x$ be a neighborhood of $x$ in $X$ with $\mathfrak{P}$. Then
\[\mathrm{supp}(f)\subseteq\mathrm{int}_XU_{x_1}\cup\mathrm{int}_X U_{x_2}\cup\cdots\subseteq\mathrm{int}_X W,\]
for some $x_1,x_2,\ldots\in\ X$ and $W=U_{x_1}\cup U_{x_2}\cup\cdots$. Since $X$ is normal, there exists an open neighborhood $V$ of $\mathrm{supp}(f)$ in $X$ whose closure $\mathrm{cl}_XV$ is contained in $\mathrm{int}_XW$. In particular, $\mathrm{cl}_X V$ has $\mathfrak{P}$, as it is closed in $W$, $W$ has $\mathfrak{P}$ (since it is a countable union of its subspace each with $\mathfrak{P}$) and $\mathfrak{P}$ is closed hereditary. Therefore $f\in C^\mathfrak{P}_{00}(X)$.
\end{proof}

In the next theorem we consider Theorem \ref{GKGF} when $\mathfrak{P}$ and $\mathfrak{Q}$ are the Lindel\"{o}f property and paracompactness, respectively. Note that the space $Y$ in Theorem \ref{GKGF} is non-compact whenever $X$ is non-$\mathfrak{P}$, however, as we see in the next theorem for this specific choice of the pair $\mathfrak{P}$ and $\mathfrak{Q}$ the space $Y$ is even non-normal whenever $X$ is non-$\mathfrak{P}$. An important corollary of this theorem is its follow up result (Theorem \ref{CTTF}) in which $\mathfrak{P}$ and $\mathfrak{Q}$ are further specialized to be separability and metrizability, respectively. (Note that metrizable spaces are paracompact and seprability coincides with the Lindel\"{o}f property in the class of metrizable spaces.) As we see this further specialization of $\mathfrak{P}$ and $\mathfrak{Q}$ allows us even to determine the vector space dimensions of $C^\mathfrak{P}_0(X)$ in terms of a certain topological characteristic (the density) of the space $X$. Before we proceed with the theorem we need some preliminaries.

Let $X$ be a locally compact paracompact space. By Theorem 5.1.27 of \cite{E} the space $X$ has a representation as
\[X=\bigcup_{i\in I}X_i,\]
where $X_i$'s are disjoint Lindel\"{o}f closed and open subspaces of $X$ and $I$ is an index set. An inspection of the proof (of Theorem 5.1.27 of \cite{E}) however reveals that the assumption that $X$ is locally compact may be replaced by the weaker assumption that $X$ is locally Lindel\"{o}f. That is, the above representation of $X$ exists whenever $X$ is locally Lindel\"{o}f and paracompact. Note that $I$ is uncountable if $X$ is non-Lindel\"{o}f.

Let $D$ be an uncountable set equipped with the discrete topology. Let $D_\lambda$ denote the subspace of $\beta D$ consisting of all elements which are in the closure in $D$ of countable subsets of $D$. It is known that there exists a continuous (2-valued) mapping $f:D_\lambda\setminus D\rightarrow[0,1]$ which is not continuously extendible over $\beta D\setminus D$. (See \cite{W}.) In particular, this implies that $D_\lambda$ is not normal. (To see this, suppose the contrary. Note that $D_\lambda\setminus D$ is closed in $D_\lambda$, as $D$ is open in $\beta D$, since $D$ is locally compact. By the Tietze--Urysohn extension theorem, $f$ is extendible to a continuous bounded mapping over $D_\lambda$, and therefore, over the whole $\beta D_\lambda$; note that $\beta D_\lambda=\beta D$, as $D\subseteq D_\lambda\subseteq\beta D$. This, however, is not possible.)

Observe that if $X$ is a space and $D$ is a subspace of $X$, then
\[U\cap\mathrm{cl}_XD=\mathrm{cl}_X(U\cap D)\]
for every closed and open subspace $U$ of $X$.

\begin{theorem}\label{HGFLK}
Let $X$ be a locally Lindel\"{o}f paracompact space. Let
\[C_L(X)=\big\{f\in C_B(X):\mathrm{supp}(f)\mbox{ is Lindel\"{o}f}\,\big\}.\]
Then $C_L(X)$ is a closed ideal of $C_B(X)$ isometrically isomorphic to $C_0(Y)$ for some unique locally compact space $Y$, namely, for
\[Y=\bigcup\{\mathrm{cl}_{\beta X}L:L\subseteq X\mbox{ is Lindel\"{o}f}\,\}.\]
In particular, $Y$ is the spectrum of $C_L(X)$. Furthermore,
\begin{itemize}
\item[\rm(1)] $Y$ is countably compact.
\item[\rm(2)] $Y$ is normal if and only if $Y$ is compact if and only if $X$ is Lindel\"{o}f if and only if $C_L(X)$ is unital.
\item[\rm(3)] $C_0(Y)=C_{00}(Y)$.
\end{itemize}
\end{theorem}

\begin{proof}
Let $\mathfrak{P}$ be the Lindel\"{o}f property and let $\mathfrak{Q}$ be paracompactness. By Example \ref{GGK}, the pair $\mathfrak{P}$ and $\mathfrak{Q}$ satisfy the assumption of Theorem \ref{GKGF}. Observe that $X$ has a representation
\[X=\bigcup_{i\in I}X_i,\]
where $X_i$'s are disjoint Lindel\"{o}f closed and open subspaces of $X$ and $I$ is an index set. For convenience, denote
\[H_J=\bigcup_{i\in J}X_i\]
for any $J\subseteq I$. Observe that each $H_J$ is closed and open in $X$, and thus it has a closed and open closure in $\beta X$. Also, as we now check
\begin{equation}\label{DSY}
\lambda_\mathfrak{P} X=\bigcup\{\mathrm{cl}_{\beta X}H_J:J\subseteq I\mbox{ is countable}\}.
\end{equation}
Let $C\in\mathrm{Coz}(X)$ such that $\mathrm{cl}_X C$ has a neighborhood $U$ in $X$ such that $\mathrm{cl}_X U$ is Lindel\"{o}f. Then $\mathrm{cl}_X C$ itself is Lindel\"{o}f, as it is closed in $\mathrm{cl}_X U$, and therefore $\mathrm{cl}_X C\subseteq H_J$ for some countable $J\subseteq I$. Thus $\mathrm{cl}_{\beta X}C\subseteq\mathrm{cl}_{\beta X}H_J$. On the other hand, if $J\subseteq I$ is countable, then $H_J$ is a cozero-set in $X$, as it is closed and open in $X$, and it is Lindel\"{o}f, as it is a countable union of its Lindel\"{o}f subspaces. Since $\mathrm{cl}_{\beta X}H_J$ is open in $\beta X$ we have
\[\mathrm{cl}_{\beta X}H_J=\mathrm{int}_{\beta X}\mathrm{cl}_{\beta X}H_J\subseteq\lambda_\mathfrak{P} X.\]
This shows (\ref{DSY}). Note that
\[\lambda_\mathfrak{P} X=\bigcup\{\mathrm{cl}_{\beta X}L:L\subseteq X\mbox{ is Lindel\"{o}f}\};\]
by (\ref{DSY}); simply observe that any Lindel\"{o}f subspace of $X$ is contained in $H_J$ for some countable $J\subseteq I$, and on the other hand, $H_J$ is Lindel\"{o}f for any countable $J\subseteq I$, as it is a countable union of its Lindel\"{o}f subspaces.

Note that (3) implies (1); we prove (3). (Observe that (1) and (3) follow from Theorem \ref{GKGF}; the proof given here, however, is made easily to be independent from Theorem \ref{GKGF}.)

(3). It suffices to check that every $\sigma$-compact subspace of $\lambda_\mathfrak{P} X$ is contained in a compact subspace of $\lambda_\mathfrak{P} X$. Let
\[A=\bigcup_{n=1}^\infty A_n,\]
where $A_n$ is compact for each positive integer $n$, be a $\sigma$-compact subspace of $\lambda_\mathfrak{P} X$. Using (\ref{DSY}), for each positive integer $n$, by compactness we have
\begin{equation}\label{DJB}
A_n\subseteq\mathrm{cl}_{\beta X}H_{J_1}\cup\cdots\cup\mathrm{cl}_{\beta X}H_{J_{k_n}}
\end{equation}
for some countable $J_1,\ldots,J_{k_n}\subseteq I$. Let
\begin{equation}\label{FFB}
J=\bigcup_{n=1}^\infty(J_{k_1}\cup\cdots\cup J_{k_n}).
\end{equation}
Then $J$ is countable and $A\subseteq\mathrm{cl}_{\beta X}H_J$. That is, $A$ is contained in the compact subspace $\mathrm{cl}_{\beta X}H_J$ of $\lambda_\mathfrak{P} X$.

(2). Let $x_i\in X_i$ for each $i\in I$. Then $D=\{x_i:i\in I\}$ is a closed discrete subspace of $X$, and since $X$ is non-Lindel\"{o}f, it is uncountable. Suppose to the contrary that $\lambda_\mathfrak{P} X$ is normal. Then, using (\ref{DSY}), it follows that
\[\lambda_\mathfrak{P} X\cap\mathrm{cl}_{\beta X}D=\bigcup\{\mathrm{cl}_{\beta X}H_J\cap\mathrm{cl}_{\beta X}D:J\subseteq I\mbox{ is countable}\}\]
is normal, as it is closed in $\lambda_\mathfrak{P} X$. Now, let $J\subseteq I$ be countable. Since $\mathrm{cl}_{\beta X}H_J$ is closed and open in $\beta X$ (using the observation preceding the statement of the theorem) we have
\[\mathrm{cl}_{\beta X}H_J\cap\mathrm{cl}_{\beta X}D=\mathrm{cl}_{\beta X}(\mathrm{cl}_{\beta X}H_J\cap D)=\mathrm{cl}_{\beta X}(H_J\cap D)=\mathrm{cl}_{\beta X}\big(\{x_i:i\in J\}\big).\]
But $\mathrm{cl}_{\beta X}D=\beta D$, as $D$ is closed in $X$ and $X$ is normal. Therefore
\[\mathrm{cl}_{\beta X}\big(\{x_i:i\in J\}\big)=\mathrm{cl}_{\beta X}\big(\{x_i:i\in J\}\big)\cap\mathrm{cl}_{\beta X}D=\mathrm{cl}_{\beta D}\big(\{x_i:i\in J\}\big).\]
Thus
\[\lambda_\mathfrak{P} X\cap\mathrm{cl}_{\beta X}D=\bigcup\{\mathrm{cl}_{\beta D}E:E\subseteq D\mbox{ is countable}\}=D_\lambda,\]
contradicting the fact that $D_\lambda$ is not normal.
\end{proof}

\begin{remark}\label{KHGDF}
A point $x$ of a space $X$ is called a \textit{complete accumulation point} of a subspace $A$ of $X$ if $|U\cap A|=|A|$ for every neighborhood $U$ of $x$ in $X$. A space $X$ is called \textit{linearly Lindel\"{o}f} if every linearly ordered (by set theoretic inclusion $\subseteq$) open cover of $X$ has a countable subcover, or equivalently, if every uncountable subspace of $X$ has a complete accumulation point in $X$. As it is observed in \cite{Ko5}, the notions of $\sigma$-compactness, the Lindel\"{o}f property and the linearly Lindel\"{o}f property all coincide in the class of locally compact paracompact spaces; this we will now show. Since $\sigma$-compactness and the Lindel\"{o}f property are identical in the realm of locally compact spaces (see Exercise 3.8.C of \cite{E}) and by the definitions it is clear that the Lindel\"{o}f property implies the linearly Lindel\"{o}f property, we check only that in any locally compact paracompact space the linearly Lindel\"{o}f property implies the linearly Lindel\"{o}f property. Let $X$ be a locally compact paracompact space. Assume a representation for $X$ as
\[X=\bigcup_{i\in I}X_i,\]
where $X_i$'s are disjoint Lindel\"{o}f open subspaces of $X$ and $I$ is an index set. Suppose that $X$ is not Lindel\"{o}f. Then $I$ is uncountable. Let $A =\{x_i:i\in I\}$, where $x_i\in X_i$ for each $i\in I$. Then $A$ is an uncountable subspace of $X$ with no accumulation point in $X$, as any accumulation point of $A$ should be contained in $X_i$ for some $i\in I$ and $X_i$ is an open subspace of $X$ which intersects $A$ in only one element. Thus, $X$ cannot be linearly Lindel\"{o}f either.

In Theorem \ref{HGFLK} for a locally compact paracompact space (more generally, for a locally Lindel\"{o}f paracompact space) $X$ we have considered the set of all $f\in C_B(X)$ such that $\mathrm{supp}(f)$ is Lindel\"{o}f. Since local compactness and paracompactness are both closed hereditary, $\mathrm{supp}(f)$, for each $f\in C_B(X)$, is locally compact and paracompact, and thus is Lindel\"{o}f if and only if it is linearly Lindel\"{o}f if and only if it is $\sigma$-compact. In other words, Theorem \ref{HGFLK} holds true (provided that we assume that $X$ is locally compact and paracompact) if one replaces the Lindel\"{o}f property by either the linearly Lindel\"{o}f property or $\sigma$-compactness.
\end{remark}

As we have already pointed out our next theorem further specializes the topological properties considered in Theorem \ref{HGFLK} to derive further results. We need the following preliminaries.

The \textit{density} of a space $X$, denoted by $\mathrm{d}(X)$, is defined by
\[\mathrm{d}(X)=\min\big\{|D|:D\mbox{ is dense in }X\big\}+\aleph_0.\]
In particular, a space $X$ is separable if and only if $\mathrm{d}(X)=\aleph_0$. Observe that in any metrizable space separability coincides with second countability (and the Lindel\"{o}f property); thus any subspace of a separable metrizable space is separable. In particular, locally separable metrizable spaces are locally Lindel\"{o}f. Note that metrizable spaces are paracompact. It now follows from our discussion preceding Theorem \ref{HGFLK} that any locally separable metrizable space $X$ can be represented as a disjoint union
\[X=\bigcup_{i\in I}X_i,\]
where $I$ is an index set, and $X_i$ is a separable (and thus Lindel\"{o}f) closed and open subspace of $X$ for each $i\in I$. (This is indeed a theorem of Alexandroff; see Problem 4.4.F of \cite{E}.) Note that $\mathrm{d}(X)=|I|$, if $I$ is infinite.

Let $I$ be an infinite set. By a theorem of Tarski there exists a collection $\mathfrak{I}$ of cardinality $|I|^{\aleph_0}$ which consists of countable infinite subsets of $I$ such that the intersection of any two distinct elements of $\mathfrak{I}$ is finite. (See \cite{Ho}.) Note that the collection of all subsets of cardinality at most $\mathfrak{m}$ in a set of cardinality $\mathfrak{n}\geq\mathfrak{m}$ is of cardinality at most $\mathfrak{n}^\mathfrak{m}$.

\begin{theorem}\label{CTTF}
Let $X$ be a locally separable metrizable space. Let
\[C_S(X)=\big\{f\in C_B(X):\mathrm{supp}(f)\mbox{ is separable}\big\}.\]
Then $C_S(X)$ is a closed ideal of $C_B(X)$ isometrically isomorphic to $C_0(Y)$ for some unique locally compact space $Y$, namely, for
\[Y=\bigcup\{\mathrm{cl}_{\beta X}S:S\subseteq X\mbox{ is separable}\}.\]
In particular, $Y$ is the spectrum of $C_S(X)$. Furthermore,
\begin{itemize}
\item[\rm(1)] $Y$ is countably compact.
\item[\rm(2)] $Y$ is normal if and only if $Y$ is compact if and only if $X$ is separable if and only if $C_S(X)$ is unital.
\item[\rm(3)] $C_0(Y)=C_{00}(Y)$.
\item[\rm(4)] $\dim C_S(X)=\mathrm{d}(X)^{\aleph_0}$, if $X$ is non-separable.
\end{itemize}
\end{theorem}

\begin{proof}
Note that metrizable spaces are paracompact and separability coincides with the Lindel\"{o}f property in the class of metrizable spaces. Thus (1)--(3) follow from Theorem \ref{HGFLK}. We prove (4).

(4). Since $X$ is non-separable, $I$ is infinite and $\mathrm{d}(X)=|I|$. Let $\mathfrak{I}$ be a collection of cardinality $|I|^{\aleph_0}$ consisting of countable infinite subsets of $I$, such that the intersection of any two distinct elements of $\mathfrak{I}$ is finite. Let $f_J=\chi_{H_J}$ for any $J\in\mathfrak{I}$. Note that no element in
\[\mathscr{F}=\{f_J:J\in\mathfrak{I}\}\]
is a linear combination of other elements (as each element of $\mathfrak{I}$ is infinite and each pair of distinct elements of $\mathfrak{I}$ has finite intersection). Observe that $\mathscr{F}$ is of cardinality $|\mathfrak{I}|$. Thus
\[\dim C_S(X)\geq|\mathfrak{I}|=|I|^{\aleph_0}=\mathrm{d}(X)^{\aleph_0}.\]
On the other hand, if $f\in C_S(X)$, then $\mathrm{supp}(f)$ is Lindel\"{o}f and thus $\mathrm{supp}(f)\subseteq H_J$, where $J\subseteq I$ is countable; therefore, we may assume that $f\in C_B(H_J)$. Conversely, if $J\subseteq I$ is countable, then each element of $C_B(H_J)$ can be extended trivially to an element of $C_S(X)$ (by defining it to be identically $0$ elsewhere). Thus $C_S(X)$ may be viewed as the union of all $C_B(H_J)$, where $J$ runs over all countable subsets of $I$. Note that if $J\subseteq I$ is countable, then $H_J$ is separable; thus any element of $C_B(H_J)$ is determined by its value on a countable set. This implies that for each countable $J\subseteq I$, the set $C_B(H_J)$ is of cardinality at most $\mathfrak{c}^{\aleph_0}=2^{\aleph_0}$. There are at most $|I|^{\aleph_0}$ countable $J\subseteq I$. Therefore
\begin{eqnarray*}
\dim C_S(X)\leq\big|C_S(X)\big|&\leq& \Big|\bigcup\big\{C_B(H_J):J\subseteq I\mbox{ is countable}\big\}\Big|\\&\leq& 2^{\aleph_0}\cdot|I|^{\aleph_0}=|I|^{\aleph_0}=\mathrm{d}(X)^{\aleph_0}.
\end{eqnarray*}
\end{proof}

\begin{remark}\label{OLJDF}
As it is remarked previously, the three notions of separability, second countability and the Lindel\"{o}f property coincide in the class of metrizable spaces. In particular, Theorem \ref{CTTF} holds true if one replaces separability by either second countability or the Lindel\"{o}f property.
\end{remark}

The aim of the next theorem is to insert a chain consisting of certain types of closed ideals of $C_B(X)$ between $C_0(X)$ and $C_B(X)$ if $X$ satisfies certain requirements; the length of such a chain will be determined by a certain topological characteristic (the Lindel\"{o}f number) of the space $X$. We need some preliminaries.

The \textit{Lindel\"{o}f number} of a space $X$, denoted by $\ell(X)$, is defined by
\[\ell(X)=\min\{\mathfrak{n}:\mbox{any open cover of $X$ has a subcover of cardinality}\leq\mathfrak{n}\}+\aleph_0.\]
In particular, a space $X$ is Lindel\"{o}f if and only if $\ell(X)=\aleph_0$. As we have observed previously any locally Lindel\"{o}f paracompact space $X$ may be represented as
\[X=\bigcup_{i\in I}X_i,\]
where $X_i$'s are disjoint Lindel\"{o}f closed and open subspaces of $X$ and $I$ is an index set. Note that $\ell(X)=|I|$ if $I$ is infinite.

A regular space is said to be \textit{$\mu$-Lindel\"{o}f}, where $\mu$ is an infinite cardinal, if every open cover of it contains a subcover of cardinality $\leq\mu$. The $\mu$-Lindel\"{o}f property gets weaker as $\mu$ gets larger. The Lindel\"{o}f property is identical to the $\aleph_0$-Lindel\"{o}f property; in particular, every Lindel\"{o}f space is $\mu$-Lindel\"{o}f.

\begin{theorem}\label{GGFD}
Let $X$ be a non-Lindel\"{o}f locally Lindel\"{o}f paracompact space. Then there is a chain
\[C_0(X)\subsetneqq H_0\subsetneqq H_1\subsetneqq\cdots\subsetneqq H_\lambda=C_B(X)\]
of closed ideals of $C_B(X)$ such that $H_\mu$, for each $\mu<\lambda$, is isometrically isomorphic to
\[C_0(Y_\mu)=C_{00}(Y_\mu),\]
where $Y_\mu$ is a locally compact countably compact space which contains $X$ densely. Furthermore,
\[\aleph_\lambda=\ell(X).\]
\end{theorem}

\begin{proof}
Let $\mu$ be an ordinal. Let $\mathfrak{P}_\mu$ be the $\aleph_\mu$-Lindel\"{o}f property and denote
\[H_\mu=C^{\mathfrak{P}_\mu}_0(X).\]
Observe that $X$ is locally Lindel\"{o}f and thus locally $\aleph_\mu$-Lindel\"{o}f. It also follows from the definitions that the $\aleph_\mu$-Lindel\"{o}f property is closed hereditary and is preserved under finite closed sums. Theorem \ref{HJG} now implies that $H_\mu$ is a closed ideal in $C_B(X)$ which is isometrically isomorphic to $C_0(Y_\mu)$ with
\[Y_\mu=\lambda_{\mathfrak{P}_\mu}X.\]
By Theorem \ref{HJG} the space $Y_\mu$ is locally compact and contains $X$ as a dense subspace. We show that
\begin{equation}\label{KJHD}
C_0(Y_\mu)=C_{00}(Y_\mu);
\end{equation}
this in particular proves that $X$ is countably compact. To do this, it suffices to check that every $\sigma$-compact subspace of $Y_\mu$ is contained in a compact subspace of $Y_\mu$. Observe that the space $X$ may be represented as
\[X=\bigcup_{i\in I}X_i,\]
where $X_i$'s are disjoint Lindel\"{o}f closed and open subspaces of $X$ and $I$ is an index set. Note that if $J\subseteq I$ is of cardinality $\leq\aleph_\mu$, then
\[\bigcup_{i\in J}X_i\]
is $\aleph_\mu$-Lindel\"{o}f, and conversely, any $\aleph_\mu$-Lindel\"{o}f subspace of $X$ is contained in such a subspace. An argument similar to the one we have given in the proof of Theorem \ref{HGFLK} shows that
\[\lambda_{\mathfrak{P}_\mu}X=\bigcup\bigg\{\mathrm{cl}_{\beta X}\bigg(\bigcup_{i\in J}X_i\bigg):J\subseteq I\mbox{ is of cardinality }\leq\aleph_\mu\bigg\}.\]
Using the above representation and arguing as in the proof of Theorem \ref{HGFLK}, it follows that every $\sigma$-compact subspace of $Y_\mu$ is contained in a compact subspace of $Y_\mu$. This proves (\ref{KJHD}).

To conclude the proof we verify that $H_\mu$'s ascend as $\mu$ increases. It is clear that $H_\mu\subseteq H_\kappa$, whenever $\mu\leq\kappa$, as the $\aleph_\mu$-Lindel\"{o}f property is stronger than the $\aleph_\kappa$-Lindel\"{o}f property. Let $\lambda$ be such that $\aleph_\lambda=\ell(X)$. Note that $\ell(X)=|I|$. In particular, $X$ is $\aleph_\lambda$-Lindel\"{o}f, as it is the union of $\aleph_\lambda$ many of its Lindel\"{o}f subspaces. Therefore $H_\lambda=C_B(X)$, as $\mathrm{supp}(f)$ is $\aleph_\lambda$-Lindel\"{o}f for any $f\in C_B(X)$, since it is closed in $X$. Next, we show that all chain inclusions are proper. Note that $C_0(X)\subsetneqq H_0$, as for the characteristic mapping $\chi_A$, where
\[A=\bigcup_{i\in K}X_i\]
and $K\subseteq I$ is infinite and countable, we have $\chi_A\in H_0$ while $\chi_A\notin C_0(X)$. Also, if $\mu<\kappa\leq\lambda$, then $\chi_B\in H_\kappa$ while $\chi_B\notin H_\mu$, where
\[B=\bigcup_{i\in L}X_i\]
and $L\subseteq I$ is of cardinality $\aleph_\kappa$.
\end{proof}

\begin{remark}\label{HFJ}
A \textit{weakly inaccessible cardinal} is defined as an uncountable limit regular cardinal. The existence of weakly inaccessible cardinals cannot be proved within \textsf{ZFC}, though, that such cardinals exist is not known to be inconsistent with \textsf{ZFC}. The existence of weakly inaccessible cardinals is sometimes assumed as an additional axiom. Observe that weakly inaccessible cardinals are necessarily the fixed points of the aleph function, that is, $\aleph_\lambda=\lambda$, if $\lambda$ is a weakly inaccessible cardinal. In Theorem \ref{GGFD} we have inserted a chain of ideals of $C_B(X)$ between $C_0(X)$ and $C_B(X)$; this chain will therefore have its length equal to the Lindel\"{o}f number $\ell(X)$ of $X$ provided that $\ell(X)$ is weakly inaccessible.
\end{remark}

\begin{remark}\label{KJHG}
Theorems \ref{HGFLK}, \ref{CTTF} and \ref{GGFD} make essential use of the fact that the spaces under consideration (locally separable metrizable spaces as well as locally Lindel\"{o}f paracompact spaces) are representable as a disjoint union of open Lindel\"{o}f subspaces. Much of Theorems \ref{HGFLK}, \ref{CTTF} and \ref{GGFD} still remain valid if one replaces the Lindel\"{o}f property by a more or less general topological property $\mathfrak{P}$, provided that the spaces under consideration are representable as a disjoint union of open subspaces each with $\mathfrak{P}$.
\end{remark}

The topological properties considered so far have all been closed hereditary and preserved under finite closed sums. We now consider pseudocompactness. (Recall that a completely regular space $X$ is called \textit{pseudocompact} if every continuous mapping $f:X\rightarrow\mathbb{R}$ is bounded.) Note that while pseudocompactness is preserved under finite closed sums (as this follows readily from its definition) it is not a closed hereditary topological property. Pseudocompactness is however hereditary with respect to regular closed subspaces, that is, every regular closed subspace of a pseudocompact space is pseudocompact. (See Exercise 3.10.F(e) of \cite{E}. Recall that a subspace of a space is called \textit{regular closed} if it is the closure of an open subspace.) What makes consideration of pseudocompactness (and also realcompactness) so important is the known structure of the space $\lambda_{\mathfrak U} X$ associated to the ideal ideal ${\mathfrak U}$ of $X$ generated by open subspace of $X$ with a pseudocompact closure in $X$.

\begin{notation}\label{JJJ}
Let $X$ be a completely regular space. Denote
\[{\mathfrak U}=\langle U:\mbox{$U$ is an open subspace of $X$ with a pseudocompact closure}\rangle.\]
\end{notation}

\begin{lemma}\label{FCG}
Let $X$ be a completely regular space. Then
\[{\mathfrak U}=\{A:\mbox{$A\subseteq U$ where $U$ is an open subspace of $X$ with a pseudocompact closure}\}.\]
\end{lemma}

\begin{proof}
Clearly $A\in{\mathfrak U}$, if $A\subseteq U$ where $U$ is an open subspace of $X$ with a pseudocompact closure in $X$.

Let $A\in{\mathfrak U}$. Then $A\subseteq U_1\cup\cdots\cup U_n$ where $U_i$ is an open subspace of $X$ with a pseudocompact closure in $X$ for each $i=1,\ldots,n$. Let $U=U_1\cup\cdots\cup U_n$. Then $U$ is an open subspace of $X$ whose closure $\mathrm{cl}_XU=\mathrm{cl}_XU_1\cup\cdots\cup\mathrm{cl}_X U_n$ is pseudocompact, as it is the finite union of pseudocompact subspaces of $X$.
\end{proof}

The following is dual to Lemma \ref{KFFF}.

\begin{lemma}\label{HF}
Let $X$ be a completely regular space. Then $X$ is locally null (with respect to the ideal $\mathfrak{U}$) if and only if $X$ is locally pseudocompact.
\end{lemma}

\begin{proof}
The proof is similar to that of Lemma \ref{KFFF}. Observe that (since $X$ is regular) for every $x\in X$ each neighborhood of $x$ in $X$ contains a regular closed neighborhood of $x$ in $X$, that is, a neighborhood in $X$ of the form $\mathrm{cl}_XU$ where $U$ is open in $X$.
\end{proof}

Considering the dualities between Lemmas \ref{KFFF} and \ref{HF}, one can state and prove results dual to Theorems \ref{JGF} and \ref{HJG}. We will now proceed with the determination of $\lambda_{\mathfrak U}X$. We need some preliminaries first.

A completely regular space $X$ is said to be \textit{realcompact} if it is homeomorphic to a closed subspaces of some product $\mathbb{R}^\mathfrak{m}$ of the real line. Realcompactness is a closed hereditary topological property. Every regular Lindel\"{o}f space (in particular, every compact space) is realcompact. Also, every realcompact pseudocompact space is compact. To every completely regular space $X$ there corresponds a realcompact space $\upsilon X$ (called the \textit{Hewitt realcompactification of $X$}) which contains $X$ as a dense subspace and is such that every continuous mapping $f:X\rightarrow\mathbb{R}$ is continuously extendible over $\upsilon X$; we may assume that $\upsilon X\subseteq\beta X$. Note that a completely regular space $X$ is realcompact if and only if $X=\upsilon X$. (See Section 3.11 of \cite{E} for further information.)

The following lemma, which may be considered as a dual result to Lemma \ref{HDHD}, is due to A. W. Hager and D. G. Johnson in \cite{HJ}; a direct proof may be found in \cite{C}. (See also Theorem 11.24 of \cite{W}.)

\begin{lemma}[Hager--Johnson \cite{HJ}]\label{A}
Let $X$ be a completely regular space. Let $U$ be an open subspace of $X$ such that $\mathrm{cl}_{\upsilon X} U$ is compact. Then $\mathrm{cl}_X U$ is pseudocompact.
\end{lemma}

Observe that any completely regular space $X$ with a dense pseudocompact subspace $A$ is pseudocompact; as for any continuous mapping $f:X\rightarrow\mathbb{R}$ we have
\[f (X)=f(\mathrm{cl}_XA)\subseteq\mathrm{cl}_\mathbb{R}f(A)\]
and $f(A)$ is bounded (since $A$ is pseudocompact).

\begin{lemma}\label{HGA}
Let $X$ be a completely regular space and let $A$ be regular closed in $X$. Then $\mathrm{cl}_{\beta X} A\subseteq\upsilon X$ if and only if $A$ is pseudocompact.
\end{lemma}

\begin{proof}
One half follows from Lemma \ref{A}, as if $\mathrm{cl}_{\beta X} A\subseteq\upsilon X$ then $\mathrm{cl}_{\upsilon X} A=\mathrm{cl}_{\beta X} A$ is compact, since it is closed in $\beta X$. For the other half, note that if $A$ is pseudocompact then so is $\mathrm{cl}_{\upsilon X} A$, as it contains $A$ as a dense subspace.  But $\mathrm{cl}_{\upsilon X} A$ is realcompact (as it is closed in $\upsilon X$ and $\upsilon X$ is so) and therefore is compact. Thus $\mathrm{cl}_{\beta X} A\subseteq\mathrm{cl}_{\upsilon X} A$.
\end{proof}

\begin{theorem}\label{PTF}
Let $X$ be a completely regular space. Then
\[\lambda_{\mathfrak U}X=\mathrm{int}_{\beta X}\upsilon X.\]
\end{theorem}

\begin{proof}
Suppose that $C\in\mathrm{Coz}(X)$ is such that $\mathrm{cl}_X C$ has a pseudocompact neighborhood $U$ in $X$. Since $U$ is pseudocompact and $\mathrm{cl}_X C$ is regular closed in $X$ (and thus in $U$) then $\mathrm{cl}_X C$ is pseudocompact. Thus $\mathrm{cl}_{\beta X} C\subseteq\upsilon X$, by Lemma \ref{HGA}, and therefore $\mathrm{int}_{\beta X}\mathrm{cl}_{\beta X} C\subseteq\mathrm{int}_{\beta X}\upsilon X$.

To show the reverse inclusion, let $t\in\mathrm{int}_{\beta X}\upsilon X$. By the Urysohn lemma there is a continuous mapping $f:\beta X\rightarrow[0,1]$ with $f(t)=0$ and $f|_{\beta X\setminus\mathrm{int}_{\beta X}\upsilon X}=\mathbf{1}$. Then
\[C=X\cap f^{-1}\big([0,1/2)\big)\in\mathrm{Coz}(X)\]
and $t\in\mathrm{int}_{\beta X}\mathrm{cl}_{\beta X} C$, as $t\in f^{-1}([0,1/2))$ and $f^{-1}([0,1/2))\subseteq\mathrm{int}_{\beta X}\mathrm{cl}_{\beta X} C$ by Lemma \ref{LKG}. Note that if $V=X\cap f^{-1}([0,2/3))$ then $V$ is an open neighborhood of $\mathrm{cl}_X C$ in $X$, and $\mathrm{cl}_X V$ is pseudocompact by Lemma \ref{HGA}, as it is regular closed in $X$ and
\[\mathrm{cl}_{\beta X}V\subseteq f^{-1}\big([0,2/3]\big)\subseteq\upsilon X.\]
Therefore $\mathrm{int}_{\beta X}\mathrm{cl}_{\beta X} C\subseteq\lambda_{\mathfrak U}X$.
\end{proof}

In our final result we will be dealing with realcompactness. Despite the fact that realcompactness is closed hereditary, it is unfortunately not preserved under finite closed sums in the realm of completely regular spaces. (In \cite{M} -- a correction in \cite{M1} -- S. Mr\'{o}wka describes a completely regular space which is not realcompact but it can be represented as the union of two of its closed realcompact subspaces; a simpler example is given by A. Mysior in \cite{My}.) So, our previous results are not applicable as long as the underlying space is required to be only completely regular. As we will see, the situation changes if we confine ourselves to the class of normal spaces.

Recall that a subspace $A$ of a space $X$ is called \textit{$C$-embedded} in $X$ if every continuous mapping $f:A\rightarrow\mathbb{R}$ is continuously extendible over $X$.

\begin{lemma}[Gillman--Jerison \cite{GJ}]\label{DDJD}
Let $X$ be a completely regular space and let $A$ be $C$-embedded in $X$. Then $\mathrm{cl}_{\upsilon X}A=\upsilon A$.
\end{lemma}

By ${\mathfrak Q}$ in the following we simply mean ${\mathfrak I}_\mathfrak{P}$, as defined in Notation \ref{GGH}, with $\mathfrak{P}$ assumed to be realcompactness; the re-definition is for convenience. (This is also consistent with the initial terminology once used to refer to realcompact spaces; realcompact spaces were originally introduced by E. Hewitt in \cite{H}; they were then called \textit{$Q$-spaces}.)

\begin{notation}\label{GYH}
Let $X$ be a space. Denote
\[{\mathfrak Q}=\{A\subseteq X:\mathrm{cl}_XA\mbox{ is realcompact}\}.\]
\end{notation}

Recall that a completely regular space $X$ is realcompact if and only if $X=\upsilon X$. Observe that in a normal space each closed subspace is $C$-embedded. (See Problem 3.D.1 of \cite{GJ}.) This observation will be used in the following.

\begin{lemma}\label{KGPF}
Let $X$ be a normal space. Then $\mathfrak{Q}$ is an ideal in $X$.
\end{lemma}

\begin{proof}
Note that ${\mathfrak Q}$ is non-empty, as it contains $\emptyset$. Let $B\subseteq A$ with $A\in\mathfrak{Q}$. In particular $\mathrm{cl}_XB\subseteq\mathrm{cl}_XA$. Since $\mathrm{cl}_XA$ is realcompact and realcompactness is closed hereditary then $\mathrm{cl}_XB$ is realcompact. That is $B\in\mathfrak{Q}$. Next, let $C,D\in\mathfrak{Q}$. Since $X$ is normal, every closed subspace of $X$ is $C$-embedded in $X$. Thus, using Lemma \ref{DDJD} we have
\begin{eqnarray*}
\upsilon\big(\mathrm{cl}_X(C\cup D)\big)&=&\mathrm{cl}_{\upsilon X}\big(\mathrm{cl}_X(C\cup D)\big)\\&=&\mathrm{cl}_{\upsilon X}(\mathrm{cl}_XC)\cup\mathrm{cl}_{\upsilon X}(\mathrm{cl}_XD)\\&=&\upsilon(\mathrm{cl}_X C)\cup\upsilon(\mathrm{cl}_XD)=\mathrm{cl}_X C\cup\mathrm{cl}_XD=\mathrm{cl}_X(C\cup D).
\end{eqnarray*}
That is $\mathrm{cl}_X(C\cup D)$ is realcompact. Therefore $C\cup D\in{\mathfrak Q}$.
\end{proof}

Once one states and proves a lemma dual to Lemma \ref{KFFF} it will be then possible to state and prove results dual to Theorems \ref{JGF} and \ref{HJG}. Our concluding result determines $\lambda_{\mathfrak Q}X$. As in the case of pseudocompactness, it turns out that $\lambda_{\mathfrak Q}X$ is a familiar subspace of $\beta X$ as well.

\begin{theorem}\label{RRGY}
Let $X$ be a normal space. Then
\[\lambda_{\mathfrak Q}X=\beta X\setminus\mathrm{cl}_{\beta X}(\upsilon X\setminus X).\]
\end{theorem}

\begin{proof}
Suppose that $C\in\mathrm{Coz}(X)$ is such that $\mathrm{cl}_X C$ has a realcompact neighborhood $U$ in $X$. Then $\mathrm{cl}_X C$ is realcompact, as it is closed in $U$. Since $\mathrm{cl}_X C$ is $C$-embedded in $X$, as $X$ is normal, by Lemma \ref{DDJD} we have $\mathrm{cl}_{\upsilon X}C=\upsilon(\mathrm{cl}_X C)=\mathrm{cl}_X C$. But then $\mathrm{int}_{\beta X}\mathrm{cl}_{\beta X} C\cap(\upsilon X\setminus X)=\emptyset$, as
\[\mathrm{cl}_{\beta X} C\cap(\upsilon X\setminus X)=\mathrm{cl}_{\upsilon X} C\cap(\upsilon X\setminus X)=\emptyset\]
and thus $\mathrm{int}_{\beta X}\mathrm{cl}_{\beta X} C\cap\mathrm{cl}_{\beta X}(\upsilon X\setminus X)=\emptyset$, that is
\[\mathrm{int}_{\beta X}\mathrm{cl}_{\beta X} C\subseteq\beta X\setminus\mathrm{cl}_{\beta X}(\upsilon X\setminus X).\]

To show the reverse inclusion, let $t\in\beta X\setminus\mathrm{cl}_{\beta X}(\upsilon X\setminus X)$. Let $f:\beta X\rightarrow[0,1]$ be a continuous mapping with $f(t)=0$ and $f|_{\mathrm{cl}_{\beta X}(\upsilon X\setminus X)}=\mathbf{1}$. Then
\[C=X\cap f^{-1}\big([0,1/2)\big)\in\mathrm{Coz}(X)\]
and $t\in\mathrm{int}_{\beta X}\mathrm{cl}_{\beta X} C$, as $t\in f^{-1}([0,1/2))$ and $f^{-1}([0,1/2))\subseteq\mathrm{int}_{\beta X}\mathrm{cl}_{\beta X} C$ by Lemma \ref{LKG}. Let $V=X\cap f^{-1}([0,2/3))$. Then $V$ is a neighborhood of $\mathrm{cl}_X C$ in $X$. Since $\mathrm{cl}_{\beta X}V\cap(\upsilon X\setminus X)=\emptyset$, as $\mathrm{cl}_{\beta X}V\subseteq f^{-1}([0,2/3])$, we have
\[\mathrm{cl}_X V=X\cap\mathrm{cl}_{\beta X}V=\upsilon X\cap\mathrm{cl}_{\beta X}V=\mathrm{cl}_{\upsilon X}V.\]
Therefore $\mathrm{cl}_X V$ is realcompact, as it is closed in $\upsilon X$. Thus $\mathrm{int}_{\beta X}\mathrm{cl}_{\beta X} C\subseteq\lambda_{\mathfrak Q}X$.
\end{proof}



\end{document}